\documentclass[preprint]{imsart}
\RequirePackage{fullpage}

\RequirePackage{amsthm,amsmath,amsfonts,amssymb,xcolor,xfrac,makecell,multirow,graphicx}
\RequirePackage[numbers,sort&compress]{natbib}
\RequirePackage[colorlinks,citecolor=blue,urlcolor=blue]{hyperref}
\RequirePackage{graphicx}
\usepackage{tikz}
\usetikzlibrary{hobby, arrows, backgrounds}

\startlocaldefs
\theoremstyle{plain}

\newtheorem{theorem}{Theorem}[section]
\newtheorem{lemma}[theorem]{Lemma}
\newtheorem{proposition}[theorem]{Proposition}
\theoremstyle{definition}
\newtheorem{definition}[theorem]{Definition}
\newtheorem{example}[theorem]{Example}

\newtheorem{remark}[theorem]{Remark}
\newtheorem{assumption}[theorem]{Assumption}

\definecolor{mygreen}{RGB}{60, 150, 60}\definecolor{myorange}{RGB}{235, 155, 30}
\definecolor{myblue}{RGB}{100, 200, 255}
\definecolor{mybluedark}{RGB}{42, 90, 161}
\definecolor{mypurple}{RGB}{150, 50, 209}

\newcommand{\Rbb}{\mathbb{R}}

\newcommand{\Nbb}{\mathbb{N}}

\newcommand{\ind}{\mathbf{1}}
\newcommand{\Mcal}{\mathcal{M}}

\newcommand{\Pbb}{\mathbb{P}}
\newcommand{\Ebb}{\mathbb{E}}

\newcommand{\MM}{\mathcal{M}}

\DeclareMathOperator{\Var}{\mathbb{V}}
\DeclareMathOperator{\Sen}{\mathbb{S}}

\DeclareMathOperator{\Cos}{\mathbb{S}}

\newcommand{\Fcal}{\mathcal{F}}

\newcommand{\Pcal}{\mathcal{P}}

\newcommand{\Ncal}{\mathcal{N}}
\newcommand{\diff}{\,\textnormal{d}}

\DeclareMathOperator{\Div}{div}

\DeclareMathOperator{\id}{id}
\newcommand{\idmat}{\mathrm{I}}

\DeclareMathOperator{\Tan}{Tan}

\newcommand{\transport}{\mathbf{t}}

\DeclareMathOperator{\MLE}{MLE}
\DeclareMathOperator{\BLE}{BLE}
\DeclareMathOperator{\OSE}{OSE}
\DeclareMathOperator{\WPE}{WPE}
\DeclareMathOperator{\OLS}{OLS}

\DeclareMathOperator*{\argmax}{arg\,max}

\newcommand{\bW}{\boldsymbol{W}}

\endlocaldefs

\begin{document}

\begin{frontmatter}
\runtitle{Wasserstein-Cram\'er-Rao Theory of Unbiased Estimation}
\title{Wasserstein-Cram\'er-Rao Theory\\ of Unbiased Estimation}

\begin{aug}
\author[A]{\fnms{Nicol\'as}~\snm{{Garc\'ia Trillos}}\ead[label=e1]{garciatrillo@wisc.edu}},
\author[B]{\fnms{Adam Quinn}~\snm{Jaffe}\ead[label=e2]{a.q.jaffe@columbia.edu}}
\and
\author[B]{\fnms{Bodhisattva}~\snm{Sen}\ead[label=e3]{b.sen@columbia.edu}}
\address[A]{Department of Statistics, University of Wisconsin Madison, WI\printead[presep={ ,\ }]{e1}}

\address[B]{Department of Statistics, Columbia University, New York, NY\printead[presep={ ,\ }]{e2,e3}}
\end{aug}

\begin{abstract}
    The quantity of interest in the classical Cram\'er-Rao theory of unbiased estimation (e.g., the Cram\'er-Rao lower bound, its exact attainment for exponential families, and asymptotic efficiency of maximum likelihood estimation)  is the variance, which represents the instability of an estimator when its value is compared to the value for an independently-sampled data set from the same distribution.
    In this paper we are interested in a quantity which represents the instability of an estimator when its value is compared to the value for an infinitesimal additive perturbation of the original data set; we refer to this as the ``sensitivity'' of an estimator.
    The resulting theory of sensitivity is based on the Wasserstein geometry in the same way that the classical theory of variance is based on the Fisher-Rao (equivalently, Hellinger) geometry, and this insight allows us to determine a collection of results which are analogous to the classical case:
    a Wasserstein-Cram\'er-Rao lower bound for the sensitivity of any unbiased estimator, a characterization of models in which there exist unbiased estimators achieving the lower bound exactly, and various results on the asymptotic achievement of the lower bound.
    More precisely, we show that Wasserstein projection estimators (and one-step approximations thereof) are asymptotically sensitivity-efficient under suitable assumptions, and that there exist Hodges-type estimators which exhibit super-sensitivity-efficiency.
    We use these results to treat many statistical examples, sometimes revealing new optimality properties for existing estimators and other times revealing entirely new estimators.
\end{abstract}



\end{frontmatter}

\setcounter{tocdepth}{1}
\tableofcontents

\section{Introduction}\label{sec:intro}

    The setting of this paper involves a general parametric statistical model $\{P_{\theta}:\theta\in\Theta\}$ of probability measures on $\Rbb^d$, where we observe independent identically-distributed (i.i.d.) samples $X_1,\ldots, X_n$ from some $P_{\theta}$ for unknown $\theta\in\Theta\subseteq\Rbb^p$, and where the goal is to estimate $\chi(\theta)$ for some fixed transformation $\chi:\Theta\to\Rbb^k$; for the sake of notational simplicity in the Introduction we will assume $p=k=1$, but the body of the paper addresses the general case.
    More specifically, we are interested in unbiased estimators $T_n:(\Rbb^d)^n\to\Rbb$ of $\chi(\theta)$, meaning $\Ebb_{\theta}[T_n(X_1,\ldots, X_n)] =\chi(\theta)$ for all $\theta\in\Theta$, and various notions of optimality thereof.

    The perspective of classical statistics is to search for unbiased estimators $T_n$ for which the variance, denoted $\Var_{\theta}(T_n)$, is small.
    One motivation for this perspective is that the variance is a measure of the instability that the estimator would experience if the data were independently \textit{re-sampled}.
    That is, if $X_1',\ldots, X_n'$ denote i.i.d. samples from $P_{\theta}$ which are independent of $X_1,\ldots, X_n$, then we have
    \begin{equation*}
        \Var_{\theta}(T_n) = \frac{1}{2}\Ebb_{\theta}\left[|T_n(X_1',\ldots, X_n')-T_n(X_1,\ldots, X_n)|^2\right].
    \end{equation*}
    The classical Cram\'er-Rao theory provides a fundamental lower bound for the variance and provides various methods for achieving this lower bound, either exactly or asymptotically.

    Our perspective in this paper is that it may additionally be desirable to search for unbiased estimators which achieve a small value of some measure of the instability that the estimator would experience if the data were \textit{perturbed}.
    That is, if $X_1',\ldots, X_n'$ are defined via $X_i' := X_i+ \xi_i$ for $1\le i \le n$ where $\xi_1,\ldots \xi_n$ are i.i.d. samples from a Gaussian distribution with mean zero and variance $\varepsilon^2>0$ and which are independent of $X_1,\ldots, X_n$, then we define the \textit{$\varepsilon$-sensitivity} of the estimator $T_n$ as
    \begin{equation*}
        \Sen_{\theta,\varepsilon}(T_n) := \Ebb_{\theta}\left[\left|\frac{T_n(X_1',\ldots,X_n') - T_n(X_1,\ldots, X_n)}{\varepsilon}\right|^2\right].
    \end{equation*}
    A stylized but concrete setting where it may be desirable to achieve a small value of the $\varepsilon$-sensitivity is any statistical application where the data are collected with measurement error; in such settings, the $\varepsilon$-sensitivity describes the size of the difference between an estimator computed on the observed variables and the same estimator computed on the latent variables.

    Mathematically, the theory is slightly cleaner if we consider \textit{infinitesimal} versions of the perturbations above.
    That is, we consider the limit of $\varepsilon$-sensitivity as $\varepsilon \rightarrow 0$ and introduce what we call the \textit{sensitivity}
    \begin{equation*}
       \Sen_{\theta}(T_n) := \Ebb_{\theta}\left[\sum_{i=1}^{n}\left\|\nabla_{x_i} T_n(X_1,\ldots,X_n)\right\|^2\right],
    \end{equation*}
    which will be our focus in this paper; here, $\nabla_{x_i} T_n(X_1,\ldots,X_n)$ denotes the gradient of the function $T_n:(\Rbb^d)^n\to\Rbb$ with respect to the argument $x_i\in\Rbb^d$ for each $1\le i\le n$.
    The sensitivity coincides (up to a constant factor, depending on the convention) with the notion of the \textit{Dirichlet energy} studied in probability theory, partial differential equations, and potential theory (e.g., \cite[Section~2.2.5]{EvansPDE}, \cite[Chapter~8]{Brezis}, \cite[Chapter~XI]{Kellogg}, \cite{DellacherieMeyer}).
    We emphasize that, in most prior uses in mathematical statistics, the Dirichlet energy is interesting only insofar as it is useful for controlling the variance (say, via a Poincar\'e inequality or a log-Sobolev inequality); in contrast, our perspective in this work is that the Dirichlet energy itself is a quantity which may be of interest to control.

    \subsection{Illustrative Examples}

    We consider some examples in order to illustrate the problems involved in constructing unbiased estimators which achieve a small sensitivity.

    \begin{example}[Gaussian location family]
        Suppose that $P_{\theta}$ is the Gaussian distribution with mean $\theta$ and variance 1, and that the goal is to estimate $\theta\in\Rbb$.
        It is natural to consider the maximum likelihood estimator (MLE) given by the sample mean $T_n^{\MLE} = n^{-1}\sum_{i=1}^{n}X_i$, for which it is easy to calculate $$\Var_{\theta}(T_n^{\MLE}) = \Sen_{\theta}(T_n^{\MLE}) = \frac{1}{n}.$$
        It follows from the classical Cram\'er-Rao theory that the MLE achieves the optimal variance, so one can immediately show (say, via the Gaussian Poincar\'e inequality which states $\Var_{\theta}(T_n) \le \Sen_{\theta}(T_n)$ for any unbiased estimator $T_n$) that the MLE also achieves the optimal sensitivity.
        Thus, the sample mean is optimal for both the variance and the sensitivity among unbiased estimators of $\theta$.
        The sample mean is, in fact, the prototypical example of an estimator with small sensitivity.
        A direct calculation shows that we have $\nabla_{x_i}T_n^{\MLE} = n^{-1}$ deterministically for each $1\le i\le n$; in terms of the $\varepsilon$-sensitivity pre-limit, $T_n^{\MLE}$ introduces many cancellations which reduce the effect of additive noise.
    \end{example}

        \begin{example}[uniform scale family]\label{ex:Unif-scale}
        Suppose that $P_{\theta}$ is the uniform distribution on $[0,\theta]$, for $\theta >0$, and let us consider two natural estimators of $\theta$.
        One is the MLE given by $T_n^{\MLE}(X_1,\ldots, X_n)=\max\{X_1,\ldots, X_n\}$, and the other is the best linear estimator (BLE) given by twice the sample mean $T_n^{\BLE}(X_1,\ldots, X_n)=2n^{-1}\sum_{i=1}^{n}X_i$.
        It is easy to show that both estimators are (asymptotically) unbiased, and that their variances are
        \begin{equation*}
            \Var_{\theta}(T_n^{\MLE}) = \frac{n\theta^2}{(n+1)^2(n+2)}\sim\frac{\theta^2}{n^2}\qquad\textnormal{ and }\qquad\Var_{\theta}(T_n^{\BLE}) = \frac{\theta^2}{3n}.
        \end{equation*}
        Observe that the MLE achieves the fast rate $n^{-2}$ (recall that the classical Cram\'er-Rao lower bound does not apply to this model) while the BLE achieves the usual rate $n^{-1}$.
        However, we can directly compute that their sensitivities are
        \begin{equation*}
            \Sen_{\theta}(T_n^{\MLE}) =1 \qquad\textnormal{ and }\qquad\Sen_{\theta}(T_n^{\BLE}) = \frac{4}{n},
        \end{equation*}
        where the first calculation follows because the MLE depends on exactly one (random) sample, i.e., $\nabla_{x_i}T_n = \ind\{X_i=X_{(n)}\}$; in the $\varepsilon$-sensitivity pre-limit the MLE is of constant order because it does not  introduce any cancellations.
        See Figure~\ref{fig:uniform} in Appendix~\ref{app:examples} for a Monte Carlo simulation of the bias, variance, and $\varepsilon$-sensitivity (for $\varepsilon=10^{-4}$)  while $n$ grows ($n=10^2,\ldots, 10^4$); they illustrate that, although our notion of sensitivity concerns infinitesimal perturbations, it captures the behavior of $\varepsilon$-sensitivity for small to moderate values of $\varepsilon>0$, which may be more directly relevant in practice.
        We also note that the simulation gives a preview of our later methodology, which concerns optimality properties for the Wasserstein projection estimator (WPE). 
    \end{example}

    \begin{example}[Laplace location family]\label{ex:laplace-intro}
        Suppose that $P_{\theta}$ is the Laplace distribution with mean $\theta$ and variance 2, i.e., its density with respect to the Lebesgue measure is $p_{\theta}(\cdot) = \frac{1}{2}\exp(-|\cdot- \theta|)$.
        Let us consider two estimators for $\theta$ in this problem.
        The first is the MLE denoted $T_n^{\MLE}$, which is simply the sample median of $X_1,\ldots, X_n$.
        The second is the best linear estimator (BLE) which is the sample mean, $T_n^{\BLE}(X_1,\ldots,X_n)=n^{-1}\sum_{i=1}^{n}X_i$.
        It is easy to show that both estimators are (asymptotically) unbiased.
        Also, their variances can be computed as
        \begin{equation*}
            \Var_{\theta}(T_n^{\MLE}) \sim \frac{1}{n}\qquad\textnormal{ and }\qquad\Var_{\theta}(T_n^{\BLE}) = \frac{2}{n}
        \end{equation*}
        meaning both achieve the same rate of decay of $n^{-1}$ but the MLE has a smaller constant prefactor, as it must from the classical Cram\'er-Rao theory.        
        However, we can also compute
        \begin{equation*}
            \Sen_{\theta}(T_n^{\MLE}) =\begin{cases}
                \sfrac{1}{2} &\textnormal{if $n$ even} \\ 
                1 &\textnormal{if $n$ odd} \\ 
            \end{cases}\qquad\textnormal{ and }\qquad\Sen_{\theta}(T_n^{\BLE}) = \frac{2}{n},
        \end{equation*}
        where the first calculation follows because the MLE depends on exactly 1 or 2 observations in the sample, depending on the parity of $n$.
        In other words, the sensitivity of the MLE is of constant order, and the sensitivity of the BLE achieves the classical rate.
    \end{example}

    The examples above show that it is not immediately clear how to construct estimators that achieve a small value of sensitivity.
    In fact, we have seen that some estimators with small variance have the drawback that their sensitivity is non-vanishing.    
    The goal of this paper is to provide a general statistical theory of unbiased estimators that achieve the smallest possible sensitivity, and to provide detailed analyses in examples like the above and many more.
    
    \subsection{Summary of Results}

    This subsection summarizes our main results and also serves as an outline of the main body of the paper.
    Broadly speaking, our work addresses the following three questions:
    \begin{itemize}
        \item[(i)] How small can the sensitivity of an unbiased estimator be?
        \item[(ii)] When is it possible to achieve this lower bound exactly?
        \item[(iii)] Is there a general way to asymptotically achieve this lower bound?
    \end{itemize}
    Our main results yield essentially complete answers to these questions.

    Before stating the results, let us emphasize that our theory is, broadly speaking, analogous to the classical Cram\'er-Rao theory of variance.
    To aid the reader's understanding, we give Table~\ref{tab:statistics} which highlights the analogies between the statistical concepts in both theories.
    As we will explore in detail throughout the paper, both theories are derived from the same geometric principle, but by considering different geometries on the space of probability measures; the classical theory of variance comes from the Fisher-Rao (equivalently, Hellinger) geometry, and our theory of sensitivity comes from the Wasserstein geometry.
    This analogy has been somewhat explored in the mathematical field of information geometry \cite{AmariMatsuda, WIM, WCR_achieve} but the statistical implications of the theory have neither been developed nor appreciated so far.

    \begin{table}[t]
        \centering
        \resizebox{0.775\width}{!}{
        \begin{tabular}{|c|c|c|c|}
        \hline
        \thead{concept} & \thead{Hellinger geometry} & \thead{Wasserstein geometry} \\
        \hline
        \makecell{measure of instability} & \makecell{variance\\\,\\$\begin{aligned}\Var_{\theta}(T_n)=\Ebb_{\theta}\left[(T_n-\chi(\theta))(T_n-\chi(\theta))^{\top}\right]\end{aligned}$} & \makecell{\,\\sensitivity (Section~\ref{subsec:sensitivity})\\\,\\$\begin{aligned}\Sen_{\theta}(T_n)=\Ebb_{\theta}\left[\sum_{i=1}^{n}D_{x_i}T_n(D_{x_i}T_n)^{\top}\right]\end{aligned}$\\\,}  \\
        \hline
        \makecell{regularity condition\\ on statistical model\\(Remark~\ref{rem:DWS-v-DQM}, Remark~\ref{rem:DWS-v-DQM-exs})} & \makecell{\,\\differentiability in quadratic mean (DQM)\\\,\\$\begin{aligned}\int\left|\sqrt{(\diff P_{\theta+th}/\diff P_{\theta})(x)}-1-t(G_{\theta}(x))^{\top}h\right|^2\diff P_{\theta}(x) = o(t^2)\end{aligned}$\\\,} & \makecell{\,\\differentiability in Wasserstein sense (DWS) (Definition~\ref{def:DWS})\\\,\\$\begin{aligned}\int\|\transport_{P_{\theta}\to P_{\theta+th}}(x)-x-t\Phi_{\theta}(x)h\|^2\diff P_{\theta}(x) = o(t^2)\end{aligned}$\\\,} \\
        \hline
        \makecell{linearization of\\statistical model} & \makecell{\,\\log-likelihood linearization ($\sfrac{1}{2}$ the score function)\\\,\\$\begin{aligned}G_{\theta}:\Rbb^d\to\Rbb^p\end{aligned}$\\\,} & \makecell{\,\\transport linearization (Definition~\ref{def:DWS})\\\,\\$\begin{aligned}\Phi_{\theta}:\Rbb^d\to\Rbb^{d\times p}\end{aligned}$\\\,}\\
        \hline
        \makecell{information matrix\\ (Remark~\ref{rem:inf-matrix-comparison})} & \makecell{\,\\Fisher information matrix\\\,\\$\begin{aligned}I(\theta) = 4\Ebb_{\theta}\left[G_{\theta}(X)(G_{\theta}(X))^{\top}\right]\end{aligned}$\\\,} & \makecell{\,\\Wasserstein information matrix (Definition~\ref{def:WIM})\\\,\\$\begin{aligned}J(\theta) = \Ebb_{\theta}\left[(\Phi_{\theta}(X))^{\top}\Phi_{\theta}(X)\right]\end{aligned}$\\\,} \\

        \hline
        \makecell{instability lower bound\\for unbiased estimators\\(Section~\ref{sec:bound})} & \makecell{\,\\Cram\'er-Rao lower bound\\\,\\$\begin{aligned}\Var_{\theta}(T_n) \succeq \frac{1}{n}D\chi(\theta)(I(\theta))^{-1}(D\chi(\theta))^{\top}\end{aligned}$\\\,} & \makecell{\,\\Wasserstein-Cram\'er-Rao lower bound (Theorem~\ref{thm:WCR})\\\,\\$\begin{aligned}\Sen_{\theta}(T_n) \succeq \frac{1}{n}D\chi(\theta)(J(\theta))^{-1}(D\chi(\theta))^{\top}\end{aligned}$\\\,} \\
        \hline
        \makecell{statistical models admitting\\exactly efficient estimators\\(Section~\ref{sec:exact}, Remark~\ref{rem:exp-v-transport})} & \makecell{\,\\exponential family\\\,\\$\begin{aligned}G_{\theta}(x) = \frac{1}{2}(D\eta(\theta))^{\top}(\phi(x)-\chi(\theta))\end{aligned}$\\\,\\$\eta:\Theta\to\Rbb^k$\\ $\phi:\Rbb^d\to\Rbb^k$\\$\Ebb_{\theta}[\phi(X)] = \chi(\theta)$\\\,} & \makecell{\,\\transport family (Definition~\ref{def:W-fam}, Theorem~\ref{thm:efficiency-W-fam})\\\,\\$\begin{aligned}\Phi_{\theta}(x) = (D\phi(x))^{\top}(\Lambda(\theta))^{-1}D\chi(\theta)\end{aligned}$\\\,\\$\chi:\Theta\to\Rbb^k$\\ $\phi:\Rbb^d\to\Rbb^k$\\$\Lambda(\theta):=\Ebb_{\theta}\left[D\phi(X)(D\phi(X))^{\top}\right]$\\\,} \\
        \hline
        \makecell{general asymptotically\\efficient estimator\\(Section~\ref{sec:asymp})} & \makecell{\,\\maximum likelihood estimator (MLE)\\\,\\$\begin{aligned}\underset{\theta\in\Theta}{\arg\max}\frac{1}{n}\sum_{i=1}^{n}\log p_{\theta}(X_i)\end{aligned}$\\\,} & \makecell{\,\\Wasserstein projection estimator (WPE)\\(Definition~\ref{def:WPE}, Theorem~\ref{thm:WPE-efficient}, Theorem~\ref{thm:WPE-efficient-gen})\\\,\\$\begin{aligned}\underset{\theta\in\Theta}{\arg\min}\,W_2^2\left(P_{\theta},\frac{1}{n}\sum_{i=1}^{n}\delta_{X_i}\right)\end{aligned}$\\\,} \\
        \hline
        \end{tabular}}
        \caption{Summary and comparison of statistical concepts in the Fisher-Rao and Wasserstein geometries.
        The statistical concepts for the Fisher-Rao geometry are classical and can be found in standard texts (e.g., \cite{Keener, vanDerVaart}). The statistical concepts for the Wasserstein geometry constitute the main results of the paper, and references for precise definitions, theorems, and discussions are given above.}
        \label{tab:statistics}
    \end{table}

    Our first set of results concerns question (i) on fundamental lower bounds for the sensitivity.
    Towards this end, we introduce (Definition~\ref{def:DWS}), the notion of \textit{differentiability in the Wasserstein sense (DWS)}, which states that the optimal transport maps between elements of a model $\Pcal := \{P_{\theta}: \theta\in\Theta\}$ vary smoothly in the parameter $\theta\in\Theta  \subseteq \mathbb{R}$. 
    Roughly speaking, a model is DWS when the \textit{transport linearization} function defined via
    \begin{equation}\label{eqn:W-score-pointwise}
        \Phi_{\theta}(x):= \lim_{t\to 0}\frac{\transport_{P_{\theta}\to P_{\theta + t}}(x)-x}{t}
    \end{equation}
    exists, where $\transport_{P_{\theta_0}\to P_{\theta_1}}$ denotes the optimal transport map from $P_{\theta_0}$ to $P_{\theta_1}$.
    Intuitively speaking, the DWS condition is the analog of the differentiability in quadratic mean (DQM) condition from classical theory, and the optimal transport map (minus identity) is the analog of the square root of the ratio of densities (minus one); see Table~\ref{tab:statistics}.
    In a DWS model, one can then define the \textit{Wasserstein information} (Definition~\ref{def:WIM}) via
    \begin{equation*}
        J(\theta) := \Ebb_{\theta}\left[\|\Phi_{\theta}(X)\|^2\right]
    \end{equation*}
    for $\theta\in\Theta$, where $X$ has distribution $P_{\theta}$.
    Thus, the Wasserstein information $J(\theta)$ is the analog of the usual Fisher information $I(\theta)$ in the classical theory.
    Finally, if $T_n$ is any sufficiently regular unbiased estimator of $\chi(\theta)$, where $\chi:\Theta\to\Rbb$ is some smooth function, then we prove (Theorem~\ref{thm:WCR}) a \textit{Wasserstein-Cram\'er-Rao lower bound} on the sensitivity which states that, under some mild technical conditions, we have
    \begin{equation*}
        \Sen_{\theta}(T_n) \ge \frac{(\chi'(\theta))^2}{n J(\theta)}
    \end{equation*}
    for all $n\in\Nbb$, which is of course the analog of the classical Cram\'er-Rao lower bound for the variance.
    
    Next we discuss question (ii) on models in which there exists an unbiased estimator exactly achieving the Wasserstein-Cram\'er-Rao lower bound.
    That is, let us say that an unbiased estimator $T_n$ of $\chi(\theta)$ is \textit{(exactly) sensitivity-efficient} if its sensitivity equals $n^{-1}(\chi'(\theta))^2(J(\theta))^{-1}$ for all $\theta\in\Theta$.
    At the same time, we define a \textit{transport family} (Definition~\ref{def:W-fam}) to be a model $\Pcal$ in which the transport linearization factors into 
    \begin{equation*}
        \Phi_{\theta}(x) = \frac{\chi'(\theta)}{\Lambda(\theta)} \nabla\phi(x),
    \end{equation*}
    where $\chi:\Rbb\to\Rbb$ is some smooth map called the \textit{parameterization} (which later will turn out to be a natural estimand), $\phi:\Rbb^d\to\Rbb$ is some function, and $\Lambda(\theta):=\Ebb_{\theta}[\|\nabla\phi(X)\|^2]$.
    Then, we show (Theorem~\ref{thm:efficiency-W-fam}) that the existence of an unbiased sensitivity-efficient estimator of $\chi(\theta)$ in the model $\Pcal$ is equivalent to $\Pcal$ being a transport family with parameterization $\chi$.
    Thus, transport families are analogous to exponential families from the classical theory, which are precisely the models in which there exists an unbiased estimator which exactly achieves the Cram\'er-Rao lower bound for the variance.
    As in the classical case, transport families include many statistical models of interest, and this allows us to treat several concrete statistical examples. (See Appendix~\ref{app:examples}, and the discussions in Section~\ref{sec:exact}.)

    Lastly, we consider question (iii) on constructing estimators which asymptotically achieve the Wasserstein-Cram\'er-Rao lower bound. Towards this end we introduce the notion of \textit{asymptotic sensitivity-efficiency} of an estimator $T_n$ (Definition~\ref{def:asymp-sens-eff}) which requires that $n\sum_{i=1}^{n}\|\nabla_{x_i}T_n(X_1,\ldots, X_n)\|^2$, whose expectation equals $n$ times the sensitivity of $T_n$, converges in probability to $(\chi'(\theta))^2(J(\theta))^{-1}$ when $X_1,\ldots, X_n$ are i.i.d.~samples from $P_{\theta}$.
    Our main object of interest is the \textit{Wasserstein projection estimator (WPE)}, defined via
    \begin{equation*}
        T_n^{\WPE}(X_1,\ldots, X_n) := \underset{\theta\in\Theta}{\arg\min}\,W_2(P_{\theta},\bar P_n),
    \end{equation*}
    where $W_2$ denotes the Wasserstein metric, and $\bar P_n:=\frac{1}{n}\sum_{i=1}^{n}\delta_{X_i}$ denotes the empirical measure of $X_1,\ldots, X_n$.
    Under suitable regularity conditions, we show (Theorem~\ref{thm:WPE-efficient} and Theorem~\ref{thm:WPE-efficient-gen}) that the WPE is asymptotically sensitivity-efficient.
    Since the usual maximum likelihood estimator (MLE) can be written as a sort of projection of the empirical measure onto the model $\Pcal$ with respect to the Kullback-Leibler (KL) divergence, we see that the WPE is analogous to the MLE from the classical theory.
    This result can be proven directly in the case $d=1$ because of the explicit representation of the Wasserstein distance via quantile functions, but the case $d\ge 2$ requires many technical assumptions that may be difficult to verify; see Subsection~\ref{subsec:WPE} for further discussion.
    We also consider some further questions of interest in the asymptotic theory.
    For instance, we show (Theorem~\ref{thm:OSE}) that approximately solving the WPE via one step of Newton's method yields an asymptotically sensitivity-efficient estimator, provided that the pilot estimator satisfies some mild regularity.
    Additionally, we show (Example~\ref{ex:Hodges}) that Hodges-type constructions can lead to super-sensitivity-efficient estimators, but also (Theorem~\ref{thm:loc-uniform-WCR}) that there are mild conditions on an estimator which imply that the Wasserstein-Cram\'er-Rao lower bound holds uniformly over local alternatives.

    Extensions to a multidimensional parameter $\Theta\subseteq\Rbb^p$ and multidimensional estimand $\chi:\Theta\to\Rbb^k$ are straightforward, and our results cover the general case of $p\ge1$ and $k\ge 1$, although the notation is a bit more cumbersome.
    In the general setting, a statistical model admits a notion of \textit{Wasserstein information matrix} which is analogous to the Fisher information matrix, and an estimator admits a notion of \textit{cosensitivity matrix} which is analogous to the covariance matrix.
    We direct the reader to the main body of the paper (Section~\ref{subsec:sensitivity} and Section~\ref{sec:bound}) for full details.

    In addition to developing our general theory of sensitivity throughout the paper, we give detailed analyses of various statistical problems of interest.
    In many cases, this reveals new optimality properties for well-known estimators, e.g., the sample mean achieves the optimal sensitivity in any DWS location family (Example~\ref{ex:loc-fam}), the sample second moment achieves the optimal sensitivity in any DWS scale family (Example~\ref{ex:var-fam}), the ordinary least squares (OLS) estimator achieves the optimal sensitivity in any DWS linear regression model with centered errors (Example~\ref{ex:regression}), and the sample variance achieves the asymptotically optimal sensitivity in a Gaussian model with unknown mean where the estimand is the variance (Example~\ref{ex:var-unknown-mean}).
    In other cases, this theory reveals estimators with optimal sensitivity that appear to be novel, e.g., an estimator with asymptotically optimal sensitivity in the uniform scale family (Example~\ref{ex:Unif-scale}) is given as follows 
    \begin{equation*}
        T_n(X_1,\ldots, X_n) = \frac{3}{2n^2}\sum_{i=1}^{n}(2i-1)X_{(i)},
    \end{equation*}
    where $X_{(1)},\ldots, X_{(n)}$ denote the order statistics of $X_1,\ldots, X_n$ (Example~\ref{ex:scale-fam}). Note that this estimator falls within the class of $L$-statistics, described in \cite[Chapter~22]{vanDerVaart}.
        
    \subsection{Related Literature}
    
    All of our results on fundamental limits for sensitivity aim to be analogous to classical results on fundamental limits for variance; as such, we do not attempt to summarize this literature, but we assume the reader has familiarity with standard sources, e.g., \cite{vanDerVaart, Keener, EfronExpFam}.
    For an introduction to the mathematical discipline of information geometry, which provides the geometric interpretation of the Cram\'er-Rao bound in terms of the Fisher-Rao geometry, see \cite{Amari}.

    The Wasserstein-Cram\'er-Rao bound and some related considerations have been previously studied in the mathematical field of information geometry.
    In particular, the works \cite{WIM,AmariMatsuda, WCR_achieve} introduced the notion of the Wasserstein information matrix, proved the Wasserstein-Cram\'er-Rao bound in some settings, and provided many calculations of Wasserstein information matrices for some concrete models of interest; our notion of sensitivity is called \textit{Wasserstein variance} in these works, and some connections are drawn to parallel work in non-equilibrium statistical mechanics \cite{ItoNonEquilibriumI,ItoNonEquilibriumII}.
    Some other works connecting information geometry to Wasserstein geometry include \cite{OT_IG, Ito, Ay}.

    We emphasize that, while some basic aspects of this Wasserstein-Cram\'er-Rao theory have indeed been studied in \cite{WIM,AmariMatsuda, WCR_achieve}, there are several senses in which our results represent a departure from this prior literature.
    First of all, our notion of DWS and our careful analytic results (e.g., Proposition~\ref{prop:ContinuityEqnLocallyLipschitz}) allow one to rigorously discuss the theory of sensitivity for estimators which need not be smooth and compactly supported; indeed, in virtually no statistical applications is an estimator regular enough to directly act as a test function in the continuity equation, although this has been assumed in all previous works.
    Second, our results on transport families allow one to directly determine estimators with optimal sensitivity in a wide variety of statistical settings, which is essential for a meaningful comparison of the classical Cram\'er-Rao theory with the Wasserstein-Cram\'er-Rao theory.
    Last, our results on Wasserstein projection estimators resolve the open questions of \cite[Section~5]{WIM} which ask how to construct (in our terminology) asymptotically sensitivity-efficient estimators in a general-purpose way.

    The notion of Wasserstein projection estimator (WPE) has been previously studied in several contexts.
    For example, the WPE is a particular example of a minimum-distance estimator (a general class of estimators about which much is known \cite{BasuMinDist}), leading to the works \cite{Bernton, Bassetti, Belili}, which establish existence and measurability of solutions, consistency, a central limit theorem, and more.
    Another example is that, in machine learning, the WPE has been studied and applied in various problems involving interpretability and fairness \cite{QuantileWP, PerturbWP, curseWP}.
    On the computational side, the work \cite{Bernton} discusses some difficulties for implementing the WPE, which can be somewhat delicate, and the work \cite{Yatracos} shows that WPE can have large asymptotic variance.
    Despite this existing literature, we are not aware of any literature computing the sensitivity of WPEs nor connecting WPE to  the Wasserstein-Cram\'er-Rao bound.

    While our notion of sensitivity shares some broad similarities with the field of robust statistics, we emphasize that these theories are different in a few important ways; this is highlighted by the example of the sample median, which has good robustness (e.g., its breakdown point is $\sfrac{1}{2}$) but bad sensitivity (e.g., its sensitivity is non-vanishing).
    In more detail, we emphasize that the objects of interest in classical robust statistics (e.g., influence functions, breakdown point \cite{HuberRonchetti}) typically come from ``contamination'' models in which some \textit{proportion} of the samples (which may be large, small, or infinitesimal) may be \textit{arbitrarily} altered; in contrast, our notion of sensitivity concerns models in which \textit{all} of the samples are altered, but each one \textit{infinitesimally}.
    In another direction, the conclusions of our results are plainly different from the conclusions of standard results in robust statistics, which typically study optimality through the lens of asymptotic variance while constraining a specified class of estimators to have sufficient robustness (e.g., \cite[Section~9.8]{Keener}); in contrast, our results understand optimality by directly minimizing the sensitivity (either exactly or asymptotically), and among a class of essentially arbitrary estimators.

    We also include a brief discussion of other fields to which our notion of sensitivity is closely related:
    \begin{itemize}
        \item \textit{Local Differential Privacy (LDP).} In settings where no centralized statistician can be trusted to ``curate'' private data, the framework of LDP suggests carefully randomizing each individual's data before processing; the many recent developments on this methodology \cite{LearnPrivately,DuchiJordanWainwrightI,DuchiJordanWainwrightII} can be seen as extensions of the classical strategy of randomized response \cite{Warner} in survey design.
        Our notion of sensitivity exactly describes the fundamental limit for how closely an estimator, when computed on randomized data, approximates the estimator computed on the unobserved private data.
        Strictly speaking our theory applies only to mechanisms in LDP that utilize carefully-calibrated additive noise, while there are other mechanisms of interest that are more complex.
        \item \textit{Distributionally Robust Optimization (DRO).}
        A recent paradigm in stochastic optimization \cite{MohajerinEsfahaniKuhn, BlanchetMurthy, GaoKleywegt} aims to minimize some statistical objective, in the worst-case regime where the empirical measure of the data may be replaced with any distribution in an ``ambiguity set'' taken to be Wasserstein balls of small radius; in this setting and under some additional assumptions, it can be shown that we have
        \begin{equation*}
            \max_{W_2(P_{\theta},\tilde P)\le \varepsilon}\Ebb_{\tilde P}|T_n(X_1,\ldots, X_n)-\chi(\theta)| = \Ebb_{\theta}|T_n(X_1,\ldots, X_n)-\chi(\theta)| + \varepsilon\,\sqrt{\Sen_{\theta}(T_n)} + o(\varepsilon)
        \end{equation*}
        as $\varepsilon\to 0$ for each $\theta\in\Theta$ (see, e.g.,  \cite[Theorem 1.9]{GarciaTrillos2022}). In other words,  the sensitivity exactly characterizes the lower-order growth of the DRO risk in terms of the size of the ambiguity set. While this discussion applies only to models of DRO in which ambiguity sets are balls in the Wasserstein metric, there are settings in which it is of interest to use ambiguity sets with other structures \cite{DRO_robust_stat}.
    \end{itemize}
    More generally, we believe that our theory of sensitivity will find applications in many other statistical settings where data are perturbed before performing some downstream analysis.
    In addition to standard models where small amounts of noise arise due to measurement error (e.g., errors-in-variables models \cite{CarollStefanksi, FFKMC}), there are several other methodological reasons for practitioners to deliberately add noise to their data sets, typically in order to achieve numerical stability (e.g., in density estimation \cite{Jittering}, model selection \cite{Breiman}) or robustness (e.g., neural network training \cite{NoiseInjection}, simulation-extrapolation \cite{SimulationExtrapolation}).

    \subsection*{Organization}

    In Section~\ref{sec:prelim} we review some background on Riemannian geometry, and on the Hellinger, Fisher-Rao, and Wasserstein geometries, and in Section~\ref{subsec:sensitivity} we detail the notion of sensitivity.
    Section~\ref{sec:bound} discusses the Wasserstein-Cram\'er-Rao lower bound and the requisite regularity of statistical models. Section~\ref{sec:exact} discusses the models in which exact sensitivity-efficient estimation is possible, and Section~\ref{sec:asymp} discusses the asymptotic sensitivity-efficiency theory.
    Examples are given in Appendix~\ref{app:examples}, some deferred proofs appear in Appendix~\ref{app:proofs}, and the other appendices contain some further technical discussions of interest.
    
    \subsection*{Notation}

    For a function $f:\Rbb\to\Rbb$, we use the standard notation $f':\Rbb\to\Rbb$ for its derivative.
    For a function $F:\Rbb^a\to\Rbb^b$, its Jacobian is written $DF$ and we take the usual convention that $DF(x)\in\Rbb^{b\times a}$ is the linear operator representing the differential of $F$ at $x$.
    Thus, for any $x,h\in \Rbb^a$, we have
    \begin{equation*}
        \lim_{t\to 0}\frac{F(x+th)-F(x)}{t} = DF(x)h.
    \end{equation*}
    When $F$ is scalar-valued (i.e., $b=1$) its Jacobian $DF(x)\in\Rbb^{1\times a}$ is a row vector; in this case we write $\nabla F(x):=(DF(x))^{\top}\in\Rbb^a$ and call it the gradient of $F$, so that the gradient is a column vector related to the Jacobian by $\nabla F = (DF)^{\top}$.
    If $F$ is a function of multiple inputs, say $F:(\Rbb^a)^n\to \Rbb^b$, then we write $D_{x_i}F$ to denote the Jacobian with respect to the argument $x_i$ for $1\le i\le n$, meaning
    \begin{equation*}
        \lim_{t\to 0}\frac{F(x_1,\ldots, x_{i-1},x_i+th,x_{i+1},\ldots,x_n)-F(x_1,\ldots, x_n)}{t} = D_{x_i}F(x_1,\ldots, x_n)h;
    \end{equation*}
    and similar for the gradient $\nabla_{x_i}T_n$ when $b=1$.

    We write $\Pcal(\Rbb^m)$ for the collection of all Borel probability measure on $\Rbb^m$.
    Then, we write $\Pcal_2(\Rbb^m)$ for the collection of all $\mu\in\Pcal(\Rbb^m)$ with finite variance, meaning $\int_{\Rbb^m}\|x\|^2\diff\mu(x)<\infty$.
    A probability measure $\mu\in\Rbb^m$ is called \textit{absolutely continuous} if it has a density with respect to the $m$-dimensional Lebesgue measure on $\Rbb^m$, and we denote by $\mathcal{P}_{2,\mathrm{ac}} (\Rbb^m)$ the space of all absolutely continuous Borel probability measures with finite variance. 
    We also write $L^2_{\mu}(\Rbb^m;\Rbb)$ for the Hilbert space of square integrable functions with respect to $\mu$ from $\Rbb^m$ to $\Rbb$, as well as $L^2_{\mu}(\Rbb^m;\Rbb^m)$ for the Hilbert space of square integrable functions from $\Rbb^m$ to $\Rbb^m$ with respect to $\mu$.
    We also write $C_c^{k}(\Rbb^m)$ for the space of all compactly-supported  $k$-times continuously differentiable functions from $\Rbb^m$ to $\Rbb$; similarly, we write $C_c^{\infty}(\Rbb^m)$ for the space of all compactly-supported  infinitely differentiable functions from $\Rbb^m$ to $\Rbb$. 
    When a function is said to have a \textit{local} property (e.g., locally Lipschitz, locally integrable), we mean each point is contained in some neighborhood on which the function satisfies said property.
    By the norm $\|A\|$ of a matrix $A$ we mean the Frobenius norm.
        
    \section{Geometric Background}\label{sec:prelim}

    The perspective of this work is that the theory of sensitivity is related to the Wasserstein geometry in the same way that the theory of variance is related to the Hellinger (equivalently, Fisher-Rao) geometry; as both are Riemannian-like structures on the space of probability measures $\mathcal{P}(\Rbb^m)$, we spend some time describing this geometric picture and the requisite background in Riemannian geometry.
    The full details are included in Appendix~\ref{app:geom-background}, but we include Table~\ref{tab:geometry} to summarize the major points.
    For general background on Riemannian geometry, we direct the reader to standard texts (e.g., \cite{LeeSmooth,LeeRiemannian}).

    \begin{table}[t]
        \centering
        \resizebox{0.775\width}{!}{
        \begin{tabular}{|c|c|c|}
        \hline
        \thead{concept} & \thead{Hellinger geometry} & \thead{Wasserstein geometry} \\
        \hline
        \makecell{\,\\variation in the\\space of measures\\\,} & re-weighting & perturbation \\
        \hline
        \makecell{metric}& \makecell{\,\\$\begin{aligned}H^2(\mu,\nu) = \int_{\Rbb^m}\left(\sqrt{\frac{\diff \mu}{\diff \lambda}}-\sqrt{\frac{\diff \nu}{\diff \lambda}}\right)^2\diff \lambda\end{aligned}$\\\,} & \makecell{\,\\$\begin{aligned}W_2^2(\mu,\nu) = \underset{\pi\in\Pi(\mu,\nu)}{\min}\int_{\Rbb^m\times\Rbb^m}\|x-y\|^2\diff \pi(x,y)\end{aligned}$\\\,}\\
        \hline
        \makecell{representation of paths} & 
        \makecell{\,\\reaction equation\\\,\\$\begin{aligned}\partial_t\mu_t =\xi_t\mu_t\end{aligned}$\\\,\\
	$\begin{aligned}\bigg(\frac{\diff}{\diff t}\int_{\Rbb^m}f \diff \mu_t = \int_{\Rbb^m}f\,\xi_t\diff \mu_t \textnormal{ for } f \in C_c^\infty (\Rbb^m)\bigg)\end{aligned}$\\\,}
        
        & \makecell{\,\\continuity equation\\\,\\$\begin{aligned}\partial_t\mu_t +\textnormal{div}(v_t\mu_t)=0\end{aligned}$\\\,\\
	$\begin{aligned}\bigg(\frac{\diff}{\diff t}\int_{\Rbb^m}f \diff \mu_t = \int_{\Rbb^m}(\nabla f)^{\top} v_t\diff \mu_t \textnormal{ for } f \in C_c^\infty (\Rbb^m)\bigg)\end{aligned}$\\\,} \\
        \hline
        \makecell{dynamical\\formulation}& \makecell{\,\\$\begin{aligned}H^2(\mu,\nu) = \frac{1}{4}\min\left\{\int_{0}^{1}\|\xi_t\|^2_{L^2_{\mu_t}}\diff t: \begin{matrix}
            \partial_t\mu_t = \xi_t\mu_t, \\
            \mu_0 = \mu, \\
            \mu_1=\nu
        \end{matrix}
        \right\}\end{aligned}$\\\,} & \makecell{\,\\$\begin{aligned}W_2^2(\mu,\nu) = \min\left\{\int_{0}^{1}\|v_t\|^2_{L^2_{\mu_t}}\diff t: \begin{matrix}
            \partial_t\mu_t + \textnormal{div}(v_t\mu_t) = 0, \\
            \mu_0 = \mu, \\
            \mu_1=\nu
        \end{matrix}
        \right\}\end{aligned}$\\\,} \\
        \hline
        \makecell{\,\\constant-speed\\geodesics\\\,} & \makecell{\,\\$\gamma(t) = \mu_t$, where $\begin{aligned}\frac{\diff \mu_t}{\diff \lambda} = \left( (1-t) \sqrt{\frac{\diff \mu}{\diff \lambda}}  + t  \sqrt{\frac{\diff \nu}{\diff \lambda}}  \right)^2\end{aligned}$\\\,} & \makecell{\,\\$\begin{aligned}\gamma(t) = ((1-t)\id + t\,\transport_{\mu\to\nu})_{\#}\mu\end{aligned}$\\\,} \\
        \hline
        \makecell{\,\\tangent space\\\,} & $\begin{aligned}\Tan_{\mu}^{H}(\Pcal(\Rbb^m)) \subseteq L^2_{\mu}(\Rbb^m;\Rbb)\end{aligned}$ & $\begin{aligned}\Tan_{\mu}^{W}(\Pcal_2(\Rbb^m)) \subseteq L^2_{\mu}(\Rbb^m;\Rbb^m)\end{aligned}$ \\
        \hline
        \makecell{\,\\logarithmic map\\\,} & $\begin{aligned}\textnormal{Log}^{H}_{\mu}(\nu) = \sqrt{\frac{\diff \nu}{\diff \mu}} - 1\end{aligned}$ & $\begin{aligned}\textnormal{Log}^{W}_{\mu}(\nu) = \transport_{\mu\to\nu}-\id\end{aligned}$ \\
        \hline
        \end{tabular}}
        \caption{Summary and comparison of geometric concepts of the Fisher-Rao and Wasserstein geometries, as discussed in Subsection~\ref{subsec:hellinger} and Subsection~\ref{subsec:wasserstein}, respectively.}
        \label{tab:geometry}
    \end{table}

    Throughout, $\Mcal$ denotes a Riemannian manifold.
    Each point $x\in\Mcal$ has a \textit{tangent space} denoted by $\Tan_x(\Mcal)$ which represents the best possible linear approximation of the non-linear space $\Mcal$, locally at the point $x$.
    More concretely, $\Tan_x(\Mcal)$ is a vector space endowed with an inner product $g_x$; the collection of tangent spaces $\{\Tan_x(\Mcal)\}_{x\in\Mcal}$ is called the \textit{tangent bundle} and the collection of inner products $\{g_x\}_{x\in\Mcal}$ is called the \textit{metric tensor}.
    Each smooth curve $\gamma:[0,1]\to\Mcal$ has a notion of an initial velocity vector $\gamma'(0)$, and we have $\gamma'(0)\in\Tan_{\gamma(0)}(\Mcal)$; conversely, $\Tan_x(\Mcal)$ equals (the closure of the linear span of) the set of initial velocities $\gamma'(0)$ where $\gamma$ ranges over all smooth curves $\gamma:[0,1]\to\Mcal$ with $\gamma(0) = x$.
    In other words, $\Tan_x(\Mcal)$ is also the space of ``directions'' which smooth curves originating at $x$ may take. 

    We may also view the Riemannian manifold $\Mcal$ as a metric space $(\Mcal,d)$, where $d$ is the \textit{geodesic distance} defined for $x,y\in\Mcal$ via the following dynamical formulation
    \begin{equation*}
        d^2(x,y) := \inf\left\{\int_{0}^{1}\|\gamma'(t)\|^2_{\gamma(t)}\diff t: \gamma \textnormal{ is a smooth curve from $x$ to $y$}\right\}
    \end{equation*}
    where $\|\cdot\|_x$ is the norm on $\Tan_{x}(\Mcal)$ induced by the inner product $g_x$; in other words, $d(x,y)$ represents the length of the shortest path from $x$ to $y$, where the ``length'' of a path can be measured by averaging the size of the tangent vector at each point in time with respect to the metric tensor at the current location.
     Closely related is the concept of a ($2$-)absolutely continuous curve, which is a curve $\gamma:[0,1]\to\Mcal$ such that there exists a function $g\in L^2[0,1]$ satisfying $d(\gamma(s),\gamma(t)) \le \int_{s}^{t}g(u)\diff u$ for all $0\le s\le t\le 1$.

    For each $x\in\Mcal$, there is a fundamental correspondence between $\Mcal$ and $\Tan_x(\Mcal)$.
    Indeed, the \textit{logarithmic map} $\textrm{Log}_x:\Mcal\to\Tan_{x}(\Mcal)$ is defined so that $\textrm{Log}_x(y)=\gamma'(0)$ where $\gamma:[0,1]\to \Mcal$ is the minimizer of the dynamical formulation of the geodesic distance from $x$ to $y$, if it exists.
    The \textit{exponential map} $\textrm{Exp}:\Tan_{x}(\Mcal)\to \Mcal$ is the inverse of the logarithmic map, wherever it exists.

    In the rest of the paper, the Riemannian manifolds of interest are $\Pcal(\Rbb^m)$ endowed with the Hellinger distance $H$ and $\Pcal_2(\Rbb^m)$ endowed with the Wasserstein metric $W_2$; these are not bona fide Riemannian manifolds, but their metrics admit dynamical formulations which allow one to think of them as being very similar to Riemannian manifolds.
    Indeed, see Table~\ref{tab:geometry} for an enumeration of the elements above, specialized to these two geometries; we specifically emphasize the logarithm maps $\textnormal{Log}^H$ and $\textnormal{Log}^W$ which will play a role in our regularity assumptions on our statistical models.
    
    One aspect of these geometries which is different from the setting of abstract Riemannian manifolds is that 2-absolutely continuous curves have an exact characterization in terms of weak solutions to a certain linear partial differential equation (PDE); for instance, a curve $\{\mu_t\}_{0\le t\le 1}$ is absolutely continuous with respect to the Wasserstein metric if and only if \cite{AmbrosioGigliSavare} $t\mapsto \mu_t$ is weakly continuous, there exists a time-varying vector field $\{v_t\}_{0\le t\le 1}$ such that $\int_{0}^{1}\int_{\Rbb^m}\|v_t\|^2\diff \mu_t\diff t<\infty$, and, for every infinitely differentiable compactly-supported function $f:\Rbb^m\to\Rbb$, the function $t\mapsto\int_{\Rbb^m}f(x)\diff \mu_t(x)$ is absolutely continuous with derivative almost everywhere given by
    \begin{equation}\label{eqn:ContEq-body}
        \frac{\diff}{\diff t}\int_{\Rbb^m}f(x)\diff \mu_t(x)=\int_{\Rbb^m}(\nabla f(x))^{\top}v_t(x)\diff \mu_t(x).
    \end{equation}
    In other words, $(\{\mu_t\}_{0\le t\le 1},\{v_t\}_{0\le t\le 1})$ is a weak-sense solution to the \textit{continuity equation}, denoted $\partial_t\mu_t + \Div(v_t\mu_t) = 0$. If $\{\mu_t\}_{0 \le t\le 1}$ is a given (2-)absolutely continuous curve with respect to $W_2$, then it is known that there is a unique $\{v_t\}_{0\le t\le 1}$ such that $(\{\mu_t\}_{0\le t\le 1},\{v_t\}_{0\le t\le 1})$ solves the continuity equation and such that $v_t$ is the gradient of a function for Lebesgue a.e.~$0\le t\le 1$.
    A similar result holds \cite{Mielke} for the Hellinger geometry, where a curve $\{\mu_t\}_{0\le t\le 1}$ is absolutely continuous with respect to the Hellinger metric if and only if there is some sufficiently regular time-varying scalar field $\{\xi_t\}_{0\le t\le 1}$ satisfying the \textit{reaction equation} $\partial_t\mu_t = \xi_t\mu_t$  in the weak sense.
    See Appendix~\ref{app:geom-background} for further detail.
    In terms of transforming $\mu_t$ to $\mu_{t+\varepsilon}$, the scalar $\xi_t\diff t$ can be interpreted as the instantaneous change in log-density (i.e., $\xi_t$ is the score function) whereas the vector $v_t\diff t$ can be interpreted as the instantaneous change in position; throughout the paper, we write $\transport_{\mu\to\nu}$ for the optimal transport map from $\mu$ to $\nu$ for general $\mu,\nu$.
    
    \section{Sensitivity of Estimators}\label{subsec:sensitivity}

    Our first goal is to precisely define the fundamental notion of sensitivity and to explore some basic properties.
    To begin, consider some probability measure $P\in\Pcal_2(\Rbb^d)$, and we let $(\Omega,\Fcal,\Pbb)$ denote a generic probability space supporting a sequence $X_1,\ldots, X_n$ of i.i.d. random variables with distribution $P$.
    We also write $X$ for an additional copy of these random variables, for ease of notation in some expressions.
    We write $\Ebb_P$ for the expectation on $(\Omega,\Fcal,\Pbb)$, where we emphasize the law $P$ of $X$.

    \begin{definition}
    	\label{def:Sensiitivity}
        For $P\in\Pcal_2(\Rbb^d)$ and $T_n:(\Rbb^d)^n\to\Rbb$, the \textit{sensitivity} is the value
        \begin{equation*}
            \Sen_{P}(T_n) := \Ebb_{P}\left[\sum_{i=1}^{n}\|\nabla_{x_i}T_n(X_1,\ldots, X_n)\|^2\right]
        \end{equation*}
        if it is well-defined, and infinity otherwise.
    \end{definition}

\begin{remark}[absolute continuity of $P$, and definition of sensitivity]\label{rem:AC-sens}
To make sense of $ \Sen_{P}(T_n)$ for a large class of statistics $T_n$, in the sequel we will mostly consider probability measures $P \in \mathcal{P}_2(\Rbb^d)$ that are absolutely continuous with respect to the Lebesgue measure. Indeed, if $T_n$ is a statistic that is not differentiable everywhere (e.g., $T_n(X_1,\ldots,X_n) = \max\{X_1,\ldots, X_n\}$ in the case $d=1$), then the expression for $\Sen_{P}(T_n)$ is well-defined only if $P$ assigns zero probability to the set of points where $T_n$ is not differentiable.
In the other direction, when $P$ is assumed to be absolutely continuous with respect to the Lebesgue measure, and $T_n$ is, say, locally Lipschitz, then Rademacher's theorem implies that $\Sen_P(T_n)$ is well-defined (although it can still take the value $+\infty$ if the derivatives of $T_n$ are not square-integrable.)   
\end{remark}

   If we simply write $\nabla T_n(X_1,\ldots,X_n)$ for a column vector of length $dn$ (notably, not as a $d\times n$ or $n\times d$ matrix), then we can write $\Sen_{P}(T_n) = \Ebb_{P}\left[\|\nabla T_n(X_1,\ldots, X_n)\|^2\right]$.  This condensed notation is useful in some settings.
    
    As motivation for studying the sensitivity of an estimator, which is closely related to the simulations of the Introduction, we consider the following related notion.

    \begin{definition}
        For $P\in\Pcal_2(\Rbb^d)$, $T_n:(\Rbb^d)^n\to\Rbb$, and $\varepsilon>0$, the \textit{$\varepsilon$-sensitivity} of $T_n$ is defined as
        \begin{equation*}
            \Sen_{P,\varepsilon}(T_n) := \Ebb_{P}\left[\left|\frac{T_n(X_1',\ldots, X_n') - T_n(X_1,\ldots, X_n)}{\varepsilon}\right|^2\right],
        \end{equation*}
        where $X_i' := X_i + \xi_i$ for $1\le i\le n$ and $\xi_1,\ldots, \xi_n$ are i.i.d. samples from a Gaussian distribution with mean zero and covariance matrix $\varepsilon^2\idmat_d \in \Rbb^{d \times d}$ such that $\xi_1,\ldots, \xi_n$ are independent of $X_1,\ldots, X_n$.
        (Here, the expectation is taken over $X_1,\ldots, X_n$ and $\xi_1,\ldots, \xi_n$.)
    \end{definition}

    As we alluded to in the Introduction, the sensitivity can be viewed as a limiting case of $\varepsilon$-sensitivity as $\varepsilon\to 0$.
    Our goal is not to establish this convergence in the most general setting, but rather to illustrate the basic connection between these notions; the following easy result (proven in Appendix~\ref{app:proofs-sensitivity}) can certainly be generalized with further work (e.g., by using the ideas in \cite{Ponce2004}). 
     
    \begin{proposition}\label{prop:eps-sen-cvg}
        If $P\in\Pcal_2(\Rbb^d)$ and $T_n: (\Rbb^d)^n\to\Rbb$ is continuously twice differentiable and compactly supported, then $\Sen_{P,\varepsilon}(T_n) \to \Sen_{P}(T_n)$ as $\varepsilon\to 0$.
    \end{proposition}

    The proof of this result reveals two things worth emphasizing.
    First, the sensitivity is a ``universal'' quantity, in the sense that it does not directly rely on the distribution of $\xi_1,\ldots, \xi_n$ being Gaussian; although the $\varepsilon$-sensitivity may change for non-Gaussian $\xi$, the limit as $\varepsilon\to 0$ depends only on their mean and covariance matrix.
    Second, the proof identifies that the approximation error of the convergence $\Sen_{P,\varepsilon}(T_n) \to \Sen_{P}(T_n)$ vanishes as long as $\varepsilon\ll n^{-1/2}$; this provides some guidance on how to set $\varepsilon$ in simulations where it is desired that the $\varepsilon$-sensitivity should approximate the sensitivity, and this was used in the design of the simulation in Figure~\ref{fig:uniform}.

    Our first main result on the notion of sensitivity is the following fundamental lower bound, which requires a few technical assumptions in order to justify our application of the continuity equation.
    We include a sketch of the proof, which is simple and instructive.

    \begin{theorem}\label{thm:W-Chap-Robb-univar}
            Suppose that $\{P_t\}_{0\le t\le 1}\subseteq \Pcal_{2,\textnormal{ac}}(\Rbb^d)$ is 2-absolutely continuous with respect $W_2$ (that is, the continuity equation \eqref{eqn:ContEq-body} holds with some time-varying vector field $\{V_t\}_{0\le t\le 1}$), and that $T_n:(\Rbb^d)^n\to\Rbb$ is locally Lipschitz.
            Also suppose
    \begin{enumerate}
    \item[(i)] $\Ebb_{P_t}|T_n(X_1,\ldots, X_n)| < \infty$, for all $0\le t\le 1$ and $\int_0^1\Ebb_{P_t}\left[|T_n(X_1,\ldots, X_n)|^2\right]  \diff t < \infty$, and
    \item[(ii)] $ \int_0^1\Ebb_{P_t}\left[\|\nabla_{x_i} T_n(X_1,\ldots, X_n)\|^2\right]  \diff t < \infty$ for all $1\le i\le n$.
    \end{enumerate}
    Then we have
        \begin{equation*}
            \Sen_{P_t}(T_n) \ge \frac{\left(\frac{\diff}{\diff t}\Ebb_{P_t}[T_n(X_1,\ldots, X_n)]\right)^2}{n\Ebb_{P_t}[\|V_t(X)\|^2]}
        \end{equation*}
        for Lebesgue almost all $0\le t\le 1$ and all $n\in\Nbb$, provided that $\Ebb_{P_t}[\|V_t(X)\|^2]>0$.
    \end{theorem}

    \begin{proof}[Sketch of Proof]
        If $\{P_t\}_{0\le t\le 1}$ satisfies the continuity equation with time-varying vector field $\{V_t\}_{0\le t\le 1}$, then it can be shown that the product measures $\{P_t^{\otimes n}\}_{0\le t\le 1}$ satisfy the continuity equation in $\Rbb^{dn}$ with the time-varying vector field $\{V_t^{\otimes n}\}_{0\le t\le 1}$ defined by $V_t^{\otimes n}(x_1,\ldots, x_n) := (V_t(x_1),\ldots, V_t(x_n))^{\top}$.
        Then one can use the regularity assumptions to show that $T_n$ can serve as a test function in the continuity equation in $(\Rbb^d)^n$ for $(\{P_t^{\otimes n}\}_{0\le t\le 1},\{V_t^{\otimes n}\}_{0\le t\le 1})$, hence Cauchy-Schwarz allows us to bound:
        \begin{align*}
            \left|\frac{\diff}{\diff t}\Ebb_{P_t}[T_n(X_1,\ldots, X_n)]\right| &= \left|\Ebb_{P_t}\left[\left\langle V_t^{\otimes n}(X_1,\ldots, X_n),\nabla T_n(X_1,\ldots, X_n)\right\rangle\right]\right| \\
            &\le \sqrt{\Ebb_{P_t}\left[\left\|V_t^{\otimes n}(X_1,\ldots, X_n)\right\|^2\right]\Sen_{P_t}(T_n)} \\
            &= \sqrt{n\Ebb_{P_t}\left[\left\|V_t(X)\right\|^2\right]\Sen_{P_t}(T_n)}.
        \end{align*}
        Rearranging the terms gives the desired result.
    \end{proof}

    \begin{remark}[independent but not i.i.d.]\label{rem:indep-not-iid}
        The factor $n^{-1}$ in Theorem~\ref{thm:W-Chap-Robb-univar} is due to the i.i.d. structure of $X_1,\ldots, X_n$, but an analogous result holds assuming only independence.
        Indeed, if the distribution $\{P_t^i\}_{0\le t\le 1}$ of each $X_i$ satisfies the continuity equation with time-varying vector field $\{V_t^i\}_{0\le t\le 1}$, then the joint distribution of $X_1,\ldots, X_n$ satisfies the continuity equation with time-varying vector field
        $(x_1,\ldots, x_n)\mapsto (V_t^1(x_1),\ldots, V_t^n(x_n))^{\top}$ (cf.~\eqref{eqn:prod-CE}), and the resulting lower bound has denominator $\sum_{i=1}^{n}\Ebb_{P_t^i}[\|V_t^i(X_i)\|^2]$. 
        This generality will later be applied in the setting of fixed-design linear regression in Example~\ref{ex:regression}.
    \end{remark}
    \begin{remark}\label{rem:loc-Lip}
The assumption that $T_n$ is locally Lipschitz in Theorem \ref{thm:W-Chap-Robb-univar} is mild and is satisfied in most examples of interest. In fact, we believe that this assumption is quite sharp. Indeed, first note that locally Lipschitz functions are precisely the ones whose composition with an arbitrary locally absolutely continuous function remains locally absolutely continuous (see \cite[Theorem 3.55]{Leoni}). Then, observe that the latter property is implicitly used in our proof of Proposition \ref{prop:ContinuityEqnLocallyLipschitz}, which is the essential ingredient in the proof of Theorem \ref{thm:W-Chap-Robb-univar}.
\end{remark} 

    We now introduce a notion of sensitivity for vector-valued functions.
    To do this, we first need to introduce a notion of cosensitivity between two functionals.

    \begin{definition}
        For $P\in\Pcal_2(\Rbb^d)$ and $T_{n,1},T_{n,2}:(\Rbb^d)^n\to\Rbb$, their \textit{cosensitivity} is
        \begin{align*}
            \Cos_{P}(T_{n,1},T_{n,2}) & := \Ebb_{P}\left[\sum_{i=1}^{n}(\nabla_{x_i} T_{n,1}(X_1,\ldots, X_n))^{\top}\nabla_{x_i} T_{n,2}(X_1,\ldots, X_n)\right] \\
            &= \Ebb_{P}[(\nabla T_{n,1}(X_1,\ldots, X_n))^{\top}\nabla T_{n,2}(X_1,\ldots, X_n)],
        \end{align*}
        provided that the above expression is well-defined.
    \end{definition}

    Of course, a vector-valued function $T_n:(\Rbb^d)^n\to\Rbb^k$ may be written as $T=(T_{n,1},\ldots, T_{n,k})$ for functions $T_{n,1},\ldots, T_{n,k}:(\Rbb^d)^n\to\Rbb$, and it is desirable to construct a cosensitivity matrix $\Cos_{P}(T_n)\in\Rbb^{k\times k}$ such that the $(\ell,\ell')$-entry of $\Cos_{P}(T_n)$ is just the cosensitivity between $T_{n,\ell}$ and $T_{n,\ell'}$, or $(\Cos_{P}(T_n))_{\ell,\ell'} := \Cos_{P}(T_{n,\ell},T_{n,\ell'})$, for $1\le \ell,\ell'\le k$.
    This is done as follows, where we recall the convention that $D_{x_i}T_n(X_1,\ldots, X_n)$ is a matrix of shape $k\times d$ for each $1\le i \le n$, and that $DT_n(X_1,\ldots, X_n)$ is a matrix of shape $k\times dn$.
    
    \begin{definition}\label{def:cosens}
        For $P\in\Pcal_2(\Rbb^d)$ and $T_n:(\Rbb^d)^n\to\Rbb^k$, the \textit{cosensitivity matrix} is the $k\times k$ matrix
        \begin{align*}
            \Cos_{P}(T_n) & := \Ebb_{P}\left[\sum_{i=1}^{n}D_{x_i}T_n(X_1,\ldots, X_n)(D_{x_i}T_n(X_1,\ldots, X_n))^{\top}\right] \\
            &= \Ebb_{P}[D T_n(X_1,\ldots, X_n)(DT_n(X_1,\ldots, X_n))^{\top}],
        \end{align*}
        provided that the above expression is well-defined.
    \end{definition}

    Note that we used the same symbol $\Sen_{P}(\cdot)$ to denote the sensitivity, cosensitivity, and cosensitivity matrix; there is no risk of ambiguity since the argument always distinguishes between these notions.

    \section{Wasserstein-Cram\'er-Rao Lower Bound for Sensitivity}\label{sec:bound}

    In this section we consider the common statistical situation where our model of interest can be written as $\Pcal = \{P_{\theta}:\theta\in\Theta\}\subseteq \Pcal_2(\Rbb^d)$ for some open set $\Theta\subseteq \Rbb^p$ and that we wish to estimate some function $\chi(\theta)$ of the unknown parameter $\theta$; here, $\chi:\Theta\to\Rbb^k$ is a differentiable transformation.
    Throughout this section, we use the subscript $\theta$ as shorthand for the subscript $P_{\theta}$ (e.g., $\Ebb_{\theta}$ for $\Ebb_{P_{\theta}}$).

    We remark that all of the notions and results in this section use ``local'' properties of the model $\{P_{\theta}:\theta\in\Theta\}$, meaning that the structure at $\theta_0\in\Theta$ depends only on the measures $\{P_{\theta}:\theta\in U\}$ for an arbitrary open set $U\subseteq \Theta$ containing $\theta_0$.
    Consequently, one may replace all ``global'' conditions with ``local'' ones without changing the conclusions.
    However, we will not explicitly write this in the remainder of the section.
    
    We need to impose some notion of regularity on the model $\Pcal$ in order for our lower bounds to hold.
    Recall in Section~\ref{sec:prelim} we introduced the notation $\transport_{\mu\to\nu}$ for the optimal transport map from $\mu$ to $\nu$ which is unique if, say, $\mu$ possesses a density with respect to the Lebesgue measure; see also Subsection~\ref{subsec:wasserstein}.

    \begin{definition}\label{def:DWS}
        We say that $\Pcal=\{P_{\theta}:\theta\in\Theta\}$ is \textit{differentiable in the Wasserstein sense (DWS) at} $\theta\in\Theta$ if there exists a function $\Phi_{\theta}:\Rbb^d\to\Rbb^{d\times p}$, called the \textit{transport linearization}, such that
        \begin{equation}\label{eqn:DWS}
            \int_{\Rbb^d}\|\transport_{P_{\theta}\to P_{\theta+th}}(x)- x -t\Phi_{\theta}(x)h\|^2\diff P_{\theta}(x) = o(t^2)
        \end{equation}
        for all $h\in\Rbb^p$,  where $\transport_{P_{\theta}\to P_{\theta+th}}$ is the optimal transport map from $P_{\theta}$ to $P_{\theta+th}$. 
        We say $\Pcal$ is \textit{DWS} if $\Theta$ is open and if $\Pcal$ is DWS at all $\theta\in\Theta$.
    \end{definition} 

     Our notions of DWS and transport linearization are novel, but the transport linearization coincides with the gradient of an object in \cite{WIM} called the \textit{Wasserstein score function}; nonetheless, the notion of DWS is advantageous in two senses.
    First, it is an $L^2$ notion which is much more general than the pointwise notion of Wasserstein score function from \cite{WIM}, and this generality is indeed needed in most statistical applications.
    Second, it allows one to compute the transport linearization directly from the optimal transport maps, whereas the results of \cite{WIM} require computing the score function $\dot{\ell}_{\theta}$ for $\Pcal$ and then solving the so-called \textit{Poisson equation} relating $\dot{\ell}_{\theta}$ and $\Phi_{\theta}$.

    Throughout the remainder of this section we will provide various interpretations of the DWS condition and the transport linearization.
    For now, let us simply state that a model $\Pcal$ being DWS means, roughly speaking, that $\Pcal$ is a submanifold of $\Pcal_2(\Rbb^d)$, and the transport linearization $\Phi_{\theta}$ allows one to transform increments in $\Theta$ near $\theta$ into tangent vectors in $\Pcal$ at $P_{\theta}$.
    We also note that our notion of DWS requires the existence of unique optimal transport maps between the elements of the family $\Pcal$, which is typically achieved as a result of the assumption that all elements of $\Pcal$ have a density with respect to the Lebesgue measure (i.e., $\Pcal\subseteq\Pcal_{2,\mathrm{ac}}(\Rbb^d)$).

    \begin{remark}[DWS via pointwise differentiation]
        In practice, one typically verifies DWS by finding the optimal transport maps for all $\theta$, computing pointwise derivative as in \eqref{eqn:W-score-pointwise}, and then showing that one can exchange limit and integration in~\eqref{eqn:DWS}.
        However, in some settings (e.g., Example~\ref{eqn:W-fam-from-ODE}), one verifies DWS directly without computing the optimal transport maps themselves.
    \end{remark}

        In the following remarks, we provide some comparisons of DWS with the notion of differentiability in quadratic mean (DQM) studied in the classical theory for variance \cite[Section~5.5]{vanDerVaart}.

    \begin{remark}[geometric interpretation of DWS and DQM]\label{rem:DWS-v-DQM}
        In order to see the analogy between DWS and DQM, observe that \eqref{eqn:DWS} can be equivalently written in terms of the logarithmic map for the Wasserstein geometry defined in \eqref{eqn:LogMapWass} as
        \begin{equation}\label{eqn:DWS-log}
            \int_{\Rbb^d}\left\| \mathrm{Log}_{P_\theta}^W(P_{\theta+th}) -t\,\Phi_{\theta}h\right\|^2\diff P_\theta = o(t^2)\qquad \mbox{for all $h\in\Rbb^p$}.
        \end{equation}
        Now let us suppose that all elements of $\Pcal$ are dominated by a common $\sigma$-finite measure $\lambda$ on $\Rbb^d$, and that $\Pcal$ is DQM, meaning for each $\theta\in\Theta$ there exists a function $G_{\theta}:\Rbb^d\to\Rbb^p$, which equals $\sfrac{1}{2}$ times the usual \textit{score function at $\theta$}, such that
        \begin{equation}\label{eqn:DQM}
            \int_{\Rbb^d}\left|\sqrt{\frac{\diff P_{\theta+th}}{\diff \lambda}(x)} - \sqrt{\frac{\diff P_{\theta}}{\diff \lambda}(x)} -t\sqrt{\frac{\diff P_{\theta}}{\diff \lambda}(x)}\,(G_{\theta}(x))^{\top}h\right|^2\diff \lambda(x) = o(t^2)
        \end{equation}
        for all $h\in\Rbb^p$.
        By factoring out the term $\sqrt{\frac{\diff P_{\theta}}{\diff \lambda}(x)}$ from the square and recalling the logarithmic map for the Hellinger geometry defined in \eqref{eqn:LogMapHellinger}, an equivalent formulation of \eqref{eqn:DQM} is
        \begin{equation}\label{eqn:DQM-log}
            \int_{\Rbb^d}\left| \mathrm{Log}_{P_\theta}^H(P_{\theta+th}) -t\,(G_{\theta})^{\top}h\right|^2\diff P_\theta = o(t^2) \qquad \mbox{for all $h\in\Rbb^p$}.
        \end{equation}
        It is clear that \eqref{eqn:DWS-log}  and \eqref{eqn:DQM-log} are of precisely the same form, but in different geometries.
    \end{remark}

    \begin{remark}[DWS vs. DQM]\label{rem:DWS-v-DQM-exs}
        While we will see throughout the paper that many concrete examples of interest are both DWS and DQM (see Appendix~\ref{app:examples}), we can see that neither implies the other by noting that: (i) DWS does not make restrictions about the supports of its elements (e.g., the uniform scale family is DWS but not DQM), (ii) DQM does not make restrictions about the moments of its elements (e.g., the Cauchy location family is DQM but not DWS), and (iii) DQM does not make restrictions about the existence of transport maps between its elements (e.g., exponential families with discrete reference measure are DQM but not DWS).
    \end{remark}

    \begin{figure}
        \centering

            \begin{tikzpicture}[scale=1.25]		
			
			
			\fill[use Hobby shortcut, closed=true, color=myblue, opacity=0.5, cm={-1, 0, 0, 1, (-7.5,0)}]
			(-5.5,0) .. (-3.5,-0.5) .. (-3.5,2) .. (-4.25,1) .. (-5.5,0);
			\node[black, above right] at (-5.5, 2) {$\textcolor{myblue}{\Theta}\subseteq\mathbb{R}^p$};
			
			\draw[thin,stealth-stealth] (-4,-1) -- ++(0, 3.5);
			\draw[thin,stealth-stealth] (-5.5,0) -- ++(4,0);
			
			\draw[very thick, color=mypurple,stealth-stealth] (-4.5, 0.25) -- (-2.5, 0.25);
			\draw[very thick, color=mypurple,stealth-stealth] (-3.75, -0.5) -- (-3.75, 1.5);
			
			\draw[very thick, color=red,-stealth] (-3.75, 0.25) -- ++(0.75,0);
			\draw[very thick, color=red,-stealth] (-3.75, 0.25) -- ++(0, 0.75);
			\node[red] at (-3.4, 0.9) {$h_1$};
			\node[red] at (-3.1, -0.15) {$h_2$};
			
			\filldraw[color=blue] (-3.75, 0.25) circle (2pt);
			\node[blue, above right] at (-4.05, 0.25) {$\theta$};

			\draw [thin, -stealth] (-1.5,0.5) to [out=20,in=160] (0.5,0.5);

			
			\coordinate (A) at (6, 0.1);
			\coordinate (B) at (2.7, -1.1);
			\coordinate (C) at (1, 0.85);
			\coordinate (D) at (4.6, 2.2);
			\fill[thin, opacity=0.5, color=myblue] (A) to[out=170,in=40] (B) to[out=120,in=350] (C) to[out=50,in=170] (D) to[out=330,in=120] (A);
			\node[black, above right] at (1, 2) {$\textcolor{myblue}{\Pcal}\subseteq\mathcal{P}_2(\mathbb{R}^d)$};
			
			\draw[very thick, color=mypurple, domain=2.25:4.25, smooth,stealth-stealth,cm={cos(-130) ,-sin(-130) ,sin(-130) ,cos(-130) ,(5.6125,-1.5)}] plot (\x,{0.25*(\x-3.5)^2});
			\draw[very thick, color=mypurple, domain=2.25:4.75, smooth,stealth-stealth] plot (\x,{2-exp(-0.5*(\x-3.5))});

			\filldraw[thin, opacity=0.75, color=red!30, xscale=-0.75, yscale=0.625,shift={(-6.125,1)}] (0, 0) to (2, -1) to (3,1) to (1, 2) to cycle;
			\draw[very thick, color=red,-stealth] (3.5, 1) -- ++(0.35, -0.45);
			\draw[very thick, color=red,-stealth] (3.5, 1) -- ++(0.5,0.25);
			\node[red] at (3.75, 1.5) {$\Phi_{\theta}h_1$};
			\node[red] at (3.25, 0.5) {$\Phi_{\theta}h_2$};
			
			\filldraw[color=blue] (3.5, 1) circle (2pt);
			\node[blue, right] at (2.8, 1.1) {$P_{\theta}$};
			
		\end{tikzpicture}
        \caption{A statistical model that is differentiable in the Wasserstein sense (DWS).
        The transport linearization $\Phi_{\theta}$ at $\theta\in\Theta$ is a linear operator that transforms increments $h$ at $\theta \in \Theta$ (left) into tangent vectors $\Phi_{\theta}h$ at $P_\theta$ for the corresponding path in the model $\Pcal=\{P_{\theta}:\theta\in\Theta\}$ (right). 
        The Wasserstein information matrix $J(\theta)$ gives rise to an inner product on the space of increments at $\theta$ (which is isomorphic to $\Rbb^p$), where the inner product between $h_1,h_2$ at $\theta\in\Theta$ (left) is set to be the $L^2_{P_\theta}(\Rbb^d;\Rbb^d)$ inner product between $\Phi_{\theta}h_1$ and $\Phi_{\theta}h_2$ (right).}        \label{fig:WassersteinScore}
    \end{figure}
    
    Now we return to the statistical setting from the beginning of Section~\ref{subsec:sensitivity}.
    An \textit{estimator} is any measurable function $T_n:(\Rbb^d)^n\to \Rbb^k$, and an estimator is \textit{unbiased for $\chi(\theta)$} if we have $\Ebb_{\theta}[T_n(X_1,\ldots, X_n)] = \chi(\theta)$ for all $\theta\in\Theta$.
    Then we come to the following notion, which extends the notion from \cite{WIM}.

    \begin{definition}\label{def:WIM}
        If $\Pcal$ is DWS at $\theta\in\Theta$ its \textit{Wasserstein information matrix} at $\theta$ is the matrix $J(\theta)\in\Rbb^{p\times p}$ defined via
        \begin{equation}\label{eqn:Wass-inf}
            J(\theta) := \Ebb_{\theta}\left[(\Phi_{\theta}(X))^{\top}\Phi_{\theta}(X)\right], \qquad \mbox{for each $\theta\in\Theta$.}
        \end{equation}
        
    \end{definition}

    \begin{remark}[geometric interpretation of transport linearization and Wasserstein information matrix]\label{rem:inf-matrix-comparison}
        The transport linearization and Wasserstein information matrix both have geometric interpretations which are analogous to the classical case, and which can be visualized in        Figure~\ref{fig:WassersteinScore}.
        First, the transport linearization $\Phi_{\theta}$ transforms increments at $\theta$ to tangent vectors at $P_{\theta}$ in $\Pcal$ when endowed with the Wasserstein geometry, whereas $\sfrac{1}{2}$ times the classical score function (which may be regarded as a log-likelihood linearization) $G_{\theta}$ transforms increments at $\theta$ to tangent vectors at $P_{\theta}$ in $\Pcal$ when endowed with the Hellinger geometry.
        Second, note that the Wasserstein information matrix $J(\theta)$ gives rise to an inner product on increments at $\theta$, via
        \begin{equation*}
            h_1^{\top}J(\theta)h_2 = \Ebb_{\theta}\left[(\Phi_{\theta}h_1(X))^{\top}\Phi_{\theta}h_2(X)\right] = \langle \Phi_{\theta}h_1,\Phi_{\theta}h_2\rangle_{L^2_{P_{\theta}}(\Rbb^d;\Rbb^d)}
        \end{equation*}
        for all $h_1,h_2\in\Rbb^p$; as such the Wasserstein information matrix describes a geometry on $\Theta$ induced by the parameterization $\theta\mapsto P_{\theta}$ into the Wasserstein geometry, whereas the Fisher information matrix describes a geometry induced by the parameterization into the Fisher-Rao geometry \cite{Rao, Kass, Amari}.
        We direct the reader to \cite{AmariMatsuda,WIM} for further details on the Riemannian geometry of statistical models in the Wasserstein metric.
    \end{remark}

    \begin{remark}[PDE interpretation of transport linearization and Wasserstein information matrix]\label{rem:PDE}
        The transport linearization and Wasserstein information matrix can also be interpreted in terms of PDEs representing the flow of mass under the model $\Pcal$.
        That is, the map $x\mapsto \Phi_{\theta}(x)h$ represents the first-order approximation of the optimal flow from $P_{\theta}$ to $P_{\theta+th}$ at $t=0$, i.e., a particle at $x$ is transported to position $x+t\Phi_{\theta}(x)h+o(t)$.
        Similarly, the Wasserstein information matrix encodes the costs of these flows, since $t^2\cdot h^{\top}J(\theta)h = \Ebb_{\theta}[\|t\Phi_{\theta}(X)h\|^2]\approx W_2^2(P_{\theta+th},P_{\theta})$ is the $L^2(P_{\theta})$ norm squared of this first-order term.
        Relatedly, when $P_{\theta}$ has a smooth positive density, the transport linearization $x\mapsto\Phi_{\theta}(x)h$ can be characterized as the solution of a linearized Monge-Amp\`ere equation obtained by differentiating, with respect to $t$, the Monge-Amp\`ere equation satisfied by the optimal transport map from $P_{\theta}$ to $P_{\theta+th}$.
    \end{remark}

    These interpretations are analogous to well-known aspects of the Fisher information matrix in information geometry \cite{Amari}, e.g., that $I(\theta)$ can be considered as an inner product on $\Theta$ and that we have the asymptotic expansion $H^2(P_{\theta+th},P_{\theta}) \approx \frac{1}{4}t^2h^{\top}I(\theta)h$.
    
    Now we have the following fundamental result (proven in Appendix~\ref{app:proofs-bound}), which is a version of the Cram\'er-Rao lower bound for the Wasserstein geometry; this result extends earlier versions of the lower bound \cite{WIM, WCR_achieve} that we discuss further below.
    We emphasize that unbiasedness is not a restriction here since we could just as well define $\chi$ via $\chi(\theta):=\Ebb_{\theta}[T_n(X_1,\ldots,X_n)]$.

    \begin{theorem}[Wasserstein-Cram\'er-Rao lower bound]\label{thm:WCR}
        Suppose that  $\Pcal=\{P_{\theta}:\theta\in\Theta\}\subseteq\Pcal_{2,\textnormal{ac}}(\Rbb^d)$ is DWS, $\chi:\Theta\to\Rbb^k$ is continuously differentiable and $T_n:(\Rbb^d)^n\to\Rbb^k$ is a locally Lipschitz unbiased estimator of $\chi(\theta)$ for all $\theta\in\Theta$.
        Also suppose:
        \begin{itemize}
            \item[(i)] The function $J:\Theta\to\Rbb^{p\times p}$ is  continuous, and $J(\theta)\in\Rbb^{p\times p}$ is invertible for all $\theta\in\Theta$.
            \item[(ii)] The function $\theta\in\Theta\mapsto \Cos_{\theta}(T_n)\in\Rbb^{k\times k}$ is continuous, and
            \item[(iii)] The function $\theta\in\Theta\mapsto \Var_{\theta}(T_n)\in\Rbb^{k\times k}$ is locally bounded.
        \end{itemize}
        Then, we have
        \begin{equation}\label{eqn:WCR}
            \Cos_{\theta}(T_n) \succeq \frac{1}{n}D\chi(\theta)(J(\theta))^{-1}(D\chi(\theta))^{\top}, \qquad \quad \mbox{for all $\theta\in\Theta$}.
        \end{equation}
    \end{theorem}

    We emphasize that we do not directly assume that one can differentiate under the integral in Theorem~\ref{thm:WCR}, while most standard statements of the Cram\'er-Rao lower bound (e.g., \cite[Section~4.5]{Keener}) assume so; in geometric terms, such results make the a priori assumption that the reaction equation (see Section~\ref{sec:prelim}) may be applied to the estimator of interest, while we give conditions which \textit{imply} that the continuity equation may be applied to the estimator of interest.
    For instance, condition (iii) is used to guarantee a corresponding pathwise condition in Theorem~\ref{thm:W-Chap-Robb-univar}.
    This also contrasts earlier work on the  Wasserstein-Cram\'er-Rao lower bound \cite{WIM, WCR_achieve} which assumes (in fact, without explicitly stating so) that the continuity equation may be applied to the estimator of interest.
    We believe it is interesting to ask even in the classical (i.e., Hellinger) setting what conditions on the estimator imply that it can serve as a test function in the reaction equation.
    
    Our next goal, and the focus of the remainder of the paper, is to understand which estimators achieve the lower bound of Theorem~\ref{thm:WCR}, either exactly or asymptotically.

    \section{Exact Sensitivity-Efficiency}\label{sec:exact}

    Our next goal is to understand in which models $\Pcal = \{P_{\theta}: \theta\in\Theta\}$ one can construct an unbiased estimator of a given estimand $\chi(\theta)$ that achieves equality in the Wasserstein-Cram\'er-Rao bound (Theorem~\ref{thm:WCR}).    
    This is analogous to the construction of exponential families in the classical case, since those are exactly the models in which one can achieve equality in the Cram\'er-Rao lower bound \cite{Wijsman, Joshi}.

    The recent work \cite[Theorem~1]{WCR_achieve} provides a characterization of \textit{functions} which achieve equality in the Wasserstein-Cram\'er-Rao bound in an arbitrary model, but our goal is statistical in that we seek a characterization of \textit{models} in which some bona fide \textit{estimator} achieves equality in the Wasserstein-Cram\'er-Rao bound.
    Here, we draw a distinction between functions (which may depend on the parameter $\theta$) and estimators (which depend only on the samples $X_1,\ldots, X_n$ from $P_{\theta}$).
    
    We give a notion of efficiency for sensitivity, regarding equality in the Wasserstein-Cram\'er-Rao bound.
    
    \begin{definition}\label{def:sens-eff}
        Suppose that $\Pcal$ is DWS, that $J(\theta)$ is invertible for each $\theta\in\Theta$, and that $T_n:(\Rbb^d)^n\to \Rbb^k$ is an unbiased estimator of $\chi(\theta)$, for some differentiable function $\chi:\Theta\to\Rbb^k$.
        We say that $T_n$ is \textit{sensitivity-efficient} if we have
        \begin{equation}\label{eqn:sens-eff}
            \Cos_{\theta}(T_n) = \frac{1}{n}D\chi(\theta)(J(\theta))^{-1}(D\chi(\theta))^{\top}, \qquad \mbox{for all $\theta\in\Theta$.}
        \end{equation}
        
    \end{definition}

    \begin{remark}[variance-efficiency vs.~sensitivity-efficiency]
        When there is no risk of ambiguity, we will simply write ``efficient'' to mean ``sensitivity-efficient'', but this differs from the usual terminology, whereby ``efficient'' means ``variance-efficient''.
    \end{remark}

    Our main result will characterize the estimands for which unbiased sensitivity-efficient estimation is possible.
    For now, an elementary observation towards this end is the following:
    If a DWS model $\Pcal$ admits an unbiased sensitivity-efficient estimator $T_n$ of $\chi(\theta)$ and if $A\in\Rbb^{k\times k}$ and $v\in\Rbb^k$ are arbitrary, then $AT_n+v$ is an efficient unbiased estimator of $A\chi(\theta) + v$.

    
    \begin{definition}\label{def:W-fam}
        For a locally Lipschitz function $\phi:\Rbb^d\to\Rbb^k$ and a differentiable function $\chi:\Theta\to\Rbb^k$, a \textit{transport family} is a DWS model $\Pcal = \{P_{\theta}: \theta\in\Theta\}\subseteq\Pcal_{2,\mathrm{ac}}(\Rbb^d)$ for open $\Theta\subseteq\Rbb^p$  if we have
        \begin{equation}\label{eqn:W-fam}
            \Phi_{\theta}(x) = (D\phi(x))^{\top}(\Lambda(\theta))^{-1}D\chi(\theta)
        \end{equation}
        for Lebesgue almost every $x\in\Rbb^d$, for all $\theta\in\Theta$, and where $\Lambda(\theta):=\Ebb_{\theta}[D\phi(X)(D\phi(X))^{\top}]$ is assumed to be invertible.
        The function $\phi$ is called the \textit{potential}, and the function $\chi$ is called the \textit{parameterization}.
        We say that $\Pcal$ is in \textit{natural parameterization} if $\chi(\theta) = \theta$ for all $\theta\in\Theta$.
    \end{definition}
    
    Next, we make some comparisons between transport families and exponential families.

    \begin{remark}[transport families vs.~exponential families]\label{rem:exp-v-transport}
        In order to see the analogy between transport families and exponential families, let us consider an exponential family $\Pcal:=\{P_{\theta}:\theta\in\Theta\}$ with parameterization $\eta:\Theta\to\Rbb^k$ and sufficient statistic $\phi:\Rbb^d\to\Rbb^k$, and expectation $\chi(\theta)=\Ebb_{\theta}[\phi(X)]$.
        Then, an exponential family is characterized by $G_{\theta}$, which equals $\sfrac{1}{2}$ times the usual score function, taking on the following form:
        \begin{equation*}
            G_{\theta}(x) = \frac{1}{2}(D\eta(\theta))^{\top}(\phi(x) -\chi(\theta)) = \frac{1}{2}(D\eta(\theta))^{\top}\phi(x) - \frac{1}{2}\nabla \Lambda(\theta)
        \end{equation*}
        where $\Lambda(\theta) = \log\int\exp((\eta(\theta))^{\top}\phi(x))\diff \lambda(x)$ and $\lambda$ is a dominating measure for $\Pcal$.
        We chose our notation to highlight, as much as possible, the analogy between these types of models, namely that the potential is analogous to the sufficient statistic, there is some normalization required via $\Lambda$, and the parameterization and the statistic interact only linearly.
        However, the analogy fails in some ways.
        For example, in a transport family the functions $\phi$ and $\chi$ are related via $D_{\theta}\Ebb_{\theta}[\phi(X)] = D\chi(\theta)$, whereas in an exponential family the functions $\phi,\chi$, and $\eta$ are related via $(D\eta(\theta))^{\top}\chi(\theta) = (D\eta(\theta))^{\top}\Ebb_{\theta}[\phi(X)] = \nabla \Lambda(\theta)$.
        Another example is that the regularity required from the natural statistics are different in each case; in a transport family we require $\phi$ to be differentiable while in an exponential family we require $\phi$ only to be measurable.
    \end{remark}

    \begin{remark}[no explicit formula for transport families]\label{rem:no-explicit-form}
        Exponential families may equivalently be defined implicitly in terms of the score function or through an explicit description of the density, with respect to a common reference measure.
        In the case of transport families, we have an implicit definition in terms of the Wasserstein linearization, but we do not in general have an explicit formula for the distribution except when $p=1$. (See Example~\ref{eqn:W-fam-from-ODE}.)
    \end{remark}

     It is also useful throughout the paper to keep in mind the special case of $p=k=1$, in which case the definition of a transport family simply states that the transport linearization can be written as $\Phi_{\theta}(x) =\chi'(\theta)\nabla\phi(x)/\Lambda(\theta)$ as we discussed in the Introduction.

    Our main result in this section (proven in Appendix~\ref{app:proofs-exact})  is the following, which, in spirit, shows that transport families are precisely the models in which unbiased sensitivity-efficient estimation is possible.
    Along the way we utilize some general calculations in transport families, which are developed in Lemma~\ref{lem:W-fam-properties}.

    \begin{theorem}\label{thm:efficiency-W-fam}
        Suppose $\Pcal=\{P_{\theta}:\theta\in\Theta\}\subseteq\Pcal_{2,\textnormal{ac}}(\Rbb^d)$ is DWS, that $p=k$, that $J(\theta)$ and $D\chi(\theta)$ are invertible for all $\theta\in\Theta$, and that $\chi:\Theta\to\Rbb^k$ is continuously differentiable.
        Then:
        \begin{itemize}
            \item[(i)] If $\Pcal$ is a transport family with potential $\phi$, parameterization $\chi$, and connected $\Theta$, and if $\Ebb_{\theta}[\|\phi(X)\|]<\infty$ for all $\theta\in\Theta$, $\theta \mapsto \Ebb_\theta[\phi(X)]$ is continuous, and the maps $\theta\mapsto \Ebb_{\theta}[\|D\phi(X)\|^2]$ and $\theta\mapsto J(\theta)$ are locally Lebesgue integrable, then the estimator $T_n:(\Rbb^d)^n\to\Rbb^k$ defined via
            \begin{equation}\label{eqn:W-fam-estimator}
                T_n(X_1,\ldots, X_n) = \frac{1}{n}\sum_{i=1}^{n}\phi(X_i)
            \end{equation}
            is an unbiased sensitivity-efficient estimator of $\chi(\theta)+v$ for all $\theta\in\Theta$, for some $v\in\Rbb^k$.
            \item[(ii)] If there exists an unbiased sensitivity-efficient estimator $T_n$ of $\chi(\theta)$ for all $n$, and if the assumptions of Theorem~\ref{thm:WCR} hold, then $\Pcal$ is a transport family with potential $\phi=T_1$ and parameterization $\chi$.
        \end{itemize}        
    \end{theorem}

    A practical consequence of Theorem~\ref{thm:efficiency-W-fam} is that we can easily find sensitivity-efficient estimators  if we can recognize a given model as a transport family.
    This allows us to handle several basic statistical examples, which we detail in Appendix~\ref{app:examples}: location families, scale families, the Pareto index model, univariate and multivariate Gaussian models, linear regression with general error distributions, and more.


    \section{Asymptotic Sensitivity-Efficiency}\label{sec:asymp}
    
    Our final goal is to understand, in a general DWS model $\Pcal =\{P_{\theta}:\theta\in\Theta\}$, how to construct an estimator of $\chi(\theta)$ which asymptotically achieves the Wasserstein-Cram\'er-Rao bound (Theorem~\ref{thm:WCR}).
    This is analogous to the construction of maximum likelihood estimators and related estimators in the classical case, which achieve asymptotic variance-efficiency under mild assumptions.
    Our results in this section, generally speaking, resolve the open question posed in the discussion of \cite[Section~5]{WIM}, which asks how to develop a general-purpose method for constructing estimators that asymptotically achieve the Wasserstein-Cram\'er-Rao bound.
    The proofs of all results in this section can be found in Appendix~\ref{app:proofs-asymp}, and further discussion of asymptotic sensitivity-efficiency in the multivariate case can be found in Appendix~\ref{app:discussions-asymp}.

    We begin with a precise notion of asymptotic sensitivity-efficiency.

    \begin{definition}\label{def:asymp-sens-eff}
        Suppose that $\Pcal$ is DWS, and that $T_n:(\Rbb^d)^n\to\Rbb^k$ is an estimator of $\chi(\theta)$.
        We say that $T_n$ is \textit{(strongly) asymptotically sensitivity-efficient} if we have
        \begin{equation}\label{eqn:sen-eff}
            n\sum_{i=1}^{n}D_{x_i} T_n(X_1,\ldots, X_n)(D_{x_i} T_n(X_1,\ldots, X_n))^{\top}\to D\chi(\theta)(J(\theta))^{-1}(D\chi(\theta))^{\top}
        \end{equation}
        in probability, when $X_1,X_2,\ldots$ are i.i.d. from $P_{\theta}$, for each $\theta\in\Theta$.
    \end{definition}

    The expectation of the left-hand side in the above display is exactly the cosensitivity matrix of $T_n$ times $n$ from Definition~\ref{def:cosens}; so the above definition can be seen as an asymptotic counterpart of exact sensitivity-efficiency from Definition~\ref{def:sens-eff}. 
    Note, however, that the definition of asymptotic sensitivity requires convergence in probability rather than in expectation.
    
    \begin{remark}[asymptotic sensitivity-efficiency vs. asymptotic variance-efficiency]
        There are some important differences between asymptotic sensitivity-efficiency and asymptotic variance-efficiency; recall that an estimator $T_n$ is called asymptotically variance-efficient if we have
        \begin{equation}\label{eqn:var-eff}
            \sqrt{n}(T_n(X_1,\ldots, X_n)-\chi(\theta)) \to \Ncal\Big(0,\,D\chi(\theta)(I(\theta))^{-1}(D\chi(\theta))^{\top}\Big)
        \end{equation}
        in distribution when $X_1,X_2,\ldots$ are i.i.d. from $P_{\theta}$, for each $\theta\in\Theta$.
        First, note that \eqref{eqn:sen-eff} requires a scaling of $n$ while \eqref{eqn:var-eff} requires a scaling of $\sqrt{n}$.
        Second, note that, as $n\to\infty$, \eqref{eqn:sen-eff} requires convergence to a non-random constant while \eqref{eqn:var-eff} requires convergence to non-degenerate (in fact, Gaussian) distribution.
        For these reasons, we have added the word ``strongly'' to the definition above.
    \end{remark}

    In this section we will mostly assume our estimators are consistent (that is, converging in probability to the estimand) rather than unbiased.
    This is more natural from the asymptotic point of view, because of the following.
    
    \begin{remark}[change of estimand]\label{rem:param-equivar}
        Suppose that $T_n$ is a consistent and asymptotically sensitivity-efficient estimator of $\theta$, and suppose that $\chi:\Theta\to\Rbb^k$ is continuously differentiable.
        Then set $\tilde{T}_n:=\chi(T_n)$ and use the chain rule to compute
        \begin{equation*}
            \sum_{i=1}^{n}D_{x_i}\tilde{T}_n(D_{x_i}\tilde{T}_n)^{\top} = D\chi(T_n)\left(\sum_{i=1}^{n}D_{x_i}T_n(D_{x_i}T_n)^{\top}\right)(D\chi(T_n))^{\top}.
        \end{equation*}
        Thus, $\tilde{T}_n$ is a consistent and asymptotically sensitivity-efficient estimator of $\chi(\theta)$.
        In the classical setting of variance, an analogous result is true because of the delta method.
    \end{remark}

    This simple observation allows us to handle some new examples, e.g., the general scale family, and estimating the variance of a Gaussian with unknown mean (see Appendix~\ref{subsec:ex-asymp}).

    The main results of this section will provide some general-purpose procedures for achieving asymptotic sensitivity-efficiency, under suitable assumptions.
    We specialize to the case of $\chi(\theta) = \theta$, and we note by Remark~\ref{rem:param-equivar} that this is no loss of generality.

    \subsection{Wasserstein Projection Estimators (WPE)}\label{subsec:WPE}

    The first method we consider involves projecting the empirical measure of the data onto the model.    
    
    \begin{definition}\label{def:WPE}
        Suppose $\Pcal=\{P_{\theta}:\theta\in\Theta\}$ is DWS, and that $X_1,\ldots,X_n$ are i.i.d.~samples from $P_{\theta}$ for some unknown $\theta\in\Theta$.
        The \textit{Wasserstein projection estimator (WPE)} is defined as
    \begin{equation}\label{eq:WPE}
        T_n^{\WPE} := \underset{\theta\in\Theta}{\arg\min}\,W_2^2(P_{\theta},\bar P_n)
    \end{equation}
    where $\bar P_n := \frac{1}{n}\sum_{i=1}^{n}\delta_{X_i}$ is the empirical measure of the data.
    \end{definition}

    It is known (see \cite[Theorem~2.2]{Bernton}) that, under mild assumptions, the minimizing set in the above definition is always non-empty and that $T_n^{\WPE}:(\Rbb^d)^n\to\Rbb^p$ can be taken to be a measurable function.
    Moreover, if $\Pcal$ is identifiable (i.e., $\theta_0\neq\theta_1$ implies $P_{\theta_0} \neq P_{\theta_1}$ for all $\theta_0,\theta_1\in \Theta$), and some technical assumptions hold, then $T_n^{\WPE}\to \theta^{\ast}$ almost surely, as $n \to \infty$, when $X_1,X_2,\ldots$ are i.i.d. samples from $P_{\theta^{\ast}}$ \cite[Theorem~2.1]{Bernton}.
    (Note that we did not require identifiability at any previous point in the paper.)

    In a location family, we show (Example~\ref{ex:Locationfamily2}) that the WPE is the sample mean, which we know (Example~\ref{ex:loc-fam}) is sensitivity-efficient, but we aim to establish this more generally.
   
    While our main results will show that the WPE can be asymptotically sensitivity-efficient, let us first give the basic calculation used in our proofs, which parallels the proof of the fact that the MLE is asymptotically variance-efficient.
    Our notation departs from previous sections slightly, as we now use $\theta^{\ast}$ for the parameter of the population distribution, and $\theta$ for a generic parameter (i.e., dummy variable).
    Assuming for the sake of simplicity that $\Theta\subseteq\Rbb$, we write $\mathcal{L}:(\Rbb^d)^n\times\Theta\to\Rbb$ for the function $\mathcal{L}(x_1,\ldots, x_n,\theta):=W_2^2(P_{\theta},n^{-1}\sum_{i=1}^{n}\delta_{x_i})$, and we note that the first-order optimality conditions for the WPE state $(\partial/\partial \theta)\mathcal{L}(X_1,\ldots, X_n,\hat\theta_n) = 0$, where $\hat \theta_n:= T_n^{\WPE}(X_1,\ldots, X_n)$.
    Differentiating these first-order conditions with respect to $x_i$, via the chain rule, yields
    \begin{equation}
       0  = \nabla_{x_i}\frac{\partial}{\partial \theta}\mathcal{L}(X_1,\ldots, X_n,\hat{\theta}_n) + \frac{\partial^2 }{\partial \theta^2} \mathcal{L}(X_1,\ldots,X_n,\hat{\theta}_n)\nabla_{x_i} \hat{\theta}_n.
       \label{eqn:WPE_ImplicitFunctionThm}
    \end{equation}
    So rearranging \eqref{eqn:WPE_ImplicitFunctionThm}, taking the squared norms of each term, and summing over $1\le i\le n$ shows
       \begin{equation}
        \begin{split}\label{eqn:WPE_Sensitivity}
            &n\sum_{i=1}^n \left\| \nabla_{x_i} \hat{\theta}_n(X_1, \dots, X_n) \right\|^2 \\
            &= \left( \frac{\partial^2 }{\partial \theta^2} \mathcal{L}(X_1,\ldots,X_n,\hat{\theta}_n)  \right)^{-2} \frac{1}{n}\sum_{i=1}^n \left\|  n\nabla_{x_i} \frac{\partial }{\partial \theta} \mathcal{L}(X_1,\ldots,X_n,\hat{\theta}_n)\right\|^2.
        \end{split}
    \end{equation}
    The last step is to show that these factors converge to $(J(\theta^{\ast}))^{-2}$ and $J(\theta^{\ast})$, respectively; roughly speaking, this comes from the fact that $\Phi_{\theta^{\ast}}$ and $J(\theta^{\ast})$ are closely related to the higher-order derivatives of $W_2^2(P_{\theta},P_{\theta^{\ast}})$ at $\theta=\theta^{\ast}$, in a similar way as the Fisher information is related to higher-order derivatives of the KL divergence.
    Turning these formal calculations above into a rigorous argument takes significant work, which we separate into the univariate ($d=1$) and multivariate ($d\ge 2$) cases.

    Before we state our precise results, we make some remarks about this formal calculation. First, observe that the first-order optimality conditions for the WPE involve a derivative with respect to $\theta$, which we then differentiated with respect to $X_1,\ldots, X_n$.
    This contrasts the calculation of the asymptotic distribution of the MLE which involves differentiating twice with respect to $\theta$.
    This is the reason that we will encounter higher-order and mixed partial derivatives in $X_1,\ldots, X_n,\theta$ in our analysis, while the classical Cram\'er-Rao theory involves only higher-order derivatives in $\theta$.
    (Broadly, this is because the reaction equation~\eqref{eqn:ReactEq} is an ODE while the continuity equation~\eqref{eqn:ContEq} is a PDE.)
    Second, observe that the first-order optimality conditions for the WPE can be interpreted as the orthogonality of the ``residual vector'' $$\transport_{\hat \theta_n\to \bar P_n}(\cdot)-\id\in L^2_{P_{\hat\theta_n}}(\Rbb^d;\Rbb^d)$$to the tangent plane of $\Pcal$ at $\hat \theta_n$.
    See Figure~\ref{fig:WassersteinProjection} for an illustration.

    \begin{figure}
        \centering
        \begin{tikzpicture}[scale=1.25]		
			
			
			\fill[use Hobby shortcut, closed=true, color=myblue, opacity=0.5, cm={-1, 0, 0, 1, (-7.5,0)}]
			(-5.5,0) .. (-3.5,-0.5) .. (-3.5,2) .. (-4.25,1) .. (-5.5,0);
			\node[black, above right] at (-5.5, 2) {$\textcolor{myblue}{\Theta}\subseteq\mathbb{R}^p$};
			
			\draw[thin,stealth-stealth] (-4,-1) -- ++(0, 3.5);
			\draw[thin,stealth-stealth] (-5.5,0) -- ++(4,0);
			
			\filldraw[color=blue] (-3.75, 0.25) circle (2pt);
			\node[blue, above right] at (-4.05, 0.25) {$\theta$};
			
			\draw[very thick, color=red,-stealth] (-4.25, 0.75) -- ++(0.75,0);
			\draw[very thick, color=red,-stealth] (-4.25, 0.75) -- ++(0, 0.75);
			
			\filldraw[color=mygreen] (-4.25, 0.75) circle (2pt);
			\node[mygreen, above left] at (-4.25, 0.75) {$\hat\theta_n$};

			\draw [thin, -stealth] (-1.5,0.5) to [out=20,in=160] (0.5,0.5);

			
			\coordinate (A) at (6, 0.1);
			\coordinate (B) at (2.7, -1.1);
			\coordinate (C) at (1, 0.85);
			\coordinate (D) at (4.6, 2.2);
			\fill[thin, opacity=0.5, color=myblue] (A) to[out=170,in=40] (B) to[out=120,in=350] (C) to[out=50,in=170] (D) to[out=330,in=120] (A);
			\node[black, above right] at (1, 2) {$\textcolor{myblue}{\Pcal}\subseteq\mathcal{P}_2(\mathbb{R}^d)$};

			\filldraw[thin, opacity=0.75, color=red!30, xscale=-0.75, yscale=0.4,shift={(-6.5,3.75)}] (0, 0) to (2, -1) to (3,1) to (1, 2) to cycle;

			\filldraw[color=blue] (3.5, 1) circle (2pt);
			\node[blue, right] at (3, 0.8) {$P_{\theta}$};

			\draw[very thick, color=red,-stealth] (3.75, 1.8) -- ++(0.3, -0.4);
			\draw[very thick, color=red,-stealth] (3.75, 1.725) -- ++(0.65,0.18);
			
			\filldraw[color=mygreen] (3.8, 1.75) circle (2pt);
			\node[mygreen, right] at (3.2, 1.75) {$P_{\hat\theta_n}$};
			
			\filldraw[color=mygreen] (3.8, 2.75) circle (2pt);
			\node[mygreen, right] at (3.2, 2.75) {$\bar P_{n}$};
			\draw[very thick, color=mygreen, dash pattern=on 2pt off 1pt,-stealth] (3.8,2.75) -- (3.8, 1.825);
			
		\end{tikzpicture}
        \caption{The Wasserstein projection estimator (WPE) in a statistical model.
        For the sake of simplicity, we write $\hat \theta_n$ in place of $T_n^{\WPE}$.
        The first-order optimality conditions for the WPE state that the ``residual'' from $P_{\hat\theta_n}$ to $\bar P_n$ is orthogonal to $\Phi_{\hat\theta_n}h$ for all $h\in \Rbb^p$.}
        \label{fig:WassersteinProjection}
    \end{figure}


    We begin with the simple case of $d=1$ (where $p$ and $k$ remain general) in which we can make a direct argument for asymptotic sensitivity-efficiency.
    The advantage of $d=1$ is that we have an explicit representation of the optimal transport maps in terms of quantile functions, and hence for other objects of interest (e.g., transport linearization, Wasserstein information matrix).
    Thus, let us write $F_{\theta}$ and $F_{\theta}^{-1}$ for the cumulative distribution function and the quantile function of $P_{\theta}$, respectively, for each $\theta\in\Theta$.
    Then we have the following alternative characterization of the Wasserstein information matrix.

    \begin{lemma}\label{DWS-1s}
        A model $\Pcal=\{P_{\theta}:\theta\in\Theta\}\subseteq\Pcal_{2,\mathrm{ac}}(\Rbb)$ is DWS at $\theta\in\Theta$ if and only if $\nabla_{\theta}F^{-1}_{\theta}(\,\cdot\,)$ exists and lies in $L^2([0,1];\Rbb^p)$, in which case the Wasserstein information matrix is given by $J(\theta) = \int_{0}^{1}\left(\nabla_{\theta}F_{\theta}^{-1}(u)\right)\left(\nabla_{\theta}F_{\theta}^{-1}(u)\right)^{\top}\diff u$.
    \end{lemma}

    This leads us to the following main result, in the special case $d=1$.
    Recall that, in a Hilbert space $\mathcal{H}$, we say that $f_n\to f$ \textit{weakly} if $\langle f_n,g\rangle\to \langle f,g\rangle$ for all $g\in\mathcal{H}$.
    
    \begin{theorem}\label{thm:WPE-efficient}
        Suppose $\Pcal=\{P_{\theta}:\theta\in\Theta\}\subseteq \Pcal_{2,\textnormal{ac}}(\Rbb)$ is DWS.
        Also suppose:
        \begin{itemize}
            \item[(i)] the map $\theta \mapsto \nabla_{\theta}F^{-1}_{\theta}(\,\cdot\,)$ from $\Theta$ to $L^2([0,1];\Rbb^p)$ is continuous, and
            \item[(ii)] the map $\theta \mapsto \nabla_{\theta}^2F^{-1}_{\theta}(\,\cdot\,)$ from $\Theta$ to $L^2([0,1];\Rbb^{p\times p})$ is weakly continuous.
        \end{itemize}
        If $X_1,X_2,\ldots$ are i.i.d. samples from $P_{\theta^{\ast}}$ and $T_n^{\WPE}$ is a consistent estimator of $\theta^{\ast}\in\Theta$ with $J(\theta^{\ast})\succ 0$, then we have
        \begin{equation*}
            n\sum_{i=1}^{n}\left(\frac{\partial T_n^{\WPE}}{\partial x_i}(X_1,\ldots, X_n)\right)\left(\frac{\partial T_n^{\WPE}}{\partial x_i}(X_1,\ldots, X_n)\right)^{\top} \to (J(\theta^{\ast}))^{-1}
        \end{equation*}
        in probability as $n\to\infty$, for each $\theta^{\ast}\in\Theta$.  
    \end{theorem}
    
    The preceding result allows us to revisit some earlier examples (see Appendix~\ref{subsec:ex-asymp}) including the general location family, the general scale family, and the uniform scale family; note that the last of these was previously studied in Example~\ref{ex:Unif-scale} in the Introduction.

    In the multivariate case of $d\ge 2$, the analysis of WPE is more complicated because we no longer have a convenient representation of transport maps in terms of quantile functions.
    Nonetheless, we can enumerate some conditions under which the WPE is asymptotically sensitivity-efficient, although the conditions are very restrictive and do not cover any cases of interest.
    In this sense, we think of our results as more of a proof-of-concept for the multidimensional case, and we welcome future work which relaxes these assumptions.
    We state the following result for $p=1$ for the sake of simplicity.

\begin{theorem}\label{thm:WPE-efficient-gen}
    Suppose that $\Pcal=\{P_{\theta}:\theta\in\Theta\}\subseteq\Pcal_2(\Rbb^d)$ with $\Theta\subseteq\Rbb$ is DWS, and that, if $X_1,X_2,\ldots$ are i.i.d. samples from $P_{\theta^{\ast}}$, then $T_n^{\WPE}$ is a consistent estimator of $\theta^{\ast}\in\Theta$ and Assumption \ref{assum:AdditionalAssumps} holds; then, we have
    \begin{equation*}
        n\sum_{i=1}^{n}\left\|\nabla_{x_i}T_n^{\WPE}(X_1,\ldots, X_n)\right\|^2 \to (J(\theta^{\ast}))^{-1}
    \end{equation*}
    in probability as $n\to\infty$, for each $\theta^{\ast}\in\Theta$.
\end{theorem}

Fundamentally, the difficulty of establishing a result on asymptotic sensitivity-efficiency of multidimensional WPE is that the current state-of-the-art results in statistical optimal transport do not have a sufficiently strong understanding of
    \begin{itemize}
        \item consistency for some high-order mixed partial derivatives of the optimal potentials, and
        \item statistical estimates on the ellipticity of the Laguerre cells in a random power diagram.
    \end{itemize}
    Nonetheless, we believe that our result provides a basic foundation for the statistical usage of the WPE, and we hope that it will motivate some further work on statistical optimal transport.
    See Appendix~\ref{app:discussions-asymp} for the precise statement of our assumptions, and for several detailed discussions about them.
    
\subsection{One-Step Estimators (OSE)}

The second method for constructing asymptotically sensitivity-efficient estimators is a modification of the famous ``one-step estimation'' procedure from the classical case \cite[Section~5.7]{vanDerVaart}.
This will allow us to construct asymptotically sensitivity-efficient estimators in many multivariate (i.e., $d\ge 2$) examples.
We require a new notion for this; if $\Pcal=\{P_{\theta}:\theta\in\Theta\}$ with $\Theta\subseteq\Rbb^p$ is DWS and its transport linearization is $\Phi_{\theta}$, then any function $\phi_{\theta}:\Rbb^d\to\Rbb^p$ satisfying $(D\phi_{\theta})^{\top}=\Phi_{\theta}$ and $\int_{\Rbb^d}\phi_{\theta}\diff P_{\theta} = 0$ for all $\theta\in\Theta$ is called the \textit{potential} of the model $\Pcal$, if it exists.
Note that we have $(D\phi_{\theta})^{\top}=\Phi_{\theta}$ and not $D\phi_{\theta}=\Phi_{\theta}$, since $\Phi_{\theta}$ is defined to be the Jacobian of the transport map as a function of the parameter,  which is a map $\Rbb^p\to\Rbb^d$.

\begin{definition}
    For any estimator $T_n:(\Rbb^d)^n\to\Rbb^p$, called the \textit{pilot estimator}, the \textit{one-step estimator (OSE)} is defined via
    \begin{equation*}
        T^{\OSE}_n:=T_n + (J(T_n))^{-1}\frac{1}{n}\sum_{i=1}^{n}\phi_{T_n}(X_i)
    \end{equation*}
    provided that $J(T_n)\succ 0$.
\end{definition}

To understand this construction and our terminology, observe for $\theta\approx \theta^{\ast}$ that we have
\begin{equation*}
    \nabla_{\theta}W_2^2(P_{\theta},P_{\theta^{\ast}})\approx-2\int_{\Rbb^d}\phi_{\theta}\diff P_{\theta^{\ast}}\approx-\frac{2}{n}\sum_{i=1}^{n}\phi_{\theta}(X_i) \quad\textnormal{and}\quad D_{\theta}^2W_2^2(P_{\theta},P_{\theta^{\ast}})\approx 2J(\theta).
\end{equation*}
Indeed, the first approximation comes from \cite[Theorem~8.4.7]{AmbrosioGigliSavare} and the centering convention $\int_{\Rbb^d}\phi_{\theta^{\ast}}\diff P_{\theta^{\ast}}=0$, and the second comes from Remark~\ref{rem:PDE}.
As such, the OSE can be interpreted as the result of applying a single step of Newton's method in order to approximately solve the Wasserstein projection problem~\eqref{eq:WPE}, when initializing at $T_n$.
Iterating this procedure provides further progress towards a local minimizer of the WPE objective, but it turns out (as in the classical case) that one step suffices for asymptotic sensitivity-efficiency, provided that the pilot estimator is not too bad.
Also note that $T_n$ being consistent implies that $T_n^{\OSE}$ is consistent, as long as $J$ is continuous and $\phi_{\theta}$ and its derivatives satisfy a suitable uniform law of large numbers (ULLNs). 

\begin{theorem}\label{thm:OSE}
    Suppose that $\Pcal=\{P_{\theta}:\theta\in\Theta\}\subseteq\Pcal_{2,\mathrm{ac}}(\Rbb^d)$ is DWS with potential $\phi_{\theta}$, that $\theta^{\ast}\in\Theta$ is fixed, and also for $\theta^{\ast}$ that
    \begin{itemize}
        \item[(i)] There exists a map $D_{\theta}\phi_{\theta}:\Rbb^{d}\to\Rbb^{p\times p}$ such that $\|\phi_{\theta^{\ast}+th}-\phi_{\theta^{\ast}}-t D_{\theta}\phi_{\theta^{\ast}}h\|_{L^2(P_{\theta^{\ast}+th})}=o(t)$ for each $h\in\Rbb^p$ as $t\to 0$.
        \item[(ii)] The sets $\{\phi_{\theta}(\cdot):\|\theta-\theta^{\ast}\|<\varepsilon\}$, $\{D_{\theta}\phi_{\theta}(\cdot):\|\theta-\theta^{\ast}\|<\varepsilon\}$, and $\{(\Phi_{\theta}(\cdot))^{\top}\Phi_{\theta}(\cdot):\|\theta-\theta^{\ast}\|<\varepsilon\}$ satisfy a ULLN under $P_{\theta^{\ast}}$ for sufficiently small $\varepsilon>0$,
        \item[(iii)] $J(\theta)$ is invertible and $\theta\mapsto J(\theta)$ is continuously differentiable for all $\theta\in\Theta$, and
        \item[(iv)] For each $\theta^{\ast}\in\Theta$ and sufficiently small $h\in\Rbb^p$, the potential $\phi_{\theta^{\ast}}:\Rbb^d\to\Rbb^p$ is locally Lipschitz and satisfies
            \begin{equation*}
                \int_{0}^{1}\left(\Ebb_{\theta^{\ast}+th}\left[\|\phi_{\theta^{\ast}}(X)\|^2\right]+\Ebb_{\theta^{\ast}+th}\left[\|\Phi_{\theta^{\ast}}(X)\|^2\right]\right)\diff t < \infty.
            \end{equation*}
    \end{itemize}
    If $X_1,X_2,\ldots$ are i.i.d.~samples from $P_{\theta^{\ast}}$ and $T_n:(\Rbb^d)^n\to\Rbb^p$ is a locally Lipschitz and consistent estimator of $\theta^{\ast}$ satisfying $\sum_{i=1}^{n}\|D_{x_i}T_n\|^2 = O_{P_{\theta^{\ast}}}(n^{-1})$, then we have
    \begin{equation*}
        n\sum_{i=1}^{n}D_{x_i}T_n^{\OSE}(X_1,\ldots,X_n)\left(D_{x_i}T_n^{\OSE}(X_1,\ldots,X_n)\right)^{\top}\to (J(\theta^{\ast}))^{-1}
    \end{equation*}
    in probability as $n\to\infty$, for each $\theta^{\ast}\in\Theta$.
\end{theorem}

This result allows us to construct asymptotically sensitivity-efficient estimators in multidimensional (i.e., $d\ge 2$) models which are not transport families.
In particular (see Appendix~\ref{subsec:ex-asymp}), we construct such estimators for submodels of Gaussian covariance estimation, including estimating the correlation when the marginal variances are known, and estimating the marginal variances when the correlation is known.

\begin{remark}[asymptotically linear estimators]
    Many estimators $T_n$ of interest admit an expansion of the form
    \begin{equation*}
        T_n - \chi(\theta) = \frac{1}{n}\sum_{i=1}^{n}\psi_{\theta}(X_i) + R_n
    \end{equation*}
    under $P_{\theta}$, where $R_n$ is a remainder term satisfying suitable regularity; in this case, the asymptotic sensitivity of $T_n$ is equal to the sensitivity of the influence function $\psi_{\theta}$, i.e.,
    \begin{equation*}
        n\Sen_{\theta}(T_n) \to\Ebb_{\theta}\left[D\psi_{\theta}(X)(D\psi_{\theta}(X))^{\top}\right].
    \end{equation*}
    In this setting, it is easy to see that $T_n$ is asymptotically sensitivity-efficient whenever $\psi_{\theta}$ satisfies $\psi_{\theta}(x) = (J(\theta))^{-1}\phi_{\theta}(x)$, where $\phi_{\theta}$ is the potential of $\Pcal$ and the remainder $R_n$ satisfies $\Sen_{\theta}(R_n) = o(n^{-1})$.
\end{remark}

\subsection{Super-Sensitivity-Efficient Estimators}

In this last subsection, we identify a subtlety in our existing asymptotic theory by addressing some basic results concerning ``superefficiency'' for the sensitivity, which are analogous to classical considerations of superefficiency for the variance (see \cite[Chapter~8.1]{vanDerVaart}).
We illustrate this phenomenon through the following basic example, which is analogous to Hodges' estimator \cite[Example~8.1]{vanDerVaart}, except that we need a smooth approximation of the indicator function in order to make $T_n$ locally Lipschitz.

\begin{example}[smoothed Hodges' estimator]\label{ex:Hodges}
    Suppose that $P_{\theta}$ is the distribution of a Gaussian random variable with mean $\theta\in\Rbb$ and variance 1.
    We have already shown (see Example~\ref{ex:loc-fam}) that $\bar{X}_n:=n^{-1}\sum_{i=1}^{n}X_i$ is an unbiased sensitivity-efficient estimator of $\theta$, but now we consider a slight modification of this estimator.
    That is, let $\eta:[0,\infty)\to[0,\infty)$ be a Lipschitz function with $\eta\equiv 0$ on $[0,1]$ and $\eta\equiv 1$ on $[2,\infty)$, and define $T_n:=\bar{X}_n\eta(|\bar{X}_n|n^{1/4})$.
    Now note that we have $T_n = \bar{X}_n$ on $\{|\bar{X}_n| > 2n^{-1/4}\}$ and that for $\theta\neq 0$ we have $\Pbb_{\theta}(|\bar{X}_n| > 2n^{-1/4})\to 1$; consequently, we have $n\sum_{i=1}^{n}\|\nabla_{x_i}T_n\|^2 \to 1$ in probability under $P_{\theta}$, for each $\theta\neq 0$.
    However, we also have $T_n = 0$ on $\{|\bar{X}_n| < n^{-1/4}\}$ as well as $\Pbb_{0}(|\bar{X}_n| <n^{-1/4})\to 1$; consequently, we have $n\sum_{i=1}^{n}\|\nabla_{x_i}T_n\|^2 \to 0$ in probability under $P_0$.
    Although this calculation shows that $T_n$ strictly dominates $\bar{X}_n$ in terms of the pointwise asymptotic sensitivity, we see in the simulation in Figure~\ref{fig:hodges} that the pointwise asymptotics can be misleading; indeed, for each fixed large $n$ there exists a small $\theta\neq 0$ with very large sensitivity.
\end{example}

\begin{figure}
        \centering
        \includegraphics[width=0.45\linewidth]{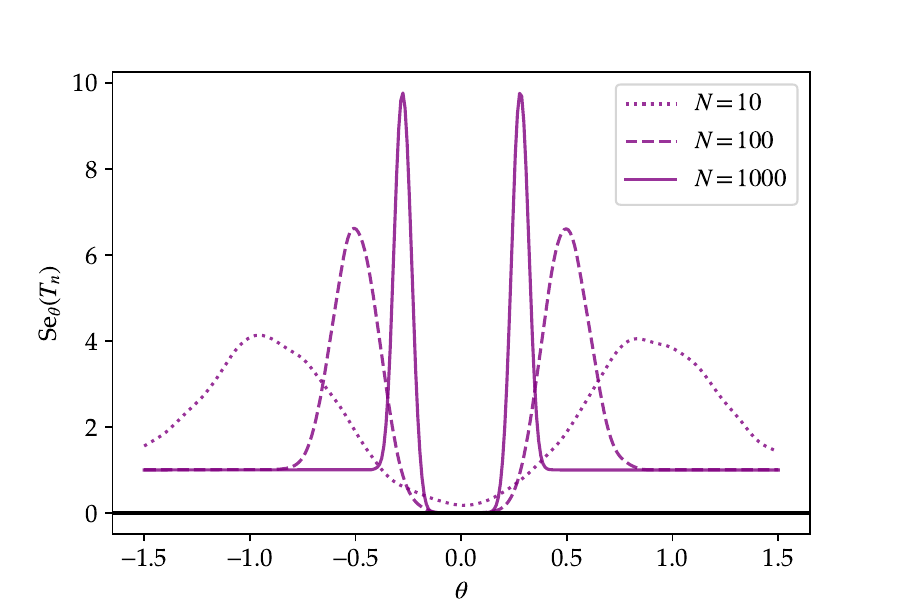}
        \caption{The super-sensitivity-efficiency phenomenon of the Hodges-type estimator in the Gaussian location family of Example~\ref{ex:Hodges}.
        We plot a Monte Carlo estimate of the sensitivity for each $\theta$ as $n$ increases.
        For each $\theta\neq 0$, we have $n\Sen_{\theta}(T_n)\to (J(\theta))^{-1}=1$ and also $n\Sen_{0}(T_n)\to 0$; but for each fixed $n$, there exist $\theta\approx 0$ with large sensitivity.}
        \label{fig:hodges}
    \end{figure}

As in the classical case of the variance, the preceding example illustrates that the pointwise notions of asymptotic efficiency we have so far studied can be misleading.
One way to preclude this phenomenon is to impose some locally uniform regularity on our estimators, which leads to the following result reminiscent of H\'ajek's notion of regularity in the classical case (see \cite[Chapter~8.5]{vanDerVaart}); it states that $T_n$ satisfies the Wasserstein-Cram\'er-Rao lower bound uniformly over local alternatives.

\begin{theorem}\label{thm:loc-uniform-WCR}
    Suppose $\Pcal=\{P_{\theta}:\theta\in\Theta\}\subseteq\Pcal_{2,\textrm{ac}}(\Rbb^d)$ is DWS, let $T_n:(\Rbb^d)^n\to\Rbb$ be any locally Lipschitz estimator such that $\chi_n(\theta):=\Ebb_{\theta}[T_n]$ is continuously differentiable for each $n\in\Nbb$, and suppose that  $\chi:\Theta\to\Rbb$ is continuously differentiable and satisfies
    \begin{equation}\label{eqn:uniform}
        \sup_{\|h\|\le R}\left\|\nabla\chi_n\left(\theta+\frac{h}{\sqrt{n}}\right)-\nabla \chi(\theta)\right\| \to 0
    \end{equation}
    for all $R>0$ as $n\to\infty$.
    Also suppose that hypotheses (i), (ii), and (iii) of Theorem~\ref{thm:WCR} are satisfied.
    Then,
    \begin{equation}\label{eqn:uniform-LB}
        \liminf_{n\to\infty}\inf_{\|h\|\le R}n\Sen_{\theta + \frac{h}{\sqrt{n}}}(T_n) \ge (\nabla \chi(\theta))^{\top}(J(\theta))^{-1}\nabla \chi(\theta)
    \end{equation}
    for all $\theta\in\Theta$ and $R>0$.
\end{theorem}

The right side of \eqref{eqn:uniform-LB} is the same as the lower bound~\eqref{eqn:WCR} from Theorem~\ref{thm:WCR} (after moving the factor $n^{-1}$ to the other side), but the left side is the worst-case sensitivity taken over local alternatives.
In particular, this result precludes the super-sensitivity-efficiency phenomenon of Example~\ref{ex:Hodges} for any estimator satisfying the mild regularity condition~\eqref{eqn:uniform}.

To see that the Hodges-type estimator of Example~\ref{ex:Hodges} does not contradict Theorem~\ref{thm:loc-uniform-WCR}, we check that it fails to satisfy the hypothesis~\eqref{eqn:uniform} at $\theta = 0$.
Indeed, $\chi_n(\theta) := \Ebb_{\theta}[T_n]$ converges pointwise to $\chi(\theta) := \theta$ (with $\nabla\chi\equiv 1$), but this convergence is not locally uniform near $\theta = 0$; along the local alternatives $\theta = h/\sqrt{n}$, the sample mean $\bar{X}_n$ concentrates at scale $n^{-1/2}$ around $h/\sqrt{n}$, hence lies in the region $\{|\bar{X}_n| < n^{-1/4}\}$ with probability tending to 1, where $T_n \equiv 0$.
Consequently, we have $\nabla\chi_n(h/\sqrt{n}) \to 0 \neq 1 = \nabla\chi(0)$, so that
\begin{equation*}
    \sup_{\|h\|\le R}\left\|\nabla\chi_n\left(\frac{h}{\sqrt{n}}\right)-\nabla \chi(0)\right\| \to 1
\end{equation*}
for all $R>0$, and \eqref{eqn:uniform} fails.
Thus, there is no contradiction with the pointwise super-sensitivity-efficiency $n\Sen_0(T_n)\to 0$ established in Example~\ref{ex:Hodges}; rather, Theorem~\ref{thm:loc-uniform-WCR} shows that this superefficiency is an artifact of the irregularity of $T_n$ near $\theta=0$.

On the other hand, the uniform convergence condition~\eqref{eqn:uniform} is satisfied whenever there exists a continuously differentiable function $\chi:\Theta\to\Rbb$ and a real sequence $\alpha_n\to 1$ such that $T_n$ is an unbiased estimator of $\alpha_n\chi(\theta)$ for all $n$; in this case, $\nabla\chi_n(\theta) = \alpha_n\nabla\chi(\theta)$ converges to $\nabla\chi(\theta)$ locally uniformly by the continuity of $\nabla\chi$.
This implies that all of the estimators we have seen in this paper (except for the Hodges' estimator in Example~\ref{ex:Hodges}) satisfy a locally uniform version of the Wasserstein-Cram\'er-Rao lower bound.

\section{Conclusion}


The fundamental insight of this work is geometric, namely that the theory of sensitivity comes from the Wasserstein geometry in precisely the same way that the theory of variance comes from the Fisher-Rao (equivalently, Hellinger) geometry.
In fact, apart from some technical considerations, the proof of each lower bound is simply an application of Cauchy-Schwarz in the tangent space for a suitable geometry on the model.
This raises the question of what theory would arise if one considered other geometries beyond Fisher-Rao, Hellinger, and Wasserstein.
What measures of instability would arise?
What notion of information appears in the lower bound?
How can one achieve the lower bound in statistical settings of interest?

We view this broad research program to be, in some sense, a mathematical foundation for the notion of ``stability'' from the predictability-computability-stability (PCS) framework of veridical data science, as initiated in \cite{YuStability} and developed through many recent works (e.g., \cite{YuStabilityHDSR, BarterYu, LimYu, YuKumbier}).
Within this framework, it is advocated that statistical procedures should be \textit{stable}, which ``includes, but is much broader than, the concept of sampling variability'' \cite{YuKumbier}.
    More generally, it is argued that statistical procedures should be stable against ``reasonable'' variations in the data or the model, the interpretation of which can vary across scientific applications.
    We emphasize the specification of a class of reasonable variations to a model is akin to considering a particular geometry on that model; such a specification describes the ``directions'' in which each distribution may move (e.g., resampling corresponds to a variation by re-weighting as in the Hellinger and Fisher-Rao geometries, and adding Gaussian noise corresponds to a variation by additive perturbation as in the Wasserstein geometry).


    Lastly, we hope that our study of the asymptotic sensitivity of the Wasserstein projection estimator (WPE) will inspire further work in statistical optimal transport.
    For example, verifying the technical hypotheses of Theorem~\ref{thm:WPE-efficient-gen}, or developing an entirely new strategy of proof, will require a much finer understanding of the semi-discrete optimal transport problem than what is currently known.

\section*{Acknowledgments}
The authors would like to thank Nabarun Deb, Tudor Manole, Debarghya Mukherjee, and Bin Yu for helpful conversations. AQJ and BS acknowledge the support of NSF grants DMS-2311062 and DMS-2515520. NGT was supported by NSF-DMS grant 2236447.

\bibliographystyle{imsart-number} 
\bibliography{refs}       


\newpage
\begin{appendix}

\addtocontents{toc}{\protect\value{tocdepth}=-1}

\section{Examples}\label{app:examples}

In this appendix, we discuss many concrete examples of statistical models in which we can construct exactly- or asymptotically sensitivity-efficient estimators of interest.
In Appendix~\ref{subsec:ex-exact} we give examples of exactly sensitivity-efficient estimators arising from Section~\ref{sec:exact}, and in Appendix~\ref{subsec:ex-asymp} we give examples of asymptotically sensitivity-efficient estimators arising from Section~\ref{sec:asymp}.
We also give Figure~\ref{fig:uniform} which plots Monte Carlo estimates of the bias, variance, and sensitivity, in the uniform scale family of Example~\ref{ex:Unif-scale} and Example~\ref{ex:unif-scale-fam}.

\subsection{Exact Sensitivity-Efficiency}\label{subsec:ex-exact}

We collect several examples and remarks about constructing exactly sensitivity-efficient estimators.
These are all based on casting a given statistical model as a transport family (Definition~\ref{def:W-fam}) and then applying our main result (Theorem~\ref{thm:efficiency-W-fam}) which identifies an exactly sensitivity-efficient estimator in terms of the sample mean of the potential.

    \begin{figure}
        \centering
        \includegraphics[width=0.95\linewidth]{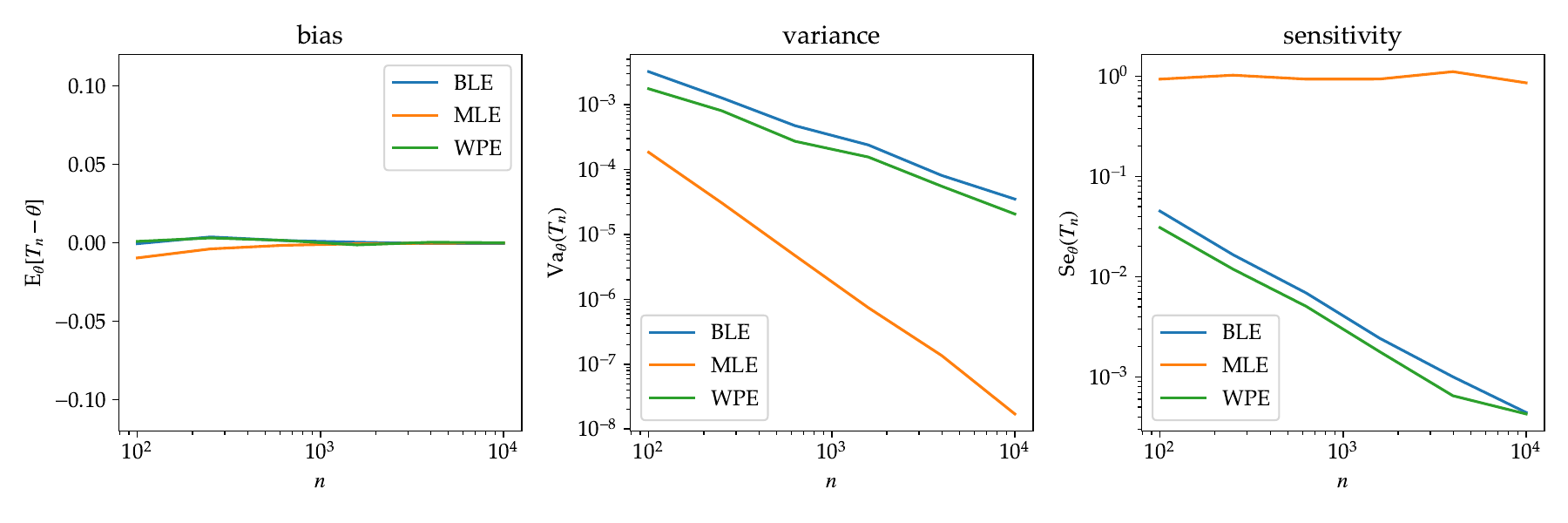}
        \caption{The variance and sensitivity for three estimators in the uniform scale family of Example~\ref{ex:Unif-scale}: the best linear estimator (BLE), the maximum likelihood estimator (MLE), and the Wasserstein projection estimator (WPE).
        We plot Monte Carlo estimates of the variance and sensitivity for fixed $\theta>0$ as $n$ increases.
        The MLE has variance of order $n^{-2}$ and sensitivity of constant order.
        The BLE has both variance and sensitivity of order $n^{-1}$.
        The WPE has both variance and sensitivity of order $n^{-1}$, and it achieves a smaller constant prefactor than the BLE for both quantities. See Example~\ref{ex:unif-scale-fam} for further details.}
        \label{fig:uniform}
    \end{figure}

We begin with two basic examples.
    
    \begin{example}[location family]\label{ex:loc-fam}
        Suppose we want to estimate an unknown parameter $\theta\in\Rbb^d$ from the samples $X_1,\ldots, X_n$, defined via $X_i:=\theta + \varepsilon_i$ for $1\le i \le n$, where $\varepsilon_1,\ldots, \varepsilon_n$ are i.i.d. samples from some mean-zero distribution $P_0$ on $\Rbb^d$ with a Lebesgue density, and satisfying $\int_{\Rbb^d}\|x\|^2\diff P_0(x)<\infty$.
        
        To determine an unbiased sensitivity-efficient estimator, let us cast this model in appropriate terms, by writing $P_{\theta}:=P_0(\,\cdot\,-\theta)$ for $\theta\in\Rbb^d$, so $p=k=d$ and $\Theta = \Rbb^d$.
        Then, note that the optimal transport maps between elements of this family are just translations $\transport_{P_{\theta_0}\to P_{\theta_1}}(x) = x + \theta_1-\theta_0$, so in particular we have $\Phi_{\theta}(x) = \idmat_d$.
        This shows that $\Pcal = \{P_{\theta}:\theta\in\Rbb^d\}$ is a transport family with natural parameterization $\chi(\theta) = \theta$ and potential $\phi:\Rbb^d\to\Rbb^d$ satisfying $D\phi(x) = \idmat_d$ for all $x\in\Rbb^d$; in other words, we may take $\phi(x) = x$.
        Therefore, we have  $J(\theta) = \idmat_d$ for all $\theta\in\Rbb^d$, and Theorem~\ref{thm:efficiency-W-fam} shows that the sample mean $T_n(X_1,\ldots, X_n):= \frac{1}{n}\sum_{i=1}^{n}X_i$ is an efficient unbiased estimator of $\theta$.
        (Of course, one can also check the efficiency of $T_n$ directly.)
    \end{example}

    \begin{example}[scale family]\label{ex:var-fam}
        Suppose we want to estimate an unknown parameter $\theta>0$ from samples $X_1,\ldots, X_n$ which are i.i.d. from $P_{\theta}$ defined via $P_{\theta}(A)=P_1(\theta^{-1}{A})$  for $A \subseteq \Rbb^d$ Borel, where $P_1$ is some distribution on $\Rbb^d$ with density with respect to the Lebesgue measure and which satisfies $\int_{\Rbb^d}\|x\|^2\diff P_1(x) = 1$.
    
        To determine a sensitivity-efficient estimator, we cast this model into the terms above by writing $p=k=1$ and $\Theta = (0,\infty)$.
        The optimal transport maps between these probability measures are $\transport_{P_{\theta_0}\to P_{\theta_1}}(x) = (\theta_1/\theta_0)x$; hence we can compute the transport linearization to be $\Phi_{\theta}(x) = x/\theta$.
        The dependence on $x$ suggests that we can take $\nabla \phi(x) = x$, i.e., $\phi(x) = \frac{1}{2}\|x\|^2$.
        Thus, we compute $\Lambda(\theta) = \Ebb_{\theta}[X^{\top} X] = \theta^2$ and this forces $\chi'(\theta) = \theta$.
        Summarizing, we have shown that $\Pcal = \{P_\theta: \theta > 0\}$ is a transport family with potential $\frac{1}{2}\|x\|^2$ and parameterization $\frac{1}{2}\theta^2$, so it follows that the Wasserstein information is $J(\theta) = 1$, and Theorem~\ref{thm:efficiency-W-fam} implies (after absorbing constant factors) that $T_n(X_1,\ldots, X_n):=\frac{1}{n}\sum_{i=1}^{n} \|X_i\|^2$ is an efficient unbiased estimator of $\chi(\theta) = \theta^2$.
        In other words, the natural parameter that can be efficiently estimated in an unbiased way is not the scale but the \textit{square} of the scale.
    \end{example}

    It is illustrative to compare both of the preceding examples with their classical counterparts.
    On the one hand, note that the regularity conditions required for variance-efficient estimation in both location families and scale families concern the existence of differentiable densities satisfying some integrability condition;
    on the other hand, the regularity conditions required for sensitivity-efficient estimation concern only the existence of a density with respect to the Lebesgue measure and the existence of moments.
    As in Remark~\ref{rem:DWS-v-DQM-exs}, neither set of condition implies the other.
    Also, we note that the preceding examples show, in fact, that the sample mean is simultaneously sensitivity-efficient for all $P_0$ (in location models), and  that the sample mean of the squared norms is simultaneously sensitivity-efficient for all $P_1$ (scale models). Contrast this with classical theory where the variance-efficient estimator in a location (or scale) family depends on $P_0$ (or $P_1$).    

    Next we address some more sophisticated examples.

    \begin{example}[Pareto family]\label{ex:pareto}
        Suppose we want to estimate a parameter $\theta>0$ from samples $X_1,\ldots, X_n$ which are i.i.d. from a distribution $P_{\theta}$ with density $p_{\theta}(x) = \theta x^{-(\theta+1)}$ for $x\ge 1$; in other words, we aim to estimate the \textit{index} $\theta$ from samples in a Pareto family.

        In order to determine an efficient unbiased estimator for some transformation of $\theta$, we set $p=k=d=1$ and take $\Theta =(2,\infty)$.
        Since the distribution function and quantile function of $P_{\theta}$ are respectively $F_{\theta}(x) = 1-x^{-\theta}$ and $F_{\theta}^{-1}(u) = (1-u)^{-1/\theta}$, we can compute the transport linearization to be $\Phi_{\theta}(x) = -x\log x/\theta$.
        This suggests that $\Pcal$ is a transport family with potential $\phi:\Rbb\to\Rbb$ satisfying $\phi'(x) = x\log x$; in other words, we may take $\phi(x) := (1/2)x^2\log x - (1/4)x^2$.
        Then some calculus yields $\chi'(\theta)/\Lambda(\theta) = -1/\theta$ hence $\chi(\theta) = (\theta-2)^{-2}$.
        Therefore, Theorem~\ref{thm:efficiency-W-fam} implies that
        \begin{equation*}
            T_n(X_1,\ldots, X_n) := \frac{1}{n}\sum_{i=1}^{n}\left(\frac{1}{2}X_i^2\log X_i - \frac{1}{4}X_i^2+\frac{1}{4}\right)
        \end{equation*}
        is an unbiased sensitivity-efficient estimator of $\chi(\theta).$
    \end{example}

    \begin{example}[Gaussian mean-variance family]\label{ex:Gaussian}
        Consider a univariate Gaussian model in  which the mean $\mu\in\Rbb$ and variance $\sigma^2>0$ are both unknown, and where we observe i.i.d. samples $X_1,\ldots, X_n$ from $\Ncal(\mu,\sigma^2)$.

        To write this in the usual form, we set $\theta = (\mu,\sigma^2)$ and $P_{\theta}:=\Ncal(\mu,\sigma^2)$ so that $\Theta \subseteq\Rbb^2$, meaning we have $p=k=2$ and $d=1$.
        Since the optimal transport maps are affine, we easily compute the transport linearization for this family to be 
        \begin{equation*}
            \Phi_\theta(x)=\begin{pmatrix} 1 & \frac{x-\mu}{2\sigma^2} \end{pmatrix}.
        \end{equation*}
        To see that this implies $\Pcal$ is a transport family, set $\phi(x) = (x,x^2)$ and $\chi(\mu,\sigma^2) = (\mu, \mu^2+\sigma^2)$.
        Then, we can compute the Wasserstein information matrix and the exactly sensitivity-efficient estimator to be
        \begin{equation*}
            J(\theta) = \begin{pmatrix}
                1 & 0\\
                0 & \dfrac{1}{4\sigma^2}
            \end{pmatrix}, \qquad\textnormal{and}\qquad T_n(X_1, \dots, X_n) = \begin{pmatrix}
                \frac{1}{n}\sum_{i=1}^n X_i \\\,\\ \frac{1}{n}\sum_{i=1}^n X_i^2 
            \end{pmatrix}.
        \end{equation*}
        In other words, the natural estimand is not the pair of mean and variance, but rather the pair of first and second moments; in fact, there is no unbiased sensitivity-efficient estimator of $\sigma^2$.
        (However, we will see in Example~\ref{ex:var-unknown-mean} that there is indeed an unbiased \textit{asymptotically} sensitivity-efficient estimator of $\sigma^2$.)
    \end{example}

    \begin{example}[bivariate Gaussian covariance estimation]\label{ex:multivar-Gaussian}
        Suppose $X_1,\ldots, X_n$ are i.i.d. from the distribution $P_{\theta}=\Ncal(0,\Sigma_{\theta})$ where
        \begin{equation*}
            \Sigma_{\theta}=\begin{pmatrix}
                \sigma^2 & \rho \\ \rho & \tau^2
            \end{pmatrix},
        \end{equation*}
        and that the goal is to estimate the vector $\theta=(\sigma^2,\tau^2,\rho)\in\Theta$, where $\Theta$ is the set of all parameters which give rise to a (strictly) positive-definite $2\times 2$ matrix above (i.e., $\sigma^2\tau^2>\rho^2$).

        We begin by computing the transport linearization for this model.
        To do this, recall from Lemma~\ref{lem:DWS-potential} that $\Phi_{\theta}$ is characterized as the gradient of the potential of paths in the model $\Pcal$ determined by lines through $\Theta$.
        Because of Example~\ref{ex:Gaussian}, a natural Ansatz for this potential is a quadratic, and indeed one can directly compute its exact form by moment matching.
        That is, write $\Delta:=\sigma^2\tau^2-\rho^2>0$ and $D:=(\sigma^2+\tau^2)(\sigma^2\tau^2-\rho^2)$ for convenience, and define
        \begin{align*}
            A(\sigma^2) &:=\frac{1}{4D}\begin{pmatrix}
                \tau^4+\Delta & -\rho\tau^2 \\ -\rho\tau^2 & \rho^2
            \end{pmatrix} \\
            A(\tau^2)&:=\frac{1}{4D}\begin{pmatrix}
                \rho^2 & -\rho\sigma^2 \\ -\rho\sigma^2 & \sigma^4+\Delta
            \end{pmatrix} \\
            A(\rho)&:=\frac{1}{2D}\begin{pmatrix}
                -\rho\tau^2 & \sigma^2\tau^2 \\ \sigma^2\tau^2 & -\rho\sigma^2
            \end{pmatrix},
        \end{align*}
        Then, we define the following quadratic functions    
        \begin{align*}
            \phi_{\theta,1}\begin{pmatrix}x\\y\end{pmatrix} &:= \begin{pmatrix}x&y\end{pmatrix}A(\sigma^2)\begin{pmatrix}x\\y\end{pmatrix}-\frac{1}{4}=\frac{\tau^4+\Delta}{4D}x^2 + \frac{\rho^2}{4D}y^2 - \frac{\rho\tau^2}{2D}xy-\frac{1}{4} \\
            \phi_{\theta,2}\begin{pmatrix}x\\y\end{pmatrix} &:= \begin{pmatrix}x&y\end{pmatrix}A(\tau^2)\begin{pmatrix}x\\y\end{pmatrix}-\frac{1}{4} = \frac{\rho^2}{4D}x^2 + \frac{\sigma^4+\Delta}{4D}y^2 - \frac{\rho\sigma^2}{2D}xy-\frac{1}{4} \\
            \phi_{\theta,3}\begin{pmatrix}x\\y\end{pmatrix} &:= \begin{pmatrix}x&y\end{pmatrix}A(\rho)\begin{pmatrix}x\\y\end{pmatrix}= -\frac{\rho\tau^2}{2D}x^2 -\frac{\rho\sigma^2}{2D}y^2 + \frac{\sigma^2\tau^2}{D}xy.
        \end{align*}
        Lastly, define $\phi:\Rbb^2\to\Rbb^3$ via
        \begin{equation*}
            \phi_{\theta}\begin{pmatrix}x\\y\end{pmatrix} = \begin{pmatrix}\phi_{\theta,1} \\ \phi_{\theta,2} \\ \phi_{\theta,3}\end{pmatrix}\begin{pmatrix}x\\y\end{pmatrix}.
        \end{equation*}
        Then, one can check directly that, for any $h\in\Rbb^3$, the path $\{P_{\theta+th}\}_{0\le t\le 1}$ solves the continuity equation with time-varying vector field $\{v_t\}_{0\le t\le 1}$ satisfying $v_0 = h_1\nabla\phi_{\theta,1}+h_2\nabla\phi_{\theta,2}+h_3\nabla\phi_{\theta,3}$.
        In particular, the transport linearization of this model is $\Phi_{\theta}=(D\phi_{\theta})^{\top}$.

        Therefore, by setting $\phi(x) = (x_1^2, x_2^2, x_1x_2)$ and $\chi(\theta) =\theta$ we can see that this model is a transport family.
        In particular, the Wasserstein information matrix is given by
        \begin{equation*}
            J(\theta) = \frac{1}{4D}\begin{pmatrix}
                \tau^4+\Delta & \rho^2 & -2\rho\tau^2 \\
                \rho^2 & \sigma^4 + \Delta & -2\rho\sigma^2 \\
                -2\rho\tau^2 & -2\rho\sigma^2 & 4\sigma^2\tau^2
            \end{pmatrix}.
        \end{equation*}
        (As a sanity check, note that, if $\sigma^2=\tau^2=1$, then the bottom-right entry equals $1/(2(1-\rho^2))$, which agrees with the Wasserstein information from the Gaussian correlation family that we will directly calculate in Example~\ref{ex:Gaussian-Corr} below.)
        Inverting this matrix yields
        \begin{equation*}
            (J(\theta))^{-1} = \begin{pmatrix}
                4\sigma^2 & 0 & 2\rho \\
                0 & 4\tau^2 & 2\rho \\
                2\rho & 2\rho & \sigma^2+\tau^2
            \end{pmatrix}
        \end{equation*}
        for all $\theta$.
        This also implies that the natural estimator
        \begin{equation*}
            T_n := \frac{1}{n}\sum_{i=1}^{n}X_iX_i^{\top},
        \end{equation*}
        is exactly sensitivity-efficient.
        (In fact, in the model with an unknown mean vector $\mu\in\Rbb^2$, one can show that
        \begin{equation*}
            T_n := \frac{1}{n}\sum_{i=1}^{n}X_iX_i^{\top}-\left(\frac{1}{n}\sum_{i=1}^{n}X_i\right)\left(\frac{1}{n}\sum_{i=1}^{n}X_i\right)^{\top}
        \end{equation*}
        is asymptotically sensitivity-efficient but not exactly sensitivity-efficient.)
    \end{example}

    \begin{example}[linear regression]\label{ex:regression}
        Suppose we want to estimate an unknown parameter $\theta\in \Theta = \Rbb^p$ from the observations $X_1,\ldots, X_n$, where $X_i:=W_i^{\top}\theta + \varepsilon_i$ and $W_1,\ldots, W_n$ are fixed and known covariates and $\varepsilon_1,\ldots, \varepsilon_n$ are i.i.d.~errors (unobserved) from some mean-zero distribution $P_0$ on $\Rbb$ with Lebesgue density and finite variance.
        In other words, we are interested in estimation of the coefficient vector in a fixed-design linear regression.
        Although this example is not precisely of the form previously discussed, a Wasserstein-Cram\'er-Rao lower bound still holds because of Remark~\ref{rem:indep-not-iid}:
        If we write $\bW$ for the usual $n\times p$ design matrix and assume that $\bW^{\top}\bW$ is invertible, then we have $\Cos_{\theta}(T_n) \succeq (\bW^{\top}\bW)^{-1}$ for any unbiased estimator $T_n:\Rbb^n\to\Rbb^p$ of $\theta$. 
        Moreover, writing $X = (X_1,\ldots, X_n)^{\top}$, it can be directly checked that the ordinary least-squares (OLS) estimator $T_n^{\OLS}(X_1,\ldots, X_n) = (\bW^{\top}\bW)^{-1}\bW^{\top}X$ is unbiased and sensitivity-efficient. 
    \end{example}

    Lastly, let us give an example of a concrete statistical model which is not a transport family.

    \begin{example}[Gaussian correlation family]\label{ex:Gaussian-Corr}
        Suppose that we want to estimate the correlation $\theta$ between two coordinates, where we observe i.i.d. samples $X_1,\ldots, X_n\in\Rbb^2$ which are drawn from the Gaussian distribution $\Ncal(0,\Sigma_{\theta})$; here, we have defined
        \begin{equation*}
            \Sigma_{\theta}:=\begin{pmatrix}
                1 & \theta \\
                \theta & 1
            \end{pmatrix}
        \end{equation*}
        for all $-1<\theta<1$.

        To cast this model in the form above, write $P_{\theta} := \Ncal(0,\Sigma_{\theta})$, so $p=k=1$ and $\Theta:=(-1,1)\subseteq\Rbb$, and $d=2$.
        By using the explicit formula for the optimal transport map between Gaussian distributions \cite{OlkinPukelsheim}, we can compute
        \begin{equation*}
            \Phi_{\theta}(x) = \frac{1}{2(1-\theta^2)}\begin{pmatrix}
                -\theta & 1 \\ 1 & -\theta
            \end{pmatrix}x,
        \end{equation*}
        and this shows that $\Pcal$ is not a transport family;
        if it were, then this would force $\nabla \phi(x) = x$ hence $\phi(x) =\frac{1}{2}\|x\|^2$, which cannot have expectation equal to $\theta$.
        In other words, there is no possibility to estimate $\theta$, or any sufficiently regular transformation thereof, in an unbiased and sensitivity-efficient way.
    \label{Example:Correlation}
    \end{example}
    
In the next proposition we show that a sufficiently regular subset of a transport family is also a transport family, which can be used to obtain new transport families from old ones.

\begin{proposition}
\label{prop:Subfamily}
Let $\Pcal = \{P_{\theta}: \theta\in\Theta\}\subseteq\Pcal_{2,\mathrm{ac}}(\Rbb^d)$ be a transport family with potential $\phi: \Rbb^d \rightarrow \Rbb^k$ and parameterization $\chi: \Theta \subseteq \mathbb{R}^p \rightarrow \mathbb{R}^k$.
Let $\Xi$ be an open subset of $\mathbb{R}^l$ with $l \le p$, and suppose that $h: \Xi \rightarrow \Theta$ is differentiable. Then the family $\tilde{\mathcal{P}}:= \{  P_{h(\xi)} : \xi \in \Xi\}$ is a transport family with potential $\phi$ and parameterization $\chi\circ h$.
\end{proposition}

\begin{proof}
From the definitions it is straightforward to see that if $\Phi_\theta$ is the transport linearization for the family $\mathcal{P}$ at $\theta$, then $\tilde{\Phi}_{\xi}(x) := \Phi_{h(\xi)}(x) Dh(\xi) $ is the transport linearization for the family $\tilde{\mathcal{P}}$ at $\xi$.
In turn, using \eqref{eqn:W-fam} we deduce that 
\[ \tilde{\Phi}_{\xi}(x) = (D\phi(x))^{\top}(\Lambda(h(\xi)))^{-1}D\chi(h(\xi)) Dh(\xi)  =  (D\phi(x))^{\top}(\Lambda(h(\xi)))^{-1}D(\chi\circ h)(\xi),\]
implying that $\tilde{\mathcal{P}}$ is also a transport family. 
\end{proof} 

    Through these examples we have seen that many statistical models of interest are indeed transport families.
    Roughly speaking, this is because many families of probability distributions are in fact \textit{defined} by the form of the optimal transport map between elements; a more general description of models of this form is given below.
    
\begin{example}[construction of transport families for $p=1$]\label{eqn:W-fam-from-ODE}
Let $\phi: \Rbb^d \rightarrow \Rbb$ be a differentiable function with Lipschitz gradient.
For each $x \in \Rbb^d$, let $\{u_{x}(t)\}_{t\ge 0}$ denote the solution to the ordinary differential equation (ODE):
\[  \begin{cases}  \frac{d}{dt} u_{x}(t) = \nabla \phi(u_x(t)), \quad t>0, \\ u_x(0)= x .    \end{cases}   \]
By the classical Picard-Lindelöf theory for ODEs (e.g., see \cite[Theorem 2.2]{Teschl2012}), the fact that $\nabla \phi$ is Lipschitz guarantees that this ODE has a unique solution for all time.
Now, for a given $\theta >0$, we define the \textit{flow map} $\Psi_\theta (x):= u_x(\theta)$, for $x \in \Rbb^d$, which can be verified to be bi-Lipschitz by, say, Gronwall's inequality and the observation that the inverse of $\Psi_\theta$ coincides with the flow map of $-\nabla \phi$ \cite[Box~4.1]{Santambrogio}.
 
Now let $P_0 \in \Pcal_{2,\mathrm{ac}}(\Rbb^d)$, and define $P_{\theta} := (\Psi_{\theta})_{\#} P_0$ for all $\theta>0$.
It is a standard fact in optimal transport theory (see e.g., \cite[Theorem~5.34]{VillaniOldAndNew}) that $\{ P_{\theta}\}_{\theta >0}$ solves the continuity equation $\partial_\theta P_\theta + \Div(P_\theta \nabla \phi) =0$.
Moreover, $\Psi_{\theta}$ being bi-Lipschitz implies that $P_{\theta}$ is absolutely continuous with respect to the Lebesgue measure for all $\theta>0$.
Finally, it follows from the proof of   \cite[Proposition~8.4.6]{AmbrosioGigliSavare} that the family $\{ P_{\theta}\}_{\theta>0}$ is DWS and that $\Phi_\theta = \nabla \phi$ for all $\theta>0$.
In particular, the family $\{ P_\theta \}_{\theta>0}$ is a transport family, with parameterization $\chi(\theta) = \theta$ and potential $\phi$; by Theorem~\ref{thm:efficiency-W-fam}, an unbiased sensitivity-efficient estimator of $\theta$ is $n^{-1}\sum_{i=1}^{n}\phi(X_i)$.
This family can be extended to $\theta<0$ by pushing forward $P_0$ by the flow map of $-\nabla \phi $.
\end{example}

\subsection{Asymptotic Sensitivity-Efficiency}\label{subsec:ex-asymp}

Next, we collect some examples about constructing asymptotically sensitivity-efficient estimators, which use a variety of methods from Section~\ref{sec:asymp}.

First, we give some examples that are based on applying the change-of-estimand principle (Remark~\ref{rem:param-equivar}) to some previous examples.

    \begin{example}[scale family, continued]
        In the setting of Example~\ref{ex:var-fam}, we have already seen that $n^{-1}\sum_{i=1}^{n}\|X_i\|^2$ is an exactly sensitivity-efficient estimator of $\theta^2$, and it is also consistent.
        Thus, Remark~\ref{rem:param-equivar} implies that $T_n = (n^{-1}\sum_{i=1}^{n}\|X_i\|^2)^{1/2}$ is a consistent and asymptotically sensitivity-efficient estimator of $\theta$.
    \end{example}

    \begin{example}[Gaussian variance, unknown mean]\label{ex:var-unknown-mean}
    Suppose that $X_1,\ldots, X_n$ are i.i.d. samples from the distribution $\Ncal(\mu,\sigma^2)$, where both $\mu\in\Rbb$ and $\sigma^2$ are unknown, and the goal is to estimate $\sigma^2$.
    In Example~\ref{ex:Gaussian} we found an (exactly) sensitivity-efficient estimator $T_n$ of $\theta = (\mu,\sigma^2+\mu^2)^{\top}$, which is also consistent.
    Thus by taking $\chi(\theta_1, \theta_2) = \theta_2-\theta_1^2$, it follows from Remark~\ref{rem:param-equivar} that the sample variance
    \begin{equation*}
        T_n = \frac{1}{n}\sum_{i=1}^{n}X_i^2 - \left(\frac{1}{n}\sum_{i=1}^{n}X_i\right)^2 = \frac{1}{n}\sum_{i=1}^{n}(X_i-\overline{X}_n)^2
    \end{equation*}
    is a consistent and asymptotically sensitivity-efficient estimator of $\sigma^2$.
    (Recall a similar result in the multivariate case from Example~\ref{ex:multivar-Gaussian}.)
    One can also obtain an unbiased, consistent, asymptotically sensitivity-efficient estimator $\tilde{T}_n$ by the usual correction, replacing the factor $n^{-1}$ by $(n-1)^{-1}$.
    \end{example}

    Second, we consider the notion of Wasserstein projection estimator (WPE), in which some general result will establish the desired efficiency.
    A basic example is the following, where the WPE is simple enough to handle directly.
 
        \begin{example}[location family, continued]\label{ex:Locationfamily2}
        We revisit Example \ref{ex:loc-fam}, and we note (say, by Brenier's theorem) that for any $\theta\in\Theta$ we have $\transport_{P_\theta \to \bar{P}_n} (x) =  \transport_{P_0 \to \bar{P}_n} (x - \theta)$.
        Consequently, we may compute $W_2^2(P_\theta, \bar{P}_n) =  \int_{\Rbb^d} \| \transport_{P_0 \to \bar{P}_n} (y) -y\|^2 \diff P_0(y)  - 2  \theta^{\top} \overline{X}_n + \|\theta\|^2$.
        It follows that $W_2^2(P_\theta, \bar{P}_n)$ is minimized when $\theta= \overline{X}_n$, hence the WPE for a location family is precisely the sample mean of $X_1,\ldots, X_n$.
        Since we showed in Example \ref{ex:loc-fam} that the sample mean is sensitivity-efficient, we conclude that the WPE is sensitivity-efficient in this case.
    \end{example}

    A more interesting example is a scale family, which we treat next.
    Before getting there we note that, when $d=1$, we can characterize the asymptotic distribution of the WPE under some additional assumptions; these assumptions can be relaxed by considering the more general conditions in \cite[Theorem 4.6]{DelBarrio2005}, but we do not pursue this here.
    We also note that similar results can be deduced from existing central limit theorems for $L$-statistics (e.g., \cite[Theorem~22.3]{vanDerVaart}).

    \begin{theorem}
    \label{thm:WPEVariance}
    Suppose $\Pcal=\{P_{\theta}:\theta\in\Theta\}\subseteq \Pcal_{2,\textnormal{ac}}(\Rbb)$ is DWS.
    Also suppose that $P_\theta$ has bounded support $[a_\theta,b_{\theta}]$ and that its density $p_\theta$ is differentiable and uniformly bounded below on $(a_{\theta}, b_{\theta})$, for all $\theta\in\Theta$.
    If $X_1,X_2,\ldots$ are i.i.d. samples from $P_{\theta^{\ast}}$ and $T_n^{\WPE}$ is a consistent estimator of $\theta^{\ast}\in\Theta$, then we have
    \[  \sqrt{n}(T_n^{\mathrm{WPE}}  - \theta^*) \rightarrow  \Ncal(0, \Sigma(\theta^{\ast})),\]
    in distribution, where
    \[  \Sigma(\theta^{\ast}) :=J(\theta^*)^{-1}\left(\int_0^1 \int_0^1 ( \min\{u, v \} - u v ) \frac{\nabla_\theta F_{\theta^{\ast}}^{-1}(u)}{p_{\theta^*}(F_{\theta^*}^{-1}(u))} \left(\frac{\nabla_\theta F_{\theta^{\ast}}^{-1}(v)}{p_{\theta^*}(F_{\theta^*}^{-1}(v))}\right)^{\top} \diff u \diff v\right) J(\theta^*)^{-1}. \]
    for all $\theta^{\ast}\in\Theta$.
    \end{theorem}

    \begin{example}[location family, continued]\label{ex:Locationfamily3}
    We have already seen in Example~\ref{ex:loc-fam} and Example~\ref{ex:Locationfamily2} that the sample mean is exactly (hence asymptotically) sensitivity-efficient in a location family.
    As a sanity check, it is useful to verify that the hypotheses of Theorem~\ref{thm:WPE-efficient} are indeed verified in this case.
    To see this, write $F_0^{-1}$ for the quantile function of $P_0$, and note that we have $F_{\theta}^{-1}(u) = F_0^{-1}(u) + \theta$.
    Thus, $\nabla_{\theta}F^{-1}_{\theta}(\cdot) = 1$ and $\nabla_{\theta}^2F^{-1}_{\theta}(\cdot) = 0$, which do not depend on $\theta$ at all.
    \end{example}
    
    \begin{example}[scale family, continued]\label{ex:scale-fam}
        We revisit Example~\ref{ex:var-fam}, where it was shown that $\theta^2$ is the only estimand (up to affine transformations) for which unbiased sensitivity-efficient estimation is possible.
        Presently, we use the WPE to find an asymptotically sensitivity-efficient estimator of $\theta$.
        Observe that the quantile function of $P_{\theta}$ is exactly $F_{\theta}^{-1}(u) = \theta F_{1}^{-1}(u)$ for all $\theta>0$ and $0<u<1$; hence we have $\nabla_{\theta}F_{\theta}^{-1}(\cdot) = F_1^{-1}(\cdot)$ and $\nabla_{\theta}^2F_{\theta}^{-1}(\cdot) = 0$, so the hypotheses of Theorem~\ref{thm:WPE-efficient} are verified.
        Moreover, we compute
        \[ W_2^2(P_{\theta},\bar P_n) =  \int_{0}^{1}\left|\bar F_n^{-1}(u) - \theta F_1^{-1}(u)\right|^2\diff u,  \]
        where $\bar F_n^{-1}$ is the quantile function of the empirical measure $\bar P_n$, and we note that this is a quadratic function in $\theta$, so the WPE is exactly
        \begin{equation*}
            T_n^{\WPE}(X_1,\ldots, X_n) = \frac{\int_{0}^{1}\bar F_n^{-1}(u)F_1^{-1}(u)\diff u}{\int_{0}^{1}\left| F_{1}^{-1}(u)\right|^2\diff u}.
        \end{equation*}
        Note that this quantity has a concrete interpretation as the coefficient of the regression of $\tilde X_n$ on $X$, where the joint distribution of $(\tilde{X}_n,X)$ is the optimal coupling of $\bar P_n$ and $P_1$.
        In summary, the estimator $T_n^{\WPE}$ defined above is asymptotically sensitivity-efficient in any scale family.
    \end{example}

\begin{example}[uniform scale family, revisited]\label{ex:unif-scale-fam}
Consider the special case of Example~\ref{ex:scale-fam} of the uniform scale family, in which case we have $F_1^{-1}(u) = u$ for $0\le u\le 1$.
In particular, the WPE can be written explicitly as
        \begin{equation*}
            T_n^{\WPE} = 3\sum_{i=1}^{n}\int_{\frac{i-1}{n}}^{\frac{i}{n}}X_{(i)}u\diff u = \frac{3}{2n^2}\sum_{i=1}^{n}(2i-1)X_{(i)},
        \end{equation*}
        where $X_{(1)},\ldots X_{(n)}$ denote the order statistics of $X_1,\ldots X_n$.
        By massaging $T_n$ in ways that do not alter its asymptotic sensitivity (e.g., multiplying by some deterministic factor $\alpha_n$ satisfying $\alpha_n\to 1$, or adding a scalar multiple of $n^{-1}\overline{X}_n$ which has sensitivity of order $O(n^{-2})$), one can derive some simplified estimators of interest; for example, $\tilde{T}_n := 6n^{-1}(2n+1)^{-1}\sum_{i=1}^{n}iX_{(i)}$ is unbiased, consistent, and asymptotically sensitivity-efficient.
        
        By Example~\ref{ex:scale-fam} and Theorem~\ref{thm:WPEVariance}, the asymptotic variance and sensitivity of $T_n^{\WPE}$ are
        \begin{equation*}
            \Sen_{\theta}(T_n^{\WPE}) \sim \frac{1}{n J(\theta)} = \frac{3}{n} \qquad \textnormal{and}\qquad\Var_{\theta}(T_n^{\WPE}) \sim \frac{\sigma^2(\theta)}{n} = \frac{\theta^2}{5n}.
        \end{equation*}   
        Since the variance and sensitivity of the BLE can be easily computed as $\Sen_{\theta}(T_n^{\BLE}) = 4n^{-1}$ and $\Var_{\theta}(T_n^{\BLE}) = n^{-1}\theta^2/3$, this rigorously proves the behavior we observed in Figure~\ref{fig:uniform} from the Introduction; the WPE strictly dominates the BLE with respect to both the variance and the sensitivity.
\end{example}

    Third, we consider constructing asymptotically sensitivity-efficient estimators via the one-step estimation procedure of Theorem~\ref{thm:OSE}, which is most useful when $d\ge 2$ and the model of interest is not a transport family.
    Most of the interesting examples will be submodels of the Gaussian covariance estimation model we introduced in Example~\ref{ex:multivar-Gaussian}, and indeed we will be able to reuse many of the calculations therein.
    These submodels relate to the full covariance model in the same way that curved exponential families relate to the corresponding full families; note that they involve $p<k$ so our Theorem~\ref{thm:efficiency-W-fam} does not apply.

    \begin{example}[Gaussian correlation family, continued]\label{Example:Correlation-OSE}
        Consider the special case of Example~\ref{ex:multivar-Gaussian} where we have $\sigma^2=\tau^2=1$ and suppose that the goal is to estimate $\rho\in (-1,1)=:\Theta$.
        From Example~\ref{ex:multivar-Gaussian} we can directly see that its Wasserstein information is $J(\rho) = \frac{1}{2}(1-\rho^2)^{-1}$, or this can be seen directly from Example~\ref{Example:Correlation}  where we can compute
        \begin{equation*}
            \phi_{\rho}\begin{pmatrix}
                x\\y
            \end{pmatrix} = \frac{-\frac{\rho}{2}(x^2+y^2)+xy}{2(1-\rho^2)}.
        \end{equation*}
        A natural estimator in this setting is $T_n=\frac{1}{n}\sum_{i=1}^{n}X_iY_i$ which is not asymptotically sensitivity-efficient; its sensitivity is equivalent to $2/n$, but the lower bound has constant equal to $2(1-\theta^2)<2$.

        We can nonetheless construct a one-step estimator (OSE) to achieve the lower bound asymptotically.
        To do this, use $T_n$ above as a pilot estimator and some explicit calculations lead to
        \begin{equation*}
            T_n^{\OSE} = \left(\frac{1}{n}\sum_{i=1}^{n}X_iY_i\right)\left(2-\frac{1}{2n}\sum_{i=1}^{n}(X_i^2+Y_i^2)\right).
        \end{equation*}
        Note that $T_n^{\OSE}$ simply has the form of multiplying the pilot estimator $T_n$ by a multiplicative factor that tends to 1 as $n\to\infty$ but which nonetheless introduces some cancellations that reduce the sensitivity.
    \end{example}

    \begin{example}[bivariate Gaussian covariance estimation with known cross-covariance]\label{ex:Multivar-Gaussian-fixed-corr}
        Consider the special case of Example~\ref{ex:multivar-Gaussian} where we have fixed cross-covariance $\rho=1$ (note that this is not the correlation), and suppose that the goal is to estimate $(\sigma^2,\tau^2)\in\Theta\subseteq \Rbb^2$ where $\Theta:=\{(\sigma^2,\tau^2): \sigma^2\tau^2>1\}$.
        By the calculations in Example~\ref{ex:multivar-Gaussian}, the Wasserstein information matrix in this model is
    \begin{equation*}
        J(\theta) = \begin{pmatrix}
            \tau^2(\sigma^2+\tau^2)-1 & 1 \\ 1 & \sigma^2(\sigma^2+\tau^2)-1
        \end{pmatrix}
    \end{equation*}
    hence
    \begin{equation*}
        (J(\theta))^{-1} = \frac{1}{(\sigma^2+\tau^2)(\sigma^2\tau^2-1)}\begin{pmatrix}
            \sigma^2(\sigma^2+\tau^2)-1 & -1 \\ -1 & \tau^2(\sigma^2+\tau^2)-1
        \end{pmatrix}.
    \end{equation*}
    One can directly calculate that the naive estimator
    \begin{equation*}
        T_n=\frac{1}{n}\sum_{i=1}^n\begin{pmatrix}
            X_i^2\\ Y_i^2
        \end{pmatrix}
    \end{equation*}
    is not asymptotically sensitivity-efficient, but we can instead use a one-step estimation (OSE) strategy to obtain an asymptotically sensitivity-efficient modification of it.

    To construct the OSE, consider $T_n$ above as the pilot estimator, and note that Theorem~\ref{thm:OSE} and the calculations of Example~\ref{ex:multivar-Gaussian} combine to show that each component of the OSE is a quadratic form with a constant intercept;
    more explicitly, we have
     \begin{equation*}
           T_n^{\OSE}= \frac{1}{n}\sum_{i=1}^{n}\begin{pmatrix}
               X_i^2 - \hat{c} \\ Y_i^2 - \hat{c}
           \end{pmatrix}
       \end{equation*}
       where
       \begin{equation*}
           \hat{c} := \frac{\frac{1}{n}\sum_{i=1}^{n}(X_iY_i-1)}{\frac{2}{n}\sum_{i=1}^{n}(X_i^2+Y_i^2)((\frac{1}{n}\sum_{i=1}^{n}X_i^2)(\frac{1}{n}\sum_{i=1}^{n}Y_i^2)-1)}
       \end{equation*}
        for all $n$.
        As before, note the appealing form of this estimator; it is just the pilot estimator $T_n$ but with a vanishing additive correction $\hat{c}$ in both components.
    \end{example}

\section{Geometric Background}\label{app:geom-background}

\subsection{Hellinger and Fisher-Rao Geometries}\label{subsec:hellinger}

    \begin{figure}[t]
        \centering
        \begin{minipage}{0.4\linewidth}
        \begin{tikzpicture}[scale=0.5]
		
		\draw[thin,stealth-stealth] (-0.8,-0.8) -- (-0.8,3.5);
		\draw[thin,stealth-stealth] (-1.5,0) -- (9.2,0);
		\node[color=black] at (10,0) {$\mathbb{R}^m$};
		
		\draw[very thick, color=blue, domain=0:8, smooth,stealth-stealth] plot (\x,{0.5+1.5*exp(-0.4*(\x-2)^2)+2*exp(-2*(\x-6)^2)+0.005*\x^2*(8-\x)^2});
		\node[blue] at (7.5,2.5) {$\mu_t$};
		

		\draw[very thick, color=red,-stealth] (1,0) -- ++(0,-0.5);
		\draw[very thick, color=red,-stealth] (2,0) -- ++(0,-0.5);
		\draw[very thick, color=red,-stealth] (8,0) -- ++(0,-0.7);
		
		\draw[very thick, color=red,-stealth] (-0,0) -- ++(0,0.4);
		\draw[very thick, color=red,-stealth] (3,0) -- ++(0,0.8);
		\draw[very thick, color=red,-stealth] (4,0) -- ++(0,1.8);
		\draw[very thick, color=red,-stealth] (5,0) -- ++(0,0.8);
		\draw[very thick, color=red,-stealth] (6,0) -- ++(0,0.6);
		\draw[very thick, color=red,-stealth] (7,0) -- ++(0,0.4);
		
		\node[color=red] at (6,-0.5) {$\xi_t$};
		
    \end{tikzpicture}~\\~\\
    
    \begin{tikzpicture}[scale=0.5]
        
		\draw[thin,stealth-stealth] (-0.8,-0.8) -- (-0.8,3.5);
		\draw[thin,stealth-stealth] (-1.5,0) -- (9.2,0);
		\node[color=black] at (10,0) {$\mathbb{R}^m$};
		
		\draw[very thick, color=blue, domain=0:8, smooth,<->] plot (\x,{0.5+1.5*exp(-0.4*(\x-2)^2)+2*exp(-2*(\x-6)^2)+0.005*\x^2*(8-\x)^2});
		\node[blue] at (7.5,2.5) {$\mu_t$};
		
		\draw[very thick, color=red,-stealth] (0,0) -- ++(-0.5,0);
		\draw[very thick, color=red,-stealth] (0.2,0) -- ++(0.5,0);
		\draw[very thick, color=red,-stealth] (0.9,0) -- ++(1,0);
		\draw[very thick, color=red,-stealth] (2.1,0) -- ++(1,0);
		\draw[very thick, color=red,-stealth] (3.3,0) -- ++(0.5,0);
		\draw[very thick, color=red,-stealth] (4,0) -- ++(0.5,0);
		\draw[very thick, color=red,-stealth] (5.2,0) -- ++(-0.5,0);
		\draw[very thick, color=red,-stealth] (6.1,0) -- ++(-0.7,0);
		\draw[very thick, color=red,-stealth] (7.3,0) -- ++(-1,0);
		\draw[very thick, color=red,-stealth] (8,0) -- ++(-0.5,0);
		
		\node[color=red] at (6,-0.5) {$v_t$};
		
	\end{tikzpicture} 
        \end{minipage}
        \begin{minipage}{0.4\linewidth}
        \begin{tikzpicture}[ scale=1.25]			


			\coordinate (A) at (6, 0.1);
			\coordinate (B) at (2.7, -1.1);
			\coordinate (C) at (1, 0.85);
			\coordinate (D) at (4.6, 2.2);
			\fill[thin, opacity=0.5, color=myblue] (A) to[out=170,in=40] (B) to[out=120,in=350] (C) to[out=50,in=170] (D) to[out=330,in=120] (A);
			
			\draw[very thick, color=mypurple, domain=2.25:4.25, smooth,stealth-stealth,cm={cos(-130) ,-sin(-130) ,sin(-130) ,cos(-130) ,(5.6125,-1.5)}] plot (\x,{0.25*(\x-3.5)^2});

			\draw[very thick, color=red,-stealth] (3.5, 1) -- ++(0.35, -0.45);
			
			\filldraw[color=blue] (3.5, 1) circle (2pt);
			\node[blue, right] at (2.8, 1.1) {$\mu_t$};

			\node[red, right] at (3.7, 1.) {$\xi_t$ or $v_t$};
			
		\end{tikzpicture}
        \end{minipage}
        \caption{A probability measure and its infinitesimal flow in the Fisher-Rao and Wasserstein geometries.
        For a probability measure $\mu_t$, the Fisher-Rao flow is determined by a mean-zero scalar field $\xi_t:\Rbb^m\to\Rbb$ (top left) and the Wasserstein flow is determined by a suitable vector field $v_t:\Rbb^m\to\Rbb^m$ (bottom left).
        In either case, the scalar and vector fields are the tangent vectors for the path $\{\mu_t\}_{t}$ in the space of probability measures (right).}
        \label{fig:FR-v-W}
    \end{figure}

To facilitate the comparison throughout the main body of the paper between the theory for sensitivity developed in this paper and the classical theory for variance, it will be beneficial to present some basic background on the Hellinger and Fisher-Rao geometries, emphasizing their formal Riemannian structures; we direct the reader to \cite{Background_Hellinger} for further information.
We encourage the reader to consult Figure~\ref{fig:FR-v-W} and Table~\ref{tab:geometry} while reading this subsection and the next.

\subsubsection{Hellinger Geometry}

We begin by introducing the Hellinger metric on the space of finite non-negative Borel measures (not necessarily probability measures), denoted $\mathcal{M}_{+}(\Rbb^m)$. For $\mu,\nu\in\mathcal M_{+}(\mathbb R^{m})$ and any $\sigma$-finite measure
$\lambda$ dominating both (i.e., $\mu,\nu\ll\lambda$), define the \emph{squared Hellinger distance} as
\[ H^2(\mu, \nu) := \int_{\Rbb^m} \left( \sqrt{\frac{\diff \mu}{\diff \lambda}} - \sqrt{\frac{\diff \nu}{\diff \lambda}} \right)^2 \diff \lambda.   \]
It can be easily checked that the value $H^2(\mu,\nu)$ is independent of the choice of $\lambda$.

Next, we introduce a dynamical formulation of the Hellinger distance. 
We say that a path $\{\mu_t\}_{0\le t\le 1}$ in $\mathcal{M}_+(\Rbb^m)$ and a time-varying scalar field $\{\xi_t\}_{0\le t\le 1}$ \textit{satisfy the reaction equation} if for all $f\in C_c^{\infty}(\Rbb^m)$ we have
\begin{equation}
    \frac{\diff}{\diff t}\int_{\Rbb^m}f(x)\diff \mu_t(x) = \int_{\Rbb^m}f(x)\xi_t(x)\diff\mu_t(x).
\end{equation}
In other words, $(\{\mu_t\}_{0\le t\le 1},\{\xi_t\}_{0\le t\le 1})$ is a solution, in the weak sense, of the ordinary differential equation (ODE)
\begin{equation}\label{eqn:ReactEq}
    \partial_t\mu_t = \xi_t\mu_t.
\end{equation}
Intuitively speaking, a pair $(\{\mu_t\}_{0\le t\le 1},\{\xi_t\}_{0\le t\le 1})$ satisfying the reaction equation means that, at each location $x\in\Rbb^m$ and each time $0\le t\le 1$, the change in the log-density $\log \mu_t(x)$ equals $\xi_t(x)$.
It turns out that the squared Hellinger distance between $\mu,\nu \in \mathcal{M}_+(\Rbb^m)$ can be written in the following form (see \cite{Background_Hellinger} and references therein):
\begin{equation}
   H^2(\mu, \nu) = \frac{1}{4}\min \left\{ \int_0^1 \lVert \xi_t \rVert^2_{L^2_{\mu_t}(\Rbb^m;\Rbb)} \diff t :\begin{matrix}
        (\{\mu_t\}_{0\le t\le 1},\{\xi_t\}_{0\le t\le 1}) \textnormal{ solves}\\
         \eqref{eqn:ReactEq}\textnormal{, and }\mu_0=\mu, \mu_1 = \nu
    \end{matrix}\right\}.
   \label{eqn:HellingerRiemannian}
\end{equation}
Plainly,~\eqref{eqn:HellingerRiemannian} states that the value $H^2$ equals the smallest possible ``size'' of a time-varying scalar field which continually re-weights $\mu$ to $\nu$.

Next, we define the constant-speed geodesics in the metric space $(\mathcal{M}_{+}(\Rbb^m), H)$, which are closely related to minimizers of the dynamical formulation~\eqref{eqn:HellingerRiemannian}.
A direct computation reveals that, if $\lambda$ is a $\sigma$-finite measure with respect to which both $\mu,\nu$ are absolutely continuous, then the path $\gamma:[0,1]\to \mathcal{M}_{+}(\Rbb^m)$ defined via
\begin{equation}\label{eqn:HellingerGeodesic}
    \gamma(t):= \mu_t, \qquad\textnormal{where}\qquad\frac{\diff \mu_t}{\diff \lambda} = \left( (1-t) \sqrt{\frac{\diff \mu}{\diff \lambda}}  + t  \sqrt{\frac{\diff \nu}{\diff \lambda}}  \right)^2,  
\end{equation}
is a constant-speed geodesic in the space $\MM_+(\Rbb^m)$, and in fact a minimizer of \eqref{eqn:HellingerRiemannian}; see \cite[Example~2.3]{Background_Hellinger}.

Finally, we discuss how these considerations allow us to endow $(\mathcal{M}_{+}(\Rbb^m), H)$ with a formal Riemannian structure.
By comparing the considerations in Section~\ref{sec:prelim} with the dynamical formulation~\eqref{eqn:HellingerRiemannian}, we should take the tangent space at each $\mu\in\Mcal_{+}(\Rbb^m)$ to be a subspace of $L^2_\mu(\Rbb^m;\Rbb)$; since \eqref{eqn:HellingerGeodesic} shows that any function in $L^2_{\mu}(\Rbb^m;\Rbb)$ is the derivative of some constant-speed geodesic, we must in fact take
 \begin{equation*}
        \Tan_{\mu}^{H}(\Mcal_{+}(\Rbb^m))=L^2_{\mu}(\Rbb^m;\Rbb).
    \end{equation*}
    We also define the logarithmic map as the initial speed of each constant-speed geodesic, which by \eqref{eqn:HellingerGeodesic} is the map $\mathrm{Log}_{\mu}^{H}:\Pcal_2(\Rbb^m)\to\Tan_{\mu}^{H}(\Pcal(\Rbb^m))$ given by
    \begin{equation}
    \label{eqn:LogMapHellinger}
        \mathrm{Log}^{H}_{\mu}(\nu):= \frac{\sqrt{\frac{\diff \nu}{\diff \lambda}}}{\sqrt{\frac{\diff \mu}{\diff \lambda}}} -1
    \end{equation}
    where $\lambda$ is a common reference measure for $\mu$ and $\nu$; if $\nu$ is absolutely continuous with respect to $\mu$, then we also have $\mathrm{Log}^{H}_{\mu}(\nu) = \sqrt{\diff \nu/\diff \mu}-1$.
    Here we use the superscript ``$H$'' for ``Hellinger'' in order to distinguish these concepts from the analogous concepts in the Wasserstein geometry.

\subsubsection{Fisher-Rao Geometry}

We next describe the Fisher-Rao geometry, which can be viewed as the geometry induced on the space of probability measures $\Pcal(\Rbb^m)$ by the Hellinger geometry on the space of finite nonnegative measures $\mathcal{M}_{+}(\Rbb^m)$.
We emphasize that the Fisher-Rao geometry is not simply the metric space resulting from restricting the metric from $\mathcal{M}_{+}(\Rbb^m)$ to $\Pcal(\Rbb^m)$; rather, it is the metric space resulting from restricting the Riemannian structure from $\mathcal{M}_{+}(\Rbb^m)$ to $\Pcal(\Rbb^m)$.

The construction is most easily understood via a dynamical formulation.
That is, for $\mu,\nu\in\Pcal(\Rbb^m)$, consider solutions $(\{\mu_t\}_{0\le t\le 1},\{\xi_t\}_{0\le t\le 1})$ to the reaction equation \eqref{eqn:ReactEq} with the property that $\mu_t$ is a probability measure for all $0\le t\le 1$;
it can be shown that this holds if and only if $\int_{\Rbb^m}\xi_t\diff \mu_t = 0$ for all $0\le t\le 1$.
In other words, we may define the \textit{squared Fisher-Rao distance} between $\mu,\nu\in\Pcal(\Rbb^m)$ as
\begin{equation}
    \mathrm{FR}^2(\mu, \nu):= \frac{1}{4}\min \left\{ \int_0^1 \lVert \xi_t \rVert^2_{L^2_{\mu_t}(\Rbb^m;\Rbb)} \diff t :\begin{matrix}
        (\{\mu_t\}_{0\le t\le 1},\{\xi_t\}_{0\le t\le 1}) \textnormal{ solves } \eqref{eqn:ReactEq}, \\
        \mu_0=\mu, \mu_1 = \nu, \textnormal{ and for } 0 \le t\le 1 \\
        \textnormal{we have}\int_{\Rbb^m}\xi_t\diff \mu_t = 0
    \end{matrix}\right\}.
   \label{eqn:FRRiemannian}
\end{equation}
It turns out that this can be written explicitly in terms of the Hellinger distance, as
\begin{equation}\label{eqn:FR-from-Helliner}
    \mathrm{FR}(\mu, \nu) = \cos^{-1}\left(1-\frac{1}{2}H^2(\mu,\nu)\right)
\end{equation}
for all $\mu,\nu\in \Pcal(\Rbb^m)$.
Geometrically speaking,~\eqref{eqn:FR-from-Helliner} allows us to interpret the Fisher-Rao and Hellinger distances between probability measures $\mu,\nu\in\Pcal(\Rbb^m)$ in terms of the unit sphere $S_\lambda:=\{f\in L^2_{\lambda}(\Rbb^m;\Rbb): \|f\|_{L^2_{\lambda}(\Rbb^m;\Rbb)} = 1\}$, where $\lambda$ is any $\sigma$-finite measure with respect to which both $\mu$ and $\nu$ are absolutely continuous: the Hellinger distance is the length of the \textit{chord} connecting $\sqrt{\diff \mu/\diff \lambda}$ to $\sqrt{\diff \nu/\diff \lambda}$ and the Fisher-Rao distance is the length of the \textit{arc} connecting $\sqrt{\diff \mu/\diff \lambda}$ to $\sqrt{\diff \nu/\diff \lambda}$.

The metric space $(\Pcal(\Rbb^m),\textnormal{FR})$ also admits a formal Riemannian structure, using the same interpretations as before.
The only difference between the Hellinger and Fisher-Rao cases is that the tangent space $\Tan_{\mu}^{\textnormal{FR}}(\Pcal(\Rbb^m))$ at $\mu$ requires its elements to be centered, meaning
\begin{equation*}
    \Tan_{\mu}^{\textnormal{FR}}(\Pcal(\Rbb^m))=\left\{\xi\in L^2_{\mu}(\Rbb^m;\Rbb): \int_{\Rbb^m}\xi\diff \mu = 0\right\}.
\end{equation*}
There also exists a notion of the logarithmic map for the Fisher-Rao geometry, but it will not be needed in this work so we direct the reader to \cite{Background_Hellinger} for further information.

\begin{remark}[Hellinger vs. Fisher-Rao geometries]\label{rem:Hellinger-FR}
For a given collection of probability measures $\mathcal{P}\subseteq\mathcal{P}(\Rbb^m)$, we highlight that both the Hellinger and Fisher-Rao geometries will lead to the same measure of instability of unbiased estimators.
The reason is that $\mathcal{P}(\Rbb^m)$ is itself an embedded submanifold of $ \MM_+(\Rbb^m) $, which is reflected in the fact that from~\eqref{eqn:FR-from-Helliner} we can see $\mathrm{FR}(\mu, \nu)\sim H(\mu,\nu)$ as $\mu\to \nu$.
From this perspective, it is equally correct to say that the classical theory for variance is connected to either the Hellinger or Fisher-Rao geometries; in this paper we find it analytically more convenient to primarily use the Hellinger geometry. 
\end{remark}

The reader may wonder how to reconcile Remark~\ref{rem:Hellinger-FR} with the fact that the classical Cram\'er-Rao theory is also connected to the KL divergence (e.g., the MLE can be interpreted as a sort of KL projection) and the $\chi^2$ divergence (e.g., in the Hammersley-Chapman-Robbins lower bound).
The answer is that all sufficiently smooth $f$-divergences (including Hellinger distance, KL divergence, and $\chi^2$ divergence, but not total variation distance) in regular enough parametric families are infinitesimally equivalent; see \cite[Theorem~7.20]{PolyanskiyWu}.

\subsection{Wasserstein Geometry}\label{subsec:wasserstein}

    Next, we review the fundamental concepts of optimal transport and the geometry of the Wasserstein space, as these will be used throughout the paper.
    We direct the reader to \cite{VillaniOldAndNew} and \cite{AmbrosioGigliSavare} for further detail on this background material, and, as before, we encourage the reader to consult Table~\ref{tab:geometry} and Figure~\ref{fig:FR-v-W} in order to facilitate the comparison between this subsection and the previous.

    First, we introduce the optimal transport problem and the Wasserstein metric.
    For $\mu,\nu\in\Pcal_2(\Rbb^m)$, the \textit{Monge optimal transport problem} is the optimization problem
    \begin{equation}\label{eqn:Monge-OT}
        \begin{cases}
            \textnormal{minimize} & \int_{\Rbb^m}\|x-T(x)\|^2\diff \mu(x)\\
            \textnormal{over} &T:\Rbb^m\to\Rbb^m \\
            \textnormal{with} &T_{\#}\mu = \nu,
        \end{cases}
    \end{equation}
    where $T_{\#}\mu$ is the \textit{pushforward} of $\mu$ by $T$, that is the probability measure on $\Rbb^m$ defined via $(T_{\#}\mu)(A) = \mu(T^{-1}(A))$ for all Borel sets $A\subseteq\Rbb^m$;
    closely related is the \textit{Kantorovich optimal transport problem} given by
    \begin{equation}\label{eqn:Kant-OT}
        \begin{cases}
            \textnormal{minimize} & \int_{\Rbb^m\times \Rbb^m}\|x-y\|^2\diff \pi(x,y)\\
            \textnormal{over} &\pi\in\Pi(\mu,\nu) \\
        \end{cases}
    \end{equation}
    where $\Pi(\mu,\nu)$ denotes the space of all \textit{couplings} of $\mu$ and $\nu$, meaning the space of all Borel probability measures on $\Rbb^m\times\Rbb^m$ whose first and second marginals are $\mu$ and $\nu$, respectively.
    We write $W_2^2(\mu,\nu)$ for the minimal value of problem~\eqref{eqn:Kant-OT}, which is called the \textit{squared 2-Wasserstein distance} between $\mu$ and $\nu$;
    the function $W_2:\Pcal_2(\Rbb^m)\times \Pcal_2(\Rbb^m)\to [0,\infty)$ is called the \textit{2-Wasserstein metric} or simply the \textit{Wasserstein metric}.
    It is known that, if we have $\mu\in\Pcal_{2,\mathrm{ac}}(\Rbb^m)$, then problems~\eqref{eqn:Monge-OT} and ~\eqref{eqn:Kant-OT} have the same value, and problem~\eqref{eqn:Monge-OT} admits a unique solution (up to null sets with respect to $\mu$) which we denote by $\transport_{\mu\to\nu}:\Rbb^m\to\Rbb^m$ and call the optimal transport map from $\mu$ to $\nu$.

    Next, we introduce the dynamical formulation of the optimal transport problem~\eqref{eqn:Monge-OT}, which is detailed in \cite[Chapter~8]{AmbrosioGigliSavare}.
    As we discussed in~\eqref{eqn:ContEq-body}, we say that a pair $(\{\mu_t\}_{0 \le t\le 1},\{v_t\}_{0 \le t\le 1})$ \textit{solves the continuity equation} if $\int_{0}^{1}\int_{\Rbb^m}\|v_t\|^2\diff \mu_t\diff t<\infty$ and if, for every infinitely differentiable compactly-supported function $f:\Rbb^m\to\Rbb$, the function $t\mapsto\int_{\Rbb^m}f(x)\diff \mu_t(x)$ is absolutely continuous with derivative almost everywhere given by
    \begin{equation}\label{eqn:ContEq}
        \frac{\diff}{\diff t}\int_{\Rbb^m}f(x)\diff \mu_t(x)=\int_{\Rbb^m}(\nabla f(x))^{\top}v_t(x)\diff \mu_t(x).
    \end{equation}
    Intuitively speaking, a pair $(\{\mu_t\}_{0 \le t\le 1},\{v_t\}_{0 \le t\le 1})$ satisfying the continuity equation means, at each location $x\in\Rbb^m$ and each time $0\le t\le 1$, the decrease in probability $\mu_t(x)$ is equal to the outgoing probability under the flow of the vector $v_t(x)$.
    It turns out that we can write the squared 2-Wasserstein distance between $\mu$ and $\nu$ in terms of solutions to the continuity equation, via
    {\small \begin{equation}\label{eqn:Benamou-Brenier}
        W_2^2(\mu,\nu) = \min\left\{\int_{0}^{1}\|v_t\|^2_{L^2_{\mu_t}(\Rbb^m;\Rbb^m)}\diff t: (\{\mu_t\}_{0\le t\le 1},\{v_t\}_{0\le t\le 1}) \textnormal{ solves \eqref{eqn:ContEq}, } \mu_0 = \mu \textnormal{, } \mu_1 = \nu \right\}
    \end{equation}}
    which is referred to as the \textit{Benamou-Brenier formula}.
    Plainly,~\eqref{eqn:Benamou-Brenier} states that the value $W_2^2$ equals the smallest possible ``size'' of a time-varying vector field which continually pushes $\mu$ to $\nu$.

    Next, we define the constant-speed geodesics in the metric space $(\Pcal_2(\Rbb^m), W_2)$, which are closely related to minimizers of the dynamical formulation~\eqref{eqn:Benamou-Brenier}.
    It turns out that, if $\mu\in\Pcal_{2,\mathrm{ac}}(\Rbb^m)$ and $\nu\in\Pcal_2(\Rbb^m)$ are arbitrary, then the constant-speed geodesic $\gamma:[0,1]\to \Pcal_2(\Rbb^m)$ from $\mu$ to $\nu$ is exactly
    \begin{equation}\label{eqn:WassGeo}
        \gamma(t):=((1-t)\id + t\,\transport_{\mu\to\nu})_{\#}\mu \qquad \qquad \mbox{ for all } \quad 0\le t\le 1.
    \end{equation}
    In ``Lagrangian'' terms, this means that a particle
$x \in \Rbb^m$ moves along the straight line $x \mapsto (1-t)x+t\transport_{\mu\to\nu} (x)$ at constant speed.

    Lastly, we put together these notions in order to discuss the formal Riemannian structure of the metric space $(\Pcal_2(\Rbb^m),W_2)$, often referred to as the \textit{Otto calculus}.
    By comparing the dynamical formulation \eqref{eqn:Benamou-Brenier} with the considerations of Section~\ref{sec:prelim}, the tangent space $\Tan_{\mu}^{W}(\Pcal_2(\Rbb^m))$ at $\mu$ should be a subset of $L^2_{\mu}(\Rbb^m;\Rbb^m)$.
    Since the tangent space at $\mu$ should be equal to (the closure of the span of) the space of derivatives of constant-speed geodesics emanating from $\mu$,~\eqref{eqn:WassGeo} implies that we should take
    \begin{equation*}
        \Tan_{\mu}^{W}(\Pcal_2(\Rbb^m)):= \overline{\textnormal{span}\{\transport_{\mu\to\nu} - \id: \nu\in\Pcal_2(\Rbb^m)\}}
    \end{equation*}
    where the closure is taken in the Hilbert space $L^2_{\mu}(\Rbb^m;\Rbb^m)$.
    Finally, the logarithmic map $\mathrm{Log}_{\mu}^{W}:\Pcal_2(\Rbb^m)\to\Tan_{\mu}^{W}(\Pcal_2(\Rbb^m))$ is defined via
    \begin{equation}\label{eqn:LogMapWass}
        \mathrm{Log}_{\mu}^{W}(\nu) := \transport_{\mu\to\nu}-\id.
    \end{equation}
    As before, we use superscripts to distinguish between geometric concepts in the Hellinger and Wasserstein geometries.

    As a technical note, let us mention that throughout the paper we will need to apply the continuity equation \eqref{eqn:ContEq} to test functions that are not compactly supported and infinitely differentiable; we will often take the test function to be our \textit{estimator} of interest, which we will only restrict to satisfy, say, some local Lipschitz and mild integrability conditions.
    Our next goal is to extend the class of test functions in a suitable way.

    To do this, let us recall the definition of absolutely continuous functions. 

\begin{definition}[absolute continuity]
For $p \geq 1$ and $a\leq b $, we say that a function $\gamma: [a, b] \rightarrow \Rbb^m$ belongs to $\mathrm{AC}^p(a,b; \Rbb^m)$ if there is $h \in L^p(a,b)$ such that
\[ | \gamma(t)  -\gamma(s)   | \leq \int_s^t h(r) \diff r\]
\label{def:AC}
for all $a\leq s \leq t \leq b$.
\end{definition}

As the following result shows, absolutely continuous functions are precisely the functions for which the fundamental theorem of calculus holds. (See \cite[Theorem 3.20]{Leoni} for a proof.)

\begin{theorem}
For $p \geq 1$ and $a\leq b $, a function $\gamma$ belongs to $\mathrm{AC}^p(a,b; \Rbb^m)$ if and only if the following conditions hold:
\begin{enumerate}
    \item[(i)] $\gamma$ is continuous in $[a,b]$,
    \item[(ii)] $\gamma$ is differentiable for Lebesgue a.e. $t \in [a,b]$ and $\dot{\gamma}(t) \in L^p(a,b; \mathbb{R}^m)$, and
    \item[(iii)] $\gamma(t) = \gamma(s) + \int_s^t \dot{\gamma}(r) \diff r$, for all $a \leq s\leq t \leq b$.  
\end{enumerate}
\end{theorem}

We will utilize the following result.
Recall that by the \textit{distributional derivative} of any function $g: [a,b] \to \Rbb$ we mean a distribution $g':C^{\infty}_{c}(\Rbb)\to\Rbb$ satisfying $\langle g',\phi\rangle = - \int_a^b g(t) \varphi'(t) \diff t$ for all test functions $\phi\in C_c^{\infty}(\Rbb)$; to understand this definition, note that, if $g$ is $C^1(\Rbb)$ and $\phi$ is supported on $(a,b)$, then the equality $\int_{a}^{b}g'(t)\phi(t)\diff t = - \int_a^b g(t) \varphi'(t) \diff t$ follows from integration by parts, so the distributional derivative coincides with the pointwise derivative. 

\begin{proposition}\label{prop:ContinuityEqnLocallyLipschitz}
Fix $a,b\in \Rbb$ with $a<b$ and suppose that $(\{\mu_t\}_{a\le t\le b}, \{v_t\}_{a\le t\le b})$ is a solution to the continuity equation in $\Rbb^m$.
Suppose, in addition, that for every $t \in [a,b]$ we have $\mu_t \in \mathcal{P}_{2,\mathrm{ac}} (\Rbb^m)$.
Then for every locally Lipschitz function $f: \Rbb^m \rightarrow \Rbb$ satisfying
\begin{enumerate}
\item[(i)] $\int_{\Rbb^m} | f (x)|\diff \mu_t(x) < \infty$, for all $t \in (a,b)$, and $\int_a^b \int_{\Rbb^m} | f (x)|^2\diff \mu_t(x) \diff t < \infty$
\item[(ii)] $ \int_a^b \int_{\Rbb^m} |\nabla f (x)|^2 \diff \mu_t(x) \diff t < \infty$.
\end{enumerate}
the function 
\[ t \in [a,b] \mapsto \int_{\Rbb^m} f(x) \diff \mu_t(x) \]
belongs to $\mathrm{AC}^1(a,b; \mathbb{R})$, and has distributional derivative given by
\begin{equation}\label{eqn:ContEqLocLip}
   \frac{\diff}{\diff t} \int_{\Rbb^m} f(x) \diff \mu_t(x) = \int_{\Rbb^m} (\nabla f(x))^{\top}    v_t(x) \diff \mu_t(x).
\end{equation}
\end{proposition}

\begin{proof} We split the proof into two steps.

\textbf{Step 1:} In this first step we make stronger assumptions and suppose that $\varphi$ is globally Lipschitz and bounded. We start by recalling \cite[Theorem 8.2.1]{AmbrosioGigliSavare}, which implies that there exists a probability measure $\pmb{\eta}$ over $\Rbb^m \times  \Gamma_{(a,b)}$ (here, $\Gamma_{(a,b)}$ represents the space of continuous maps from $[a,b]$ into $\Rbb^m$) satisfying the following two properties:
\begin{enumerate}
\item[(i)]  $\pmb \eta$ is concentrated on the set of pairs $(x, \gamma) \in \Rbb^m \times \Gamma_{(a,b)}$ such that $\gamma \in \mathrm{AC}^2(a,b;\Rbb^m)$, and $\gamma$ is a solution to the ODE $\dot{\gamma}(t) = v_t(\gamma(t))$ for almost every $t \in (a,b)$, with initial condition $\gamma(a) =x$;
\item[(ii)] $\mu_t = (e_{t})_{\#} \pmb \eta$, for every $t \in [a,b]$, where the map $e_t : \Rbb^m \times \Gamma_{(a,b)} \rightarrow \Rbb^m$ is given by $e_t:(x , \gamma) \mapsto \gamma(t)$. 
\end{enumerate}
From property (i) of $\pmb \eta$ and \cite[Theorem 3.55]{Leoni} we can deduce that, for $\pmb{\eta}$-a.e.~$(x, \gamma)$, the function $t \in [a,b] \mapsto \varphi(\gamma(t))$ is absolutely continuous. However, we will need to use our additional assumptions on the measures $\mu_t$ to conclude that for $\pmb \eta$-a.e.~$(x, \gamma)$ we can apply the chain rule to rewrite the derivative of $\varphi(\gamma(t))$ and ultimately get to the representation \eqref{eqn:ContEqLocLip}.

Toward this aim, let us consider the set 
\[ A := \{ y \in \Rbb^m : \varphi \textrm{ is not differentiable at } y  \},\]
and define the function $g: \Gamma_{(a,b)} \rightarrow \Rbb$ given by
\[ g(\gamma):= \int_a^b \mathbf{1}_{A}(\gamma(t)) \diff t.\]
By Tonelli's theorem and property (ii) of $\pmb \eta$ we have
\begin{align*}
   \int_{\Rbb^m \times \Gamma_{(a,b)}} g(\gamma) \diff  \pmb {\eta}(x, \gamma) & = \int_a^b \int_{\Rbb^m \times \Gamma_{(a,b)}}  \mathbf{1}_{A}(\gamma(t)) \diff \pmb \eta(x, \gamma) \diff t 
   \\& = \int_a^b \mu_t(A) \diff t 
   \\& =0,
\end{align*}
where the last line follows from Rademacher's theorem and the fact that $\mu_t$ was assumed to be absolutely continuous with respect to the Lebesgue measure. The above implies that $g(\gamma) =0$ for $\pmb \eta$-a.e.~$(x, \gamma)$. In particular, for $\pmb \eta$-a.e.~$(x,\gamma)$ we have: $\gamma(t) \not \in A$
for a.e.~$t \in [a,b]$. Combining the previous fact with property (i) of $\pmb{\eta}$, the fact that absolutely continuous functions are a.e. differentiable, and the classical chain rule, we deduce that for $\pmb{\eta}$-a.e.~$(x, \gamma)$ we have
\[ \frac{\diff }{\diff t} \varphi(\gamma(t)) =  (\nabla \varphi(\gamma(t)))^{\top} \dot{\gamma}(t), \quad \text{a.e. } t \in [a,b]. \]
Using the above we can now conclude that, for all $s,t \in [a,b]$, 
\begin{align*}
\int_{\Rbb^m} \varphi(x) \diff \mu_t - \int_{\Rbb^m} \varphi(x) \diff \mu_s  & =  \int_{\Rbb^m \times \Gamma_{(a,b)}} (\varphi(\gamma(t)) - \varphi(\gamma(s))) \diff \pmb{\eta}(x, \gamma)
\\& =  \int_{\Rbb^m \times \Gamma_{(a,b)}} \int_s^t \frac{\diff}{\diff r} \varphi\circ \gamma(r) \diff r \diff \pmb{\eta}(x, \gamma)
\\& =\int_{\Rbb^m \times \Gamma_{(a,b)}} \int_s^t (\nabla \varphi(\gamma(r)))^{\top} \dot{\gamma}(r)  \diff r  \diff \pmb{\eta}(x, \gamma) 
\\& = \int_{\Rbb^m \times \Gamma_{(a,b)}} \int_s^t (\nabla \varphi(\gamma(r)))^{\top} v_r(\gamma(r))  \diff r  \diff \pmb{\eta}(x, \gamma) 
\\& = \int_s^t \left(\int_{\Rbb^m} (\nabla \varphi(x))^{\top}v_r(x) \diff \mu_r(x) \right) \diff r,
\end{align*}
where in the last line we used Fubini's theorem and property (ii) of $\pmb{\eta}$. This implies \eqref{eqn:ContEqLocLip}.


\textbf{Step 2:} We now consider a general locally Lipschitz $\varphi$ as in the statement of this proposition and approximate $\varphi$ by bounded and globally Lipschitz functions $\{ \varphi_n \}_{n \in \mathbb{N}}$, for which we can apply the result in Step 1. Indeed, for every $n \in \mathbb{N}$ we consider a smooth cutoff function $\xi_n: \Rbb^m \rightarrow [0,1]$ satisfying $\xi_n \equiv 1$ in $B_n(0)$ (the ball of radius $n$ centered at $0 \in \Rbb^m$), $\xi_n(x) \equiv 0$ outside of $B_{n+1}(0)$, and $\lVert \nabla \xi_n \rVert_\infty \leq C$ for some constant independent of $n$. We define $\varphi_n(x) := \xi_n(x)\varphi(x)$, $x \in \Rbb^m$. Note that $\varphi_n$ is globally Lipschitz and bounded for every $n \in \mathbb{N} $.

Let us now consider arbitrary $s, t \in [a,b]$. Then, for every $n \in \mathbb{N}$, we have
\[ \int_{\Rbb^m} \varphi_n(x) \diff \mu_t(x) -  \int_{\Rbb^m} \varphi_n(x) \diff \mu_s(x) = \int_s^t \int_{\Rbb^m } (\nabla \varphi_n(x))^{\top} v_r(x) \diff \mu_r(x) \diff r,  \]
thanks to Step 1. Using the assumptions $(i)$ and $(ii)$ on $\varphi$ and the dominated convergence theorem, we can take the limit as $n \rightarrow \infty$ on both sides of the above expression and conclude that
\[ \int_{\Rbb^m} \varphi(x) \diff \mu_t(x) -  \int_{\Rbb^m} \varphi(x) \diff \mu_s(x) = \int_s^t \int_{\Rbb^m } (\nabla \varphi(x))^{\top} v_r(x) \diff \mu_r(x) \diff r.  \]
This establishes the desired result.    
\end{proof}

\begin{remark}
We emphasize that a key assumption in the proof of Proposition \ref{prop:ContinuityEqnLocallyLipschitz} is that the measures $\mu_t$ are absolutely continuous with respect to the Lebesgue measure. Indeed, thanks to this assumption we were able to guarantee that for almost all trajectories $\gamma$ in the probabilistic representation of the continuity equation we could apply the chain rule $\frac{\diff}{\diff t} \varphi(\gamma(t)) = (\nabla \varphi(\gamma(t)))^{\top} \dot{\gamma}(t)$ for almost every $t \in [a,b]$, even if $\varphi$ was only assumed to be Lipschitz. Note that, in general, the chain rule for the composition of a Lipschitz function and an absolutely continuous path $\gamma$ does not always hold. As an example of this, consider the function $\varphi: \Rbb^2 \rightarrow \Rbb$ given by $\varphi(x_1, x_2)= \max \{ x_1, x_2 \}$ and let $\gamma(t)=(t,t)$. Since the range of the map $\gamma$ is contained in the set where $\varphi$ is not differentiable, we cannot make sense of the expression $(\nabla\varphi (\gamma(t)))^{\top} \dot{\gamma}(t)$, while, of course, the function $\varphi\circ \gamma(t)$ is differentiable and thus we can make sense of $\frac{\diff}{\diff t} \varphi(\gamma(t))$. 
\label{rem:ACMeasures}
\end{remark} 

\begin{remark}[differentiability Lebesgue almost everywhere]\label{rem:DistrDeriv}
The conclusion of Proposition~\ref{prop:ContinuityEqnLocallyLipschitz} states that the function $t\mapsto \int_{\Rbb^m} f(x) \diff \mu_t(x)$ has a distributional derivative which is a function, hence we may interpret \eqref{eqn:ContEqLocLip} as an equality for Lebesgue almost every $t \in [0,1]$.
Observe that this differs from the case of Lemma~\ref{lem:DWS-potential} which gives equality for all $0\le t\le 1$.
In general, we cannot expect this function to be differentiable for all $0\le t\le 1$ without making additional assumptions. 
\end{remark}

\begin{remark}[absolute continuity assumption]\label{rem:ContEqAbsCont}
Because Proposition~\ref{prop:ContinuityEqnLocallyLipschitz} requires $\{\mu_t\}_{0\le t\le 1}$ to possess densities with respect to the Lebesgue measure, we restricted our attention in the main body to statistical models $\Pcal = \{P_{\theta}:\theta\in\Theta\}$ satisfying $\Pcal\subseteq\Pcal_{2,\mathrm{ac}}(\Rbb^d).$
(Also, see Remark~\ref{rem:AC-sens} and Remark~\ref{rem:ACMeasures}.)
\end{remark}

\section{Proofs of Results}\label{app:proofs}

This section contains the proofs of the main results from the body of the paper, as well as some auxiliary results that are needed along the way.

\subsection{Proofs from Section~\ref{subsec:sensitivity}}\label{app:proofs-sensitivity}

First we have the following regarding the convergence of $\varepsilon$-sensitivity to sensitivity.

\begin{proof}[Proof of Proposition~\ref{prop:eps-sen-cvg}]
    By assumption, $\nabla^2T_n$ is continuous and compactly supported; thus there exists a constant $C>0$ such that Taylor's theorem yields
    \begin{equation*}
        \bigg|T_n( X_1', \dots, X_n') - T_n(X_1, \dots, X_n) - (\nabla T_n(X_1, \dots, X_n))^{\top}(\xi_1,\ldots, \xi_n)\bigg| \le C\sum_{i=1}^{n} \|\xi_i\|^2
    \end{equation*}
    almost surely.
    By squaring this and taking the expectation, we have
    {\small \begin{equation}\label{eqn:sen-eps-n}
        \Ebb_{P}\left[\bigg|T_n( X_1', \dots, X_n') - T_n(X_1, \dots, X_n) - (\nabla T_n(X_1, \dots, X_n))^{\top}(\xi_1,\ldots, \xi_n)\bigg|^2\right] \le Cdn (n+2)\varepsilon^4
    \end{equation}}
    for all $\varepsilon>0$, so rescaling and taking the limit as $\varepsilon\to 0$ reveals that the random variables
    \begin{equation*}
        \frac{T_n( X_1', \dots, X_n') - T_n(X_1, \dots, X_n)}{\varepsilon} \qquad \textnormal{ and }\qquad\frac{(\nabla T_n(X_1, \dots, X_n))^{\top}(\xi_1,\ldots, \xi_n)}{\varepsilon}
    \end{equation*}
    possess the same limit in $L^2_{P}(\Omega;\Rbb)$.
    Thus, it suffices to show that the $L^2_{P}(\Omega;\Rbb)$ norm of the latter converges to $\Sen_{P}(T_n)$, which can be done as follows:
    \begin{align*}
        & \Ebb_{P}\left[\left|\frac{(\nabla T_n(X_1, \dots, X_n))^{\top}(\xi_1,\ldots, \xi_n)}{\varepsilon}\right|^2\right] \\
        & = \Ebb_{P}\left[\Ebb\left[\left|\frac{(\nabla T_n(X_1, \dots, X_n))^{\top}(\xi_1,\ldots, \xi_n)}{\varepsilon}\right|^2\bigg|X_1,\ldots, X_n\right]\right] \\
        &= \Ebb_{P}\left[\|\nabla T_n(X_1,\ldots, X_n)\|^2\right],
    \end{align*}
    where we used the tower property and the fact $\Ebb[\|v^{\top}(\xi_1,\ldots,\xi_n)\|^2] = \varepsilon^2 \|v\|^2$ for all $v\in\Rbb^{dn}$.
    \end{proof}

    Next we turn to the proof of the lower bound in Theorem~\ref{thm:W-Chap-Robb-univar}.
    To do this, we require the following auxiliary result:
    
    \begin{lemma}\label{lem:CE-tensorize}
        If $\{P_t\}_{0\le t\le 1}$ satisfies the continuity equation with time-varying vector field $\{V_t\}_{0\le t\le 1}$, then the product measures $\{P_t^{\otimes n}\}_{0\le t\le 1}$ satisfy the continuity equation in $\Rbb^{dn}$ with the time-varying vector field $\{V_t^{\otimes n}\}_{0\le t\le 1}$ defined by
    \begin{equation}\label{eqn:prod-CE}
        V_t^{\otimes n}(x_1,\ldots, x_n) := \begin{pmatrix}
            V_t(x_1),\ldots, V_t(x_n)
        \end{pmatrix}^{\top}
    \end{equation}
    for all $x_1,\ldots, x_n\in\Rbb^d$ and $0\le t\le 1$.
    \end{lemma}

    \begin{proof}
        by Nachbin's theorem (see \cite{Lesmes1974} and \cite{Prolla1976} and the references therein) and \cite[Remark~8.1.1]{AmbrosioGigliSavare}, 
    it suffices to consider test functions $f \in C^1_c(\Rbb^{dn})$ of the form 
    \[ f(x_1, \dots, x_n) = f_1(x_1) \cdots f_n(x_n),\] 
    for functions $f_1, \dots, f_n \in C_c^1(\Rbb^d)$. Note that for such functions we have
    \begin{align*}
        &\frac{\diff}{ \diff t} \int_{\Rbb^{dn}} f(x_1, \dots, x_n) \diff P_t^{\otimes n}(x_1, \dots, x_n) \\
        & = \frac{\diff}{\diff t} \prod_{i=1}^n \int_{\Rbb^d} f_i(x_i) \diff P_t(x_i) 
        \\& = \sum_{i=1}^n \int_{\Rbb^{dn}}  \left(\nabla_{x_i}f_i(x_i)\right)^{\top} V_t(x_i) \prod_{j \not = i} f_j(x_j) \diff P_t^{\otimes n }(x_1, \dots, x_n)   
        \\&=  \int_{\Rbb^{dn}} \left(\nabla f(x_1, \dots, x_n)\right)^{\top} V_t^{\otimes n} (x_1, \dots, x_n) \diff P_t^{\otimes n }(x_1, \dots, x_n),
    \end{align*}
    as needed.
    \end{proof}
    
    \begin{proof}[Proof of Theorem~\ref{thm:W-Chap-Robb-univar}.]
    By the previous result, the pair $(\{P_t^{\otimes n}\}_{0\le t\le 1},\{V_t^{\otimes n}\}_{0\le t\le 1})$ solves the continuity equation in $(\Rbb^d)^n$.
    Also, observe that by assumptions $(i)$ and $(ii)$, the hypotheses of Proposition~\ref{prop:ContinuityEqnLocallyLipschitz} are satisfied, hence we can apply the continuity equation and Remark~\ref{rem:DistrDeriv} to $T_n$, and then use Cauchy-Schwarz inequality to get
        \begin{align*}
            \left|\frac{\diff}{\diff t}\Ebb_{P_t}[T_n(X_1,\ldots, X_n)]\right| &= \left|\Ebb_{P_t}\left[\left\langle V_t^{\otimes n}(X_1,\ldots, X_n),\nabla T_n(X_1,\ldots, X_n)\right\rangle\right]\right| \\
            &\le \sqrt{\Ebb_{P_t}\left[\left\|V_t^{\otimes n}(X_1,\ldots, X_n)\right\|^2\right]\Sen_{P_t}(T_n)} \\
            &= \sqrt{n\Ebb_{P_t}\left[\left\|V_t(X)\right\|^2\right]\Sen_{P_t}(T_n)}.
        \end{align*}
        Rearranging this gives the desired result.
    \end{proof}

    \subsection{Proofs from Section~\ref{sec:bound}}\label{app:proofs-bound}

    Next we turn to the proof of the Wasserstein-Cram\'er-Rao lower bound.  We begin with the following result which states that, for a DWS model $\Pcal$, the affine path $t\mapsto \theta+th$ traces a curve of probability measures $\{P_{\theta+th}\}_{0\le t\le 1}$ whose pointwise derivative we can characterize; this is not equivalent to the statement that $\{P_{\theta+th}\}_{0\le t\le 1}$ solving the continuity equation \eqref{eqn:ContEq-body} with time-varying vector field given by $V_t=\Phi_{\theta+th}h$, where $\Phi_{\theta}h(x):=\Phi_{\theta}(x)h$, since we have not assumed $\int_{0}^{1}\int_{\Rbb^d}\|\Phi_{\theta+th}(x)h\|^2\diff P_{\theta+th}(x)\diff t<\infty$.
 
    \begin{lemma}\label{lem:DWS-potential}
        If $\Pcal=\{P_{\theta}:\theta\in\Theta\}\subseteq\Pcal_{2,\textrm{ac}}(\Rbb^d)$ is DWS, then, for every $f\in C_c^{\infty}(\Rbb^d)$, $\theta\in\Theta$, and sufficiently small $h\in\Rbb^p$ the map $t\mapsto \int_{\Rbb^d} f(x)\,\diff P_{\theta+th}(x)$ is differentiable and satisfies        
        \begin{equation*}
            \frac{\diff}{\diff t}\int_{\Rbb^d}f(x)\diff P_{\theta+th}(x)=\int_{\Rbb^d}(\nabla f(x))^{\top}\Phi_{\theta+th}(x)h\diff P_{\theta+th}(x).
        \end{equation*}
        for all $0\le t\le 1$.
    \end{lemma}

\begin{proof}
        For $0\le s\le t\le 1$, use the mean value theorem to get that for each $x\in\Rbb^d$ there exists some point $z_{s,t,x}$ on the line segment from $x $ to $\transport_{P_{\theta+sh}\to P_{\theta+th}}(x)$ such that we have
        \begin{align*}
            &\left|\frac{\int_{\Rbb^d}f\diff P_{\theta+th} - \int_{\Rbb^d}f\diff P_{\theta+sh}}{t-s} - \int_{\Rbb^d}(\nabla f)^{\top}\Phi_{\theta+sh}(x)h\diff P_{\theta+sh}\right| \\
            &\le \int_{\Rbb^d}\left|\frac{f(\transport_{P_{\theta+sh}\to P_{\theta+th}}(x))-f(x)}{t-s}-(\nabla f(x))^{\top}\Phi_{\theta+sh}(x)h\right|\diff P_{\theta+sh}(x) \\
            &= \int_{\Rbb^d}\left|(\nabla f(z_{s,t,x}))^{\top}\left(\frac{\transport_{P_{\theta+sh}\to P_{\theta+th}}(x)-x}{t-s}\right)-(\nabla f(x))^{\top}\Phi_{\theta+sh}(x)h\right|\diff P_{\theta+sh}(x) \\
            &\le  \int_{\Rbb^d}\left|(\nabla f(z_{s,t,x})-\nabla f(x))^{\top}\left(\frac{\transport_{P_{\theta+sh}\to P_{\theta+th}}(x)-x}{t-s}\right)\right|\diff P_{\theta+sh}(x) \\
            &\qquad +\int_{\Rbb^d}\left|(\nabla f(x))^{\top}\left(\frac{\transport_{P_{\theta+sh}\to P_{\theta+th}}(x)-x}{t-s}-\Phi_{\theta+sh}(x)h\right)\right|\diff P_{\theta+sh}(x).
        \end{align*}
        Now we show that both terms on the right vanish as $t\to s$.
        For the second term, use Cauchy-Schwarz and Jensen's inequality to get
        \begin{align*}
            &\int_{\Rbb^d}\left|(\nabla f(x))^{\top}\left(\frac{\transport_{P_{\theta+sh}\to P_{\theta+th}}(x)-x}{t-s}-\Phi_{\theta+sh}(x)h\right)\right|\diff P_{\theta+sh}(x) \\
            &\le \sup_{z\in\Rbb^d}\|\nabla f(z)\|\sqrt{\int_{\Rbb^d}\left\|\frac{\transport_{P_{\theta+sh}\to P_{\theta+th}}(x)-x}{t-s}-\Phi_{\theta+sh}(x)h\right\|^2\diff P_{\theta+sh}(x)}
        \end{align*}
        which vanishes by the definition of DWS.
        For the first term, use Cauchy-Schwarz to obtain
        \begin{align*}
            &\int_{\Rbb^d}\left|(\nabla f(z_{s,t,x})-\nabla f(x))^{\top}\left(\frac{\transport_{P_{\theta+sh}\to P_{\theta+th}}(x)-x}{t-s}\right)\right|\diff P_{\theta+sh}(x) \\
            &\le \sqrt{\int_{\Rbb^d}\|\nabla f(z_{s,t,x})-\nabla f(x)\|^2\diff P_{\theta+sh}(x)\int_{\Rbb^d}\left\|\frac{\transport_{P_{\theta+sh}\to P_{\theta+th}}(x)-x}{t-s}\right\|^2\diff P_{\theta+sh}(x)}.
        \end{align*}
        The second factor is bounded by the definition of DWS, and the first factor can be further bounded by the mean value theorem and the fact that $\|z_{s,t,x}-x\|\le\|\transport_{P_{\theta+sh}\to P_{\theta+th}}(x)-x\|$, leading to
        \begin{align*}
            &\int_{\Rbb^d}\|\nabla f(z_{s,t,x})-\nabla f(x)\|^2\diff P_{\theta+sh}(x)\\
            &\le \sup_{z\in\Rbb^d}\|\nabla^2f(z)\|^2\int_{\Rbb^d}\|z_{s,t,x}-x\|^2\diff P_{\theta+sh}(x)\\
            &\le \sup_{z\in\Rbb^d}\|\nabla^2f(z)\|^2\int_{\Rbb^d}\|\transport_{P_{\theta+sh}\to P_{\theta+th}}(x)-x\|^2\diff P_{\theta+sh}(x)
        \end{align*}
        where $\|\nabla^2f(z)\|$ denotes the operator norm of the matrix $\nabla^2f(z)$.
        The right side vanishes by the definition of DWS, and this finishes the proof.
    \end{proof}

    Next we turn to the proof of the Wasserstein-Cram\'er-Rao lower bound.

    \begin{proof}[Proof of Theorem~\ref{thm:WCR}]
        It suffices to show that we have
        \begin{equation}\label{eqn:transport-fam-1}
            \lambda^{\top}\Cos_{\theta}(T_n)\lambda \ge \frac{1}{n} \lambda^{\top}D\chi(\theta)(J(\theta))^{-1}(D\chi(\theta))^{\top}\lambda
        \end{equation}
        for all $\theta\in\Theta$ and $\lambda\in\Rbb^k$.
        Note that the result is true trivially if $(D\chi(\theta))^{\top}\lambda=0$, so it suffices to assume $(D\chi(\theta))^{\top}\lambda\neq0$.
        To do this, set $h:=(J(\theta))^{-1}(D\chi(\theta))^{\top}\lambda\in\Rbb^p$, and note that we lose no generality by assuming $\theta+th\in\Theta$ for all $0\le t\le 1$; indeed, $\Theta$ is an open set and the validity of~\eqref{eqn:transport-fam-1} is invariant under rescaling $\lambda$.
        Now consider the path $\{P_{\theta+th}\}_{0\le t\le 1}$ of probability measures and set $V_t(x)=\Phi_{\theta+th}(x)h$, for which we can compute
        \begin{align*}
            \int_{\Rbb^d}\|V_t(x)\|^2\diff P_{\theta+th}(x) &= \int_{\Rbb^d}h^{\top}(\Phi_{\theta+th}(x))^{\top}\Phi_{\theta+th}(x)h\diff P_{\theta+th}(x) \\
            &= h^{\top}J(\theta+th)h \\
            &= \lambda^{\top}D\chi(\theta)(J(\theta))^{-1}J(\theta+th)(J(\theta))^{-1}(D\chi(\theta))^{\top}\lambda,
        \end{align*}
        and this is continuous by assumption (i) hence $t\mapsto \int_{\Rbb^d}\|V_t(x)\|^2\diff P_{\theta+th}(x)$ is locally integrable.
        Also, DWS implies $W_2$ continuity which further implies weak continuity.
        Combining these observations with Lemma~\ref{lem:DWS-potential} shows that $\{P_{\theta+th}\}_{0\le t\le 1}$ solves the continuity equation with the time-varying vector field $\{V_t\}_{0\le t\le 1}$.
        Thus, \cite[Theorem~8.3.1]{AmbrosioGigliSavare} implies that $\{P_{\theta+th}\}_{0\le t\le 1}$ is absolutely continuous with respect to $W_2$.
        Next we consider the function $\lambda^{\top}T_n:(\Rbb^d)^n\to\Rbb$, which is integrable by assumption, and which satisfies
        \begin{equation*}
            \int_{0}^{1}\Ebb_{P_{\theta+th}}[|\lambda^{\top}T_n(X_1,\ldots,X_n)|^2]\diff t \le \|\lambda\|^2\int_{0}^{1}\Ebb_{P_{\theta+th}}[\|T_n(X_1,\ldots,X_n)\|^2]\diff t < \infty
        \end{equation*}
        by assumption (iii).
        Also, the continuity of $\theta\mapsto\Cos_{\theta}(T_n)$ implies
        \begin{align*}
            \int_{0}^{1}\Ebb_{P_{\theta+th}}[|\nabla(\lambda^{\top}T_n(X_1,\ldots,X_n))|^2]\diff t &=\int_{0}^{1}\Ebb_{P_{\theta+th}}[|\lambda^{\top}DT_n(X_1,\ldots,X_n)|^2]\diff t \\
            &\le \lambda^{\top}\int_{0}^{1}\Cos_{\theta+th}(T_n)\diff t\,\lambda < \infty.
        \end{align*}
        Consequently, we may apply Theorem~\ref{thm:W-Chap-Robb-univar} to obtain
        \begin{equation}\label{eqn:WCR-proof-2}
            \Sen_{\theta+th}(\lambda^{\top}T_n) \ge \frac{1}{n}\cdot\frac{(\frac{\diff}{\diff t}\Ebb_{\theta+th}[\lambda^{\top}T_n(X_1,\ldots, X_n)])^2}{\Ebb_{\theta+th}\left[\|\Phi_{\theta+th}(X)h\|^2\right]}
        \end{equation}
        for Lebesgue a.e.~$0\le t\le 1$.
        Next, straightforward calculations yield
        \begin{align*}
            \Sen_{\theta+th}(\lambda^{\top}T_n) &= \Ebb_{\theta+th}\left[\|\nabla \lambda^{\top}T_n(X_1,\ldots, X_n)\|^2\right] \\
            &= \Ebb_{\theta+th}\left[\|(DT_n(X_1,\ldots, X_n))^{\top}\lambda\|^2\right] \\
            &= \Ebb_{\theta+th}\left[\lambda^{\top}DT_n(X_1,\ldots, X_n)(DT_n(X_1,\ldots, X_n))^{\top}\lambda\right] \\
            &= \lambda^{\top}\Cos_{\theta+th}(T_n)\lambda,
        \end{align*}
        as well as
        \begin{align*}
            \Ebb_{\theta+th}\left[\|\Phi_{\theta+th}(X)h\|^2\right]&= \Ebb_{\theta+th}\left[h^{\top}(\Phi_{\theta+th}(X))^{\top}\Phi_{\theta+th}(X) h\right] \\
            &= h^{\top}J(\theta+th)h
        \end{align*}
        and
        \begin{equation*}
            \frac{\diff}{\diff t}\Ebb_{\theta+th}[\lambda^{\top}T_n(X_1,\ldots, X_n)] = \lambda^{\top}D\chi(\theta+th)h,
        \end{equation*}
        so the inequality \eqref{eqn:WCR-proof-2} reads
        \begin{equation}\label{eqn:WCR-proof-3}
            \begin{split}
                \lambda^{\top}\Cos_{\theta+th}(T_n)\lambda &\ge \frac{1}{n}\cdot\frac{(\lambda^{\top}D\chi(\theta+th)h)^2}{h^{\top}J(\theta+th)h} \\
            &= \frac{(\lambda^{\top}D\chi(\theta+th)(J(\theta))^{-1}(D\chi(\theta))^{\top}\lambda)^2}{n\lambda^{\top}D\chi(\theta)(J(\theta))^{-1}J(\theta+th)(J(\theta))^{-1}(D\chi(\theta))^{\top}\lambda}
            \end{split}
        \end{equation}
        for Lebesgue a.e.~$0\le t\le 1$.
        Since a set of full Lebesgue measure is dense, we may take $t\to 0$ along a suitable subsequence and then apply the continuity from $(i)$ to get that the right side of~\eqref{eqn:WCR-proof-3} converges to
        \begin{equation*}
            \frac{(\lambda^{\top}D\chi(\theta)(J(\theta))^{-1}(D\chi(\theta))^{\top}\lambda)^2}{n\lambda^{\top}D\chi(\theta)(J(\theta))^{-1}J(\theta)(J(\theta))^{-1}(D\chi(\theta))^{\top}\lambda} = \frac{1}{n}\lambda^{\top}D\chi(\theta)(J(\theta))^{-1}(D\chi(\theta))^{\top}\lambda.
        \end{equation*}
        Similarly, we use the continuity from $(ii)$ to get that the left side of~\eqref{eqn:WCR-proof-3} converges to $\lambda^{\top}\Cos_{\theta}(T_n)\lambda$.
        This establishes~\eqref{eqn:transport-fam-1} and finishes the proof.
    \end{proof}

    \subsection{Proofs from Section~\ref{sec:exact}}\label{app:proofs-exact}

    Next we consider proofs of the results concerning exact sensitivity-efficiency in transport families.
    Along the way we will need the following result containing some basic calculations:
    
    \begin{lemma}\label{lem:W-fam-properties}
        Suppose $p=k$, that $\Pcal = \{P_{\theta}: \theta\in\Theta\}\subseteq\Pcal_{2,\textnormal{ac}}(\Rbb^d)$ for $\Theta\subseteq\Rbb^p$ is a transport family with potential $\phi: \Rbb^d \to \Rbb^p$ and continuously differentiable parameterization $\chi: \Theta \to \Rbb^p$, and set $\Lambda(\theta):=\Ebb_{\theta}[D\phi(X)(D\phi(X))^{\top}]$.
        Then:
        \begin{itemize}
            \item[(i)] $J(\theta) = (D\chi(\theta))^{\top}(\Lambda(\theta))^{-1}D\chi(\theta)$ for all $\theta \in \Theta$.
            \item[(ii)] If $J(\theta)$ and $D\chi(\theta)$ are invertible, then $\Lambda(\theta) = D\chi(\theta)(J(\theta))^{-1}(D\chi(\theta))^{\top}$.
            \item[(iii)] If $\Theta$ is connected, $\Ebb_{\theta}[\|\phi(X)\|]<\infty$ for all $\theta\in\Theta$, and the maps $\theta\mapsto \Ebb_{\theta}[\|\phi(X)\|^2]$, $\theta\mapsto \Ebb_{\theta}[\|D\phi(X)\|^2]$, and $\theta\mapsto J(\theta)$ are locally Lebesgue integrable, then there exists some $v\in\Rbb^k$ with $$\Ebb_{\theta}[\phi(X)] = \chi(\theta)+v$$ for Lebesgue a.e. $\theta \in \Theta$.
            If, in addition, $\theta \mapsto \Ebb_\theta[\phi(X)]$ is continuous, the equality holds for all $\theta \in \Theta$. 
        \end{itemize}
    \end{lemma}

     \begin{proof}
        Write $\Phi_{\theta}$ for the transport linearization of $\Pcal$.
        For property (i), we just use the given form of $\Phi_{\theta}(x)$ to compute
        \begin{align*}
            J(\theta) &= \Ebb_{\theta}[(\Phi_{\theta}(X))^{\top}\Phi_{\theta}(X)] \\
            &= \Ebb_{\theta}[(D\chi(\theta))^{\top}(\Lambda(\theta))^{-1}D\phi(X)(D\phi(X))^{\top}(\Lambda(\theta))^{-1}D\chi(\theta)] \\
            &= (D\chi(\theta))^{\top}(\Lambda(\theta))^{-1}\Ebb_{\theta}[D\phi(X)(D\phi(X))^{\top}](\Lambda(\theta))^{-1}D\chi(\theta) \\
            &= (D\chi(\theta))^{\top}(\Lambda(\theta))^{-1}D\chi(\theta).
        \end{align*}
        For property (ii), simply rearrange property (i) using the invertibility of $J(\theta)$ and $D\chi(\theta)$.
        
        For property (iii), define the function $F:\Theta\to\Rbb^k$ via $F(\theta):=\Ebb_{\theta}[\phi(X)]$. The assumptions allow us to combine Proposition \ref{prop:ContinuityEqnLocallyLipschitz} and \cite[Theorem 11.45]{Leoni} to conclude that $F$ is a Sobolev function and obtain a formula for its partial derivatives; in condensed notation, this yields
        \begin{equation}\label{eqn:W-fam-props-2}
            \begin{split}
                DF(\theta) &= \Ebb_{\theta}\left[D\phi(X)\Phi_{\theta}(X)\right] \\
                &= \Ebb_{\theta}[D\phi(X)(D\phi(X))^{\top}](\Lambda(\theta))^{-1}D\chi(\theta) \\
                &= D\chi(\theta),
            \end{split}
        \end{equation}
        for Lebesgue almost every $\theta\in \Theta$. Since $\chi$ was assumed to be continuously differentiable hence locally Lipschitz, it is also a Sobolev function. This implies that the function $F- \chi$ is a Sobolev function whose partial derivatives are almost everywhere equal to zero. Since $\Theta$ was assumed to be connected, this implies that $F(\theta)-\chi(\theta)$ is equal to a constant vector $v$ for Lebesgue almost every $\theta \in \Theta$.
        Lastly, note that $\chi$ is continuously differentiable hence continuous so if $F$ is also assumed continuous then $F-\chi$ being Lebesgue-a.e.~equal to a constant implies that $F-\chi$ equals a constant.
    \end{proof}

    \begin{remark}[ambiguity with additive constants]\label{rem:W-fam-shift}
        Since in a transport family $\Pcal=\{P_{\theta}:\theta\in\Theta\}$ the potential $\phi$ only appears through its Jacobian $D\phi$, there is always some ambiguity about additive constants; this is the reason for the vector $v\in\Rbb^k$ appearing in property $(iii)$ of Lemma~\ref{lem:W-fam-properties} above, and the reason for various arbitrary additive constants appearing throughout Section~\ref{sec:exact}.
    \end{remark}

    \begin{remark}
        The additional hypothesis in statement (iii) above, requiring the continuity of $\theta\mapsto \Ebb_{\theta}[\phi(X)]$, is easily established in many settings.
        For example, if $\|\phi(x)\|\lesssim \|x\|^2$ then the continuity follows by combining DWS with the characterization of the $W_2$ topology in terms of weak convergence plus convergence of moments.
        In some settings we do not have this growth bound (e.g., Example~\ref{ex:pareto}) but explicit calculations are available.
    \end{remark}

    \begin{proof}[Proof of Theorem~\ref{thm:efficiency-W-fam}]
        We begin with claim (i).
        Indeed, under these assumptions, let $T_n$ be given as above, and note that property (iii) of Lemma~\ref{lem:W-fam-properties} implies $T_n$ is an unbiased estimator of $\chi(\theta)+v$ for some $v\in\Rbb^k$.
        Next, we compute
        \begin{equation*}
            D_{x_i}T_n(X_1,\ldots, X_n) =\frac{1}{n}D \phi(X_i);
        \end{equation*}
        hence
        \begin{equation}\label{eqn:W-fam-1}
            \sum_{i=1}^{n}D_{x_i}T_n(X_1,\ldots, X_n)(D_{x_i}T_n(X_1,\ldots, X_n))^{\top}
            =\frac{1}{n^2}\sum_{i=1}^{n}D \phi(X_i)(D\phi(X_i))^{\top}.
        \end{equation}
        Then take the expectation of \eqref{eqn:W-fam-1} with respect to $P_{\theta}$  and apply property (ii) of Lemma~\ref{lem:W-fam-properties} to get
        \begin{equation*}
            \Cos_{\theta}(T_n) = \frac{1}{n}\Ebb_{\theta}[D\phi(X)(D\phi(X))^{\top}] = \frac{1}{n}\Lambda(\theta) = \frac{1}{n}D\chi(\theta)(J(\theta))^{-1}(D\chi(\theta))^{\top}.
        \end{equation*}
        This shows that $T_n$ achieves the lower bound, hence it is sensitivity-efficient as claimed.
 
        For claim (ii), suppose that $T_n$ is an unbiased sensitivity-efficient estimator of $\chi(\theta)$.
        Then fix $n=1$ and $\lambda\in\Rbb^k$, and note that in the proof of Theorem~\ref{thm:WCR}, we showed
        \begin{equation}\label{eqn:transport-fam-1-new}
            \lambda^{\top}\Ebb_{\theta}\left[D T_1(X)(D T_1(X))^{\top}\right]\lambda \ge \lambda^{\top}D\chi(\theta)(J(\theta))^{-1}(D\chi(\theta))^{\top}\lambda.
        \end{equation}
        If we define $\tilde{T}_n:=\lambda^{\top}T_n$, $\tilde{\chi}:=\lambda^{\top}\chi$, and $h:=(J(\theta))^{-1}(D\chi(\theta))^{\top}\lambda\in\Rbb^p$, then the above rearranges to
        \begin{align*}
            \sqrt{\Ebb_{\theta}\left[\|\nabla \tilde{T}_1(X)\|^2\right] \Ebb_{\theta}\left[\|\Phi_{\theta}(X)h\|^2\right]}\ge(\nabla\tilde{\chi}(\theta))^{\top}h = \Ebb_{\theta}\left[(\nabla \tilde{T}_1(X))^{\top}\Phi_{\theta}(X)h\right].
        \end{align*}
        Thus, equality in~\eqref{eqn:transport-fam-1-new} implies that equality holds in Cauchy-Schwarz, hence $\nabla\tilde{T}_1$ must be collinear with $\Phi_{\theta}h$, almost surely under $P_{\theta}$.
        Since this holds for every $\lambda\in\Rbb^k$, we conclude that there exists some $A(\theta)\in\Rbb^{k\times p}$ for each $\theta\in\Theta$ such that we have
        \begin{equation*}
            (DT_1(x))^{\top}A(\theta) = \Phi_{\theta}(x)
        \end{equation*}
        almost surely under $P_{\theta}$.
        Observe that $A(\theta)$ must be invertible, since $A(\theta)h = 0$ implies $h^{\top}J(\theta)h = h^{\top}(A(\theta))^{\top}\Sen_{\theta}(T_n)A(\theta)h = 0$ so $(J(\theta))^{-1}$ is not strictly positive-definite.        
        Now let $\phi:=T_1$, and set $\Lambda(\theta):=\Ebb_{\theta}[D\phi(X)(D\phi(X))^{\top}]$ for $\theta\in\Theta$.
        Next, we use the unbiasedness of $T_1$ to identify $A(\theta)$.
        Indeed, since $\Ebb_{\theta}[\phi(X)] = \chi(\theta)$ for all $\theta\in\Theta$, the same computation as in the proof of property $(iii)$ of Lemma~\ref{lem:W-fam-properties} yields
        \begin{equation}\label{eqn:W-fam-2}
            D\chi(\theta) = \Ebb_{\theta}\left[D\phi(X)\Phi_{\theta}(X)\right] = \Ebb_{\theta}\left[D\phi(X)(D\phi(X))^{\top}\right]A(\theta) = \Lambda(\theta)A(\theta)
        \end{equation}
        for all $\theta\in\Theta$, where the second equality uses the collinearity established above.
        Since $\Lambda(\theta)$ is invertible, \eqref{eqn:W-fam-2} is a linear equation in $A(\theta)$ with the unique solution
        \begin{equation*}
            A(\theta) = (\Lambda(\theta))^{-1}D\chi(\theta).
        \end{equation*}
        Therefore, we have shown
        \begin{equation*}
            \Phi_{\theta}(x) = (D\phi(x))^{\top}(\Lambda(\theta))^{-1}D\chi(\theta)
        \end{equation*}
        almost surely under $P_{\theta}$. Hence $\Pcal$ is a transport family with parameterization $\chi$.
        This finishes the proof.
    \end{proof}

    \subsection{Proofs from Section~\ref{sec:asymp}}\label{app:proofs-asymp}

    Lastly, we prove our results concerning the Wasserstein projection estimator (WPE), mainly its asymptotic sensitivity-efficiency under some assumptions.
    
    \begin{proof}[Proof of Lemma~\ref{DWS-1s}]
        Recall that we have
        \begin{equation*}
            \transport_{P_{\theta}\to P_{\theta + th}}(x) = F_{\theta+th}^{-1}(F_{\theta}(x)),
        \end{equation*}
        so we may compute
        \begin{align*}
            &\int_{\Rbb}|\transport_{P_{\theta}\to P_{\theta + th}}(x)-x - t\Phi_{\theta}(x)h|^2\diff P_{\theta}(x) \\
            &= \int_{\Rbb}|F_{\theta+th}^{-1}(F_{\theta}(x))-F_{\theta}^{-1}(F_{\theta}(x)) - t\Phi_{\theta}(F_{\theta}^{-1}(F_{\theta}(x)))h|^2\diff P_{\theta}(x)\\
            &= \int_{\Rbb}|F_{\theta+th}^{-1}(u)-F_{\theta}^{-1}(u) - t\Phi_{\theta}(F_{\theta}^{-1}(u))h|^2\diff u.
        \end{align*}
        This means that the transport linearization, if it exists, must be given by
        \begin{equation*}
            \Phi_{\theta}(F_{\theta}^{-1}(u)) = \left(\nabla_{\theta}F^{-1}_{\theta}(u)\right)^{\top}
        \end{equation*}
        for Lebesgue a.e. $0\le u\le 1$, hence
        \begin{equation*}
            \Phi_{\theta}(x) = \left(\nabla_{\theta}F^{-1}_{\theta}(F_{\theta}(x))\right)^{\top}
        \end{equation*}
        for $P_{\theta}$-a.e. $x\in\Rbb$.
        In this case, we can compute the Wasserstein information matrix to be
        \begin{equation*}
            J(\theta) = \int_{\Rbb}\nabla_{\theta}F^{-1}_{\theta}(F_{\theta}(x))\left(\nabla_{\theta}F^{-1}_{\theta}(F_{\theta}(x)\right)^{\top}\diff P_{\theta}(x) = \int_{0}^{1}\nabla_{\theta}F^{-1}_{\theta}(u)\left(\nabla_{\theta}F^{-1}_{\theta}(u)\right)^{\top}\diff u,
        \end{equation*}
        as claimed.
    \end{proof}

    \begin{proof}[Proof of Theorem~\ref{thm:WPE-efficient}]
        To begin, we let $x_1,\ldots, x_n\in\Rbb$ denote non-random points and we make some calculations.
        For convenience, let us adopt the notation $F_{\theta}(x) =: F(\theta,x)$ and $F_{\theta}^{-1}(u) =:H(\theta,u)$ for simplicity.
        Then define the function $G_n:\Theta\to\Rbb$ via
        \begin{equation}
            G_n(\theta) = W_2^2(P_{\theta}, \bar P_n) = \int_{0}^{1}|H(\theta,u)-\bar F_n^{-1}(u)|^2\diff u,
            \label{eqn:DiscreteObjective}
        \end{equation}
        where $\bar{F}_n^{-1}$ is the quantile function of the measure $\bar P_n:=\frac{1}{n}\sum_{i=1}^{n}\delta_{x_i}$, i.e., the empirical measure of (non-random) points $x_1,\ldots, x_n\in\Rbb$. By $T_n^{\WPE}(x_1,\ldots,x_n)$ we thus just mean a global minimizer of $G_n$.
        Writing $\hat \theta_n = T_n^{\WPE}$ for convenience, we see that
        \begin{equation*}
            \nabla_{\theta}G_n(\theta) = 2\int_{0}^{1}\left(H(\theta,u)-\bar F_n^{-1}(u)\right)\nabla_{\theta} H(\theta,u)\diff u
        \end{equation*}
        which implies        \begin{equation}\label{eqn:first-order-cond}
            \int_{0}^{1}\left(H(\hat\theta_n,u)-\bar F_n^{-1}(u)\right)\nabla_{\theta}H(\hat\theta_n,u)\diff u = 0.
        \end{equation}
        Before differentiating \eqref{eqn:first-order-cond}, let us note that $\hat\theta_n$ is indeed a differentiable function of the data, on an event of probability tending to one.
        Indeed, define
        \begin{equation*}
            \hat M_n := \left(\int_{0}^{1}\nabla_{\theta}H(\hat\theta_n,u)(\nabla_{\theta}H(\hat\theta_n,u))^{\top}\diff u+\int_{0}^{1}\left(H(\hat\theta_n,u)-\bar F_n^{-1}(u)\right)\nabla_{\theta}^2H(\hat\theta_n,u)\diff u\right),
        \end{equation*}
        which is one half of the Hessian $\nabla_{\theta}^2G_n(\hat\theta_n)$, and note that we show below (without circularity, since that argument uses only assumptions $(i)$ and $(ii)$, the consistency of $T_n^{\WPE}$, and the consistency of the empirical quantiles) that $\hat M_n \to J(\theta^{\ast})$ almost surely.
        Since $J(\theta^{\ast})$ is invertible and $\Theta$ is open, it follows that there is a sequence of events with probability tending to one on which $\hat\theta_n$ lies in the interior of $\Theta$ and $\hat M_n$ is invertible.
        On these events, the implicit function theorem applied to \eqref{eqn:first-order-cond} implies that $\hat\theta_n$ is a continuously differentiable function of $x_1,\ldots, x_n$; since the conclusion of the theorem is a convergence in probability, it suffices to argue on these events for the remainder of the proof.

        Now we differentiate \eqref{eqn:first-order-cond} with respect to $x_i$ for $1\le i\le n$, assuming that the values $x_1,\ldots, x_n$ are distinct
        (This will be satisfied almost surely when we consider the random inputs $X_1,\ldots, X_n$ since these are i.i.d. samples from a distribution which is absolutely continuous with respect to the Lebesgue measure.)
        If we write $\tau$ for the permutation of $\{1,\ldots,n\}$ which sorts $x_1,\ldots, x_n$ into increasing order, then we have
        \begin{equation}\label{eqn:first-order-cond-diff}
            \begin{split}
                &\int_{0}^{1}\left(\left(\nabla_{\theta} H(\hat\theta_n,u)\right)^{\top}\frac{\partial \hat\theta_n}{\partial x_i}-\ind\left\{\frac{\tau(i)-1}{n}\le u <\frac{\tau(i)}{n}\right\}\right)\nabla_{\theta}H(\hat\theta_n,u)\diff u \\
                &\qquad\qquad\qquad\qquad\qquad+\int_{0}^{1}\left(H(\hat\theta_n,u)-\bar F_n^{-1}(u)\right)\nabla_{\theta}^2H(\hat\theta_n,u)\frac{\partial \hat\theta_n}{\partial x_i}\diff u = 0.
            \end{split}
        \end{equation}
        Now \eqref{eqn:first-order-cond-diff} rearranges to
        \begin{equation*}
            \frac{\partial \hat\theta_n}{\partial x_i} = \hat M_n^{-1}\int_{\frac{\tau(i)-1}{n}}^{\frac{\tau(i)}{n}}\nabla_{\theta}H(\hat\theta_n,u)\diff u.
        \end{equation*}
        Then by taking the outer product, summing over $1\le i\le n$, and scaling by $n$, we get
        \begin{equation}\label{eqn:WPE-sens-outer}
            \begin{split}
            &n\sum_{i=1}^{n}\left(\frac{\partial \hat\theta_n}{\partial x_i}\right)\left(\frac{\partial \hat\theta_n}{\partial x_i}\right)^{\top}\\
            &= \hat M_n^{-1}\,n\sum_{i=1}^{n}\left(\int_{\frac{\tau(i)-1}{n}}^{\frac{\tau(i)}{n}}\nabla_{\theta}H(\hat\theta_n,u)\diff u\right)\left(\int_{\frac{\tau(i)-1}{n}}^{\frac{\tau(i)}{n}}\nabla_{\theta}H(\hat\theta_n,u)\diff u\right)^{\top}\hat M_n^{-1} \\
            &= \hat M_n^{-1}\,n\sum_{i=1}^{n}\left(\int_{\frac{i-1}{n}}^{\frac{i}{n}}\nabla_{\theta}H(\hat\theta_n,u)\diff u\right)\left(\int_{\frac{i-1}{n}}^{\frac{i}{n}}\nabla_{\theta}H(\hat\theta_n,u)\diff u\right)^{\top}\hat M_n^{-1}.
            \end{split}
        \end{equation}
        Notice that the permutation $\tau$ disappears in the second equality, hence this identity holds as long as $x_1,\ldots, x_n$ are distinct.

        Now we return to the setting that $X_1,\ldots, X_n$ are i.i.d. samples from $P_{\theta}$ for some fixed $\theta\in\Theta$, and we consider the asymptotics of \eqref{eqn:WPE-sens-outer} as $n\to\infty$.
        First, we claim that $\hat M_n\to J(\theta)$ almost surely as $n\to\infty$, under $P_{\theta}$.
        To see this, note that the consistency of WPE, assumption $(i)$, and Lemma~\ref{DWS-1s} imply that
        \begin{equation*}
            \int_{0}^{1}\nabla_{\theta}H(\hat\theta_n,u)(\nabla_{\theta}H(\hat\theta_n,u))^{\top}\diff u \to \int_{0}^{1}\nabla_{\theta}H(\theta,u)(\nabla_{\theta}H(\theta,u))^{\top}\diff u = J(\theta) 
        \end{equation*}
        holds almost surely under $P_{\theta}$.
        Also, empirical quantile consistency almost surely (say, in $L^2([0,1])$) and assumption $(ii)$ imply
        \begin{equation*}
            \int_{0}^{1}\left(H(\hat\theta_n,u)-\bar F_n^{-1}(u)\right)\nabla_{\theta}^2H(\hat\theta_n,u)\diff u \to 0
        \end{equation*}
        almost surely under $P_{\theta}$.
        Thus, we have $\hat M_n\to J(\theta)$ almost surely, as claimed.

        Lastly, we claim that
        \begin{equation}\label{eqn:WPE-last}
            \begin{split}
                &n\sum_{i=1}^{n}\left(\int_{\frac{i-1}{n}}^{\frac{i}{n}}\nabla_{\theta}H(\hat\theta_n,u)\diff u\right)\left(\int_{\frac{i-1}{n}}^{\frac{i}{n}}\nabla_{\theta}H(\hat\theta_n,u)\diff u\right)^{\top} \\
                &\to \int_{0}^{1}\nabla_{\theta}H(\theta,u)(\nabla_{\theta}H(\theta,u))^{\top}\diff u,
            \end{split}
        \end{equation}
        holds in probability under $P_{\theta}$; by Lemma~\ref{DWS-1s} again, the right side is equal to $J(\theta)$.
        To see this, let us write $P_n:L^2([0,1];\Rbb^{p})\to L^2([0,1];\Rbb^{p})$ for the linear operator of conditional expectation with respect to the $\sigma$-algebra generated by the partition $\{0,\sfrac{1}{n},\sfrac{2}{n},\ldots, 1\}$.
        Note that $\{P_n\}_{n\in\Nbb}$ converges to the identity in the strong operator topology, i.e.,  $f_n\to f$ in norm implies $P_nf_n\to f$ in norm.
        Thus, it follows that \eqref{eqn:WPE-last} holds in probability; we have
        \begin{equation*}
            (P_nf)(u) := n\int_{\frac{i-1}{n}}^{\frac{i}{n}}f(x)\diff x \qquad\textnormal{where}\qquad \frac{i-1}{n}\le u< \frac{i}{n},
        \end{equation*}
        we set $f_n:= \nabla H(\hat{\theta}_n,\cdot)$ for each $n\in\Nbb$, and we recall that a sequence of random variables converges in probability if and only if each subsequence admits a further subsequence converging in probability.
        Finally, Lemma~\ref{DWS-1s} implies that the right side of \eqref{eqn:WPE-last} equals $J(\theta)$, so we have shown
        \begin{equation*}
            n\sum_{i=1}^{n}\left(\frac{\partial \hat\theta_n}{\partial x_i}\right)\left(\frac{\partial \hat\theta_n}{\partial x_i}\right)^{\top} \to (J(\theta))^{-1}J(\theta)(J(\theta))^{-1} = (J(\theta))^{-1}
        \end{equation*}
        in probability under $P_{\theta}$, as needed.
    \end{proof}

    \begin{proof}[Proof of Theorem~\ref{thm:WPEVariance}]
For convenience, write $\hat\theta_n := T_n^{\WPE}(X_1,\ldots, X_n)$.
As in the proof of Theorem~\ref{thm:WPE-efficient}, the function $G_n$ defined in \eqref{eqn:DiscreteObjective} has gradient 
\begin{equation*}
            \nabla_{\theta}G_n(\theta) = \int_{0}^{1}\left( F_{\theta}^{-1}(u)-\bar F_n^{-1}(u)\right) \nabla_\theta F_{\theta}^{-1}(u)\diff u,
        \end{equation*}
hence its Hessian is
\begin{align*}
   \nabla_{\theta}^2 G_n(\theta) & =      \int_{0}^{1} \nabla_\theta F^{-1}_{\theta}(u) (\nabla_\theta F_{\theta}^{-1}(u))^{\top}\diff u    +  \int_0^1 ( F_{\theta}^{-1}(u)  - \bar{F}_{n}^{-1}(u) ) \nabla_{\theta}^2 F_{\theta}^{-1}(u) \diff u,
   \\ &= J(\theta) +    \int_0^1 ( F_{\theta}^{-1}(u) - \bar{F}_{n}^{-1}(u) ) \nabla^2_\theta F_{\theta}^{-1}(u) \diff u
   \label{eqn:HessianGn}
\end{align*}
where we used Lemma~\ref{DWS-1s} in the second equality.
Now we consider a Taylor expansion of $\nabla_\theta G_n$ around $\theta^*$ to get the bound
\begin{equation*}
    \nabla_{\theta}G_n(\hat \theta_n) = \nabla_{\theta}G_n(\theta^{\ast}) + \nabla_{\theta}^2G_n(\hat\theta_n)(\hat\theta_n-\theta^{\ast}) + o_{\Pbb_{\theta}}(\|\hat\theta_n-\theta^{\ast}\|).
\end{equation*}
Since $\hat\theta_n$ satisfies $\nabla_{\theta}G_n(\hat\theta_n) = 0$, it follows that $\sqrt{n}(\hat\theta_n-\theta^{\ast})$ has the same asymptotic distribution as
\begin{equation*}
    -\left(\nabla_{\theta}^2G_n(\hat\theta_n)\right)^{-1}\sqrt{n}\nabla_{\theta}G_n(\theta^{\ast}).
\end{equation*}
To find this asymptotic distribution, we first note that, under our assumptions, we have
\[ \nabla_{\theta}^2G_n(\theta^*) \rightarrow J(\theta^*)\]
almost surely under $\Pbb_{\theta^{\ast}}$ as $n\to\infty$.
Second, note that the functional central limit theorem for quantiles (see \cite[Theorem 4.6]{DelBarrio2005}) implies  
\begin{align*}
   \sqrt{n }\nabla_\theta G_n(\theta^*) & = \sqrt{n} \int_{0}^{1}\left(F_{\theta^{\ast}}^{-1}(u)-\bar F_n^{-1}(u)\right) \nabla_\theta F_{\theta^{\ast}}^{-1}(u)\diff u 
   \\ & \rightarrow  \int_{0}^1 \frac{B(u)}{p_{\theta^*}(F_{\theta^*}^{-1}(u))} \nabla_\theta F_{\theta^{\ast}}^{-1}(u) \diff u   , 
\end{align*}
in distribution under $\Pbb_{\theta^\ast}$ as $n\to\infty$, where $\{B(u)\}_{0\le u\le 1}$ is a standard Brownian bridge.
Combining the two observations above with Slutsky's theorem implies desired result.
\end{proof}

\begin{remark}
In \cite[Section 3]{CarlierGalichonGalichon}, the authors present a similar analysis to obtain formulas for the second derivatives of the squared-Wasserstein distance between a parameterized measure and an empirical measure. Details are presented in a two-dimensional setting. 
\end{remark}

\begin{proof}[Proof of Theorem~\ref{thm:WPE-efficient-gen}]
Because of Assumption \ref{assum:AdditionalRegularity}, we may apply Proposition~\ref{prop:Derivatives} to see that~\eqref{eqn:WPE_Sensitivity} holds, i.e. the function $\mathcal{L}(x_1,\ldots,x_n,\theta)=W_2^2(P_{\theta},n^{-1}\sum_{i=1}^{n}\delta_{x_i})$ satisfies
\begin{align*}
    &n\sum_{i=1}^n \left\| \nabla_{x_i} \hat{\theta}_n(X_1, \dots, X_n) \right\|^2 \\
    &= \left( \frac{\partial^2 }{\partial \theta^2} \mathcal{L}(X_1,\ldots,X_n,\hat{\theta}_n)  \right)^{-2} \frac{1}{n}\sum_{i=1}^n \left\|  n\nabla_{x_i} \frac{\partial }{\partial \theta} \mathcal{L}(X_1,\ldots,X_n,\hat{\theta}_n)\right\|^2.
\end{align*}
(Recall that we previously stated it only as a heuristic calculation.)
We will handle the convergence of each term separately.

First, we consider the summation.
Use~\eqref{eqn:AuxWPEPartialDerivsDeveloped1} to divide it into a ``bulk'' and ``error'' term
    \begin{equation*}
        \frac{1}{n}\sum_{i=1}^{n}\left\| n \nabla_{x_i} \frac{\partial }{\partial \theta} \mathcal{L}(X_1,\ldots,X_n,\hat{\theta}_n)\right\|^2 = \mathfrak{B}_n + \mathfrak{E}_n,
    \end{equation*}
    where
\begin{align*}
\mathfrak{B}_n &:= \frac{1}{n} \sum_{i=1}^n \left\| n \int_{V_i} \Phi_{\hat{\theta}_n}(x) p_{\hat{\theta}_n}(x) \diff x   \right\|^2, \qquad\textnormal{ and }\\
\mathfrak{E}_n &:= \frac{1}{n} \sum_{i=1}^n \left\| n\sum_{j \sim i} \int_{\Gamma_{ij}} \left(2(\Phi_{\hat \theta_n}(x))^{\top}(X_j - X_i) +  \left(  \frac{\partial b_j}{\partial \theta}  - \frac{\partial b_i}{\partial \theta}  \right) \right) \frac{x-X_i} {\|X_i - X_j\|}p_{\hat \theta_n}(x) \diff \mathcal{H}^{d-1}(x)  \right\|^2.
\end{align*}
We have $\mathfrak{B}_n\to J(\theta^{\ast})$ in probability, following the same reasoning as Theorem~\ref{thm:WPE-efficient} for the case $d=1$, using the asymptotic consistency of $\hat{\theta}_n$ and the assumed regularity of the map $(x, \theta ) \mapsto (p_\theta(x),\Phi_\theta(x))$.
For the error term, write $R_n$ for the error appearing in condition $(i)$ of the theorem, and use \eqref{eqn:CTransformDualPotential} to compute
\begin{align*}
\frac{\partial b_j}{\partial \theta} -  \frac{\partial b_i}{\partial \theta} & =  \frac{\partial }{\partial \theta} \psi^c_{P_{\theta}\to \bar P_n}(X_j) -  \frac{\partial }{\partial \theta} \psi^c_{P_{\theta}\to \bar P_n}(X_i)
\\& = \int_0^1\left(\nabla_{x}\frac{\partial }{\partial \theta} \psi^c_{P_{\theta}\to \bar P_n}((1-t)X_i + t X_j)\right)^{\top}(X_j - X_i) \diff t
\\& = \int_0^1\left(\nabla_{x}\frac{\partial }{\partial \theta} \psi^c_{P_{\theta}\to P_{\theta^*}}((1-t)X_i + t X_j)\right)^{\top} (X_j - X_i) \diff t + R_n\|X_i-X_j\|.
\end{align*}
We can further manipulate first term on the right side, using \eqref{aux:GradientDualPotential} and Lemma \ref{lem:AuxDuality}:
\begin{align*}
&\int_0^1\left(\nabla_{x}\frac{\partial }{\partial \theta} \psi^c_{P_{\theta}\to P_{\theta^*}}((1-t)X_i + t X_j)\right)^{\top} (X_j - X_i) \diff t
\\ & = \int_0^1\left(\frac{\partial }{\partial \theta}\nabla_{x} \psi^c_{P_{\theta}\to P_{\theta^*}}((1-t)X_i + t X_j)\right)^{\top} (X_j - X_i) \diff t
\\& = -\int_0^1\left(\frac{\partial}{\partial \theta} ( 2 \transport_{P_{\theta^*} \to P_\theta}((1-t)X_i+tX_j) - 2((1-t)X_i + tX_j))\right)^{\top} (X_j- X_i) \diff t
\\& =- 2 \int_0^1\left(\frac{\partial}{\partial \theta} \transport_{P_{\theta^*} \to P_\theta}((1-t)X_i+tX_j)\right)^{\top} (X_j- X_i) \diff t + O(\|X_j-X_i\|^2).
\end{align*}
When $\theta= \hat{\theta}_n$, we can use the smoothness of part $(iii)$ of Assumption~\ref{assum:AdditionalRegularity} to further write
\begin{align*}
&-2 \int_0^1\left(\frac{\partial}{\partial \theta} \transport_{P_{\theta^*} \to P_{\hat \theta_n}}((1-t)X_i+tX_j)\right)^{\top} (X_j- X_i) \diff t \\
&= -2\int_0^1 \left(\Phi_{\theta^*}((1-t)X_i+tX_j)\right)^{\top} (X_j- X_i) \diff t
+ O(\|\hat{\theta}_n - \theta^*\|\cdot\|X_i - X_j\|) \\
&= -2 (\Phi_{\theta^*}(X_i))^{\top} (X_j- X_i)  + O\bigg((\|\hat{\theta}_n - \theta^*\| + \|X_i-X_j\| )\|X_i - X_j\|\bigg).
\end{align*}
Combining the three previous displays yields
\begin{equation*}
\frac{\partial b_j}{\partial \theta} -  \frac{\partial b_i}{\partial \theta} = -2 (\Phi_{\theta^*}(X_i))^{\top} (X_j- X_i)  + O\bigg((R_n+\|\hat{\theta}_n - \theta^*\| + \|X_i-X_j\| )\|X_i - X_j\|\bigg),
\end{equation*}
so plugging in to the definition of $\mathfrak{E}_n$ shows
\begin{align*}
    &\mathfrak{E}_n\\
    &= \frac{1}{n} \sum_{i=1}^n \left\| n\sum_{j \sim i} \int_{\Gamma_{ij}} \left( O\bigg(R_n+\|\hat{\theta}_n - \theta^*\| + \|X_i-X_j\|\bigg)\right) (x-X_i)p_{\hat \theta_n}(x) \diff \mathcal{H}^{d-1}(x)  \right\|^2 \\
    &= O\left((R_n+\|\hat{\theta}_n-\theta^{\ast}\|)^2\frac{1}{n} \sum_{i=1}^n \left\| n\sum_{j \sim i} \int_{\Gamma_{ij}} (x-X_i)p_{\hat \theta_n}(x) \diff \mathcal{H}^{d-1}(x)  \right\|^2\right) \\
    &\qquad + O\left(\frac{1}{n} \sum_{i=1}^n \left\| n\sum_{j \sim i} \int_{\Gamma_{ij}} \|X_i-X_j\| (x-X_i)p_{\hat \theta_n}(x) \diff \mathcal{H}^{d-1}(x)  \right\|^2 \right).
\end{align*}
Now we can show that both of these terms vanish.
For the first term, simply bound
\begin{align*}
    &O\left((R_n+\|\hat{\theta}_n-\theta^{\ast}\|)^2\frac{1}{n} \sum_{i=1}^n \left\| n\sum_{j \sim i} \int_{\Gamma_{ij}} (x-X_i)p_{\hat \theta_n}(x) \diff \mathcal{H}^{d-1}(x)  \right\|^2\right)\\
    &= O\left((R_n+\|\hat{\theta}_n-\theta^{\ast}\|)^2\frac{1}{n} \sum_{i=1}^n\left( n\int_{\partial V_i} \|x-X_i\|p_{\hat \theta_n}(x) \diff \mathcal{H}^{d-1}(x)\right)\right)^2 \\
    &= O(R_n+\|\hat{\theta}_n-\theta^{\ast}\|)^2\cdot O\left(\frac{1}{n} \sum_{i=1}^n\left( n\int_{\partial V_i} \|x-X_i\|p_{\hat \theta_n}(x) \diff \mathcal{H}^{d-1}(x)\right)\right)^2,
\end{align*}
which vanishes in probability by assumption.
For the second term, set $    D_n:=\max\left\{\|X_i-X_j\|: i\sim j\right\}$ which vanishes in probability since $U$ is bounded and connected and $\{p_{\theta}\}_{\theta\in\Theta}$ are uniformly bounded from below; then we have
\begin{align*}
    &O\left(\frac{1}{n} \sum_{i=1}^n \left\| n\sum_{j \sim i} \int_{\Gamma_{ij}} \|X_i-X_j\| (x-X_i)p_{\hat \theta_n}(x) \diff \mathcal{H}^{d-1}(x)  \right\|^2 \right) \\
    &= O(D_n^2)\cdot O\left(\frac{1}{n} \sum_{i=1}^n \left\| n\sum_{j \sim i} \int_{\Gamma_{ij}}(x-X_i)p_{\hat \theta_n}(x) \diff \mathcal{H}^{d-1}(x)  \right\|^2 \right) \\
    &= O(D_n^2)\cdot O\left(\frac{1}{n} \sum_{i=1}^n \left( n\int_{\partial V_i}\|x-X_i\|p_{\hat \theta_n}(x) \diff \mathcal{H}^{d-1}(x)  \right)^2 \right)
\end{align*}
which vanishes in probability.
Therefore, we conclude
\begin{equation*}
    \frac{1}{n}\sum_{i=1}^{n}\left\| n \nabla_{x_i} \frac{\partial }{\partial \theta} \mathcal{L}(X_1,\ldots, X_n,\hat{\theta}_n)\right\|^2 \to 4J(\theta^{\ast})
\end{equation*}
in probability as $n\to\infty$.

On the other hand, equation~\eqref{eqn:AuxWPEPartialDerivsDeveloped2} and condition $(i)$ imply that $\frac{\partial^2}{\partial \theta^2}\mathcal{L}(X_1,\ldots, X_n,\hat{\theta}_n)$ converges in probability to
\[ \left( \int_{\Rbb^d}  \frac{\partial}{\partial \theta}\psi_{P_\theta \rightarrow P_{\theta^*}}(x) \frac{\partial}{\partial \theta } p_{\theta}(x)\diff x + \int_{\Rbb^d} \psi_{P_\theta \rightarrow P_{\theta^*}}(x) \frac{\partial^2}{\partial \theta^2 } p_\theta (x)   \diff x \right) \bigg|_{\theta=\theta^*}. \]
Now note that we have $\psi_{P_\theta \rightarrow P_{\theta^*}} \equiv 0$ and thus the above expression reduces to 
\begin{align*}
\left( \int_{\Rbb^d}  \frac{\partial}{\partial \theta}\psi_{P_\theta \rightarrow P_{\theta^*}}(x) \frac{\partial}{\partial \theta } p_{\theta}(x)\diff x \right) \bigg|_{\theta=\theta^*}.
\end{align*}
Since Assumption~\ref{assum:AdditionalRegularity} implies that we can apply the pointwise continuity equation and then the divergence theorem, we can use equation \eqref{aux:GradientDualPotential} to conclude
\begin{align*}
&\left( \int_{\Rbb^d}  \frac{\partial}{\partial \theta}\psi_{P_\theta \rightarrow P_{\theta^*}}(x) \frac{\partial}{\partial \theta } p_{\theta}(x)\diff x \right) \bigg|_{\theta=\theta^*} \\
& =  \left( \int_{\Rbb^d}  \left(\nabla_x \frac{\partial}{\partial \theta}\psi_{P_\theta \rightarrow P_{\theta^*}}(x)\right)^{\top}\Phi_{\theta^*}(x) p_{\theta^*}(x) \diff x\right)\bigg|_{\theta=\theta^*}
 \\& =  \left( \int_{\Rbb^d}   \left(\frac{\partial}{\partial \theta} \nabla_x \psi_{P_\theta \rightarrow P_{\theta^*}}(x)\right)^{\top} \Phi_{\theta^*}(x) p_{\theta^*}(x) \diff x\right)\bigg|_{\theta=\theta^*}
 \\& = 2 \int_{\Rbb^d} |\Phi_{\theta^*}(x)|^2  p_{\theta^*}(x) \diff x = 2 J(\theta^*).
\end{align*}
Thus, we have $\frac{\partial^2}{\partial \theta^2}\mathcal{L}(X_1,\ldots, X_n,\hat\theta_n) \to 2J(\theta^{\ast})$ in probability.

Therefore, we have shown that the right-side of equation~\eqref{eqn:WPE_Sensitivity} converges to $(2J(\theta^{\ast}))^{-2}4J(\theta^{\ast}) = (J(\theta^{\ast}))^{-1}$, and this finishes the proof.
\end{proof}

\begin{proof}[Proof of Theorem~\ref{thm:OSE}]
    It suffices to show, for all $\lambda\in\Rbb^p$, that $\tilde{T}_n:=\lambda^{\top}T_n^{\OSE}$ has asymptotic sensitivity equivalent to $n^{-1}\lambda^{\top}(J(\theta))^{-1}\lambda$.
    To do this, let us define
    \begin{equation*}
        Z_n:=\frac{1}{n}\sum_{i=1}^n\phi_{T_n}(X_i)\in\Rbb^p.
    \end{equation*}
    Then, we differentiate $\tilde T_n$ using the chain rule and condition (i) to obtain 
    \begin{align*}
        \nabla_{x_j}\tilde T_n
        &=
        (D_{x_j}T_n)^\top\lambda +
        (D_{x_j}T_n)^\top
        D(J^{-1})(T_n)^{\ast}[\lambda Z_n^{\top}] \\
        &\quad+
        (D_{x_j}T_n)^\top
        \left(\frac{1}{n}\sum_{i=1}^nD_\theta\phi_{T_n}(X_i)\right)^{\top}(J(T_n))^{-1}\lambda \\
        &\quad+
        \frac1n
        (D\phi_{T_n}(X_j))^{\top}(J(T_n))^{-1}\lambda.
    \end{align*}
    where $D(J^{-1})(T_n):\Rbb^{p}\to\Rbb^{p\times p}$ is the Jacobian of the map $\theta\mapsto (J(\theta))^{-1}$ at $T_n$, and $D(J^{-1})(T_n)^{\ast}:\Rbb^{p\times p}\to\Rbb^{p}$ is its adjoint.
    (Note that $D(J^{-1})(T_n)^{\ast}[A]$ is the adjoint map applied to $A$, not the conjugate transpose of the matrix $D(J^{-1})(T_n)[A]$.)
    Now define
    \begin{equation*}
        V_n
        :=
        \lambda
        +
        D(J^{-1})(T_n)^{\ast}[\lambda Z_n^{\top}]
        +
        \left(
        \frac1n\sum_{i=1}^nD_\theta\phi_{T_n}(X_i)
        \right)^{\top}
        (J(T_n))^{-1}\lambda,
    \end{equation*}
    so that we have
    \begin{equation*}
    \nabla_{x_j}\tilde T_n
    =
    (D_{x_j}T_n)^\top V_n
    +
    \frac1n
    (D\phi_{T_n}(X_j))^\top
    (J(T_n))^{-1}\lambda.
    \end{equation*}
    In particular, we can compute:
    \begin{equation}\label{eqn:OSE-sens}
        \begin{split}
            n\sum_{j=1}^{n}\|\nabla_{x_j}\tilde T_n\|^2
            &=
            n\sum_{j=1}^{n}\|(D_{x_j}T_n)^\top V_n\|^2 \\
            &\quad+
            \frac{1}{n}\sum_{j=1}^{n}
            \Big\|
            (D\phi_{T_n}(X_j))^\top
            (J(T_n))^{-1}\lambda
            \Big\|^2 \\
            &\quad+
            2\sum_{j=1}^{n}
            \Big\langle
            (D_{x_j}T_n)^\top V_n,
            (D\phi_{T_n}(X_j))^\top
            (J(T_n))^{-1}\lambda
            \Big\rangle,
        \end{split}
    \end{equation}
    so it suffices to analyze the right side asymptotically as $n\to\infty$.
 
    To do this, we claim $V_n\to 0$ in probability.
    Indeed, combine the consistency of $T_n$ and ULLN in condition (ii) to obtain the following convergences in probability
     \begin{equation*}
        \frac{1}{n}\sum_{i=1}^{n}\phi_{T_n}(X_i)\to \Ebb_{\theta^{\ast}}\left[\phi_{\theta^{\ast}}(X)\right]=0\qquad \textnormal{and}\qquad\frac{1}{n}\sum_{i=1}^{n}\partial_{\theta}\phi_{T_n}(X_i)\to \Ebb_{\theta^{\ast}}\left[\partial_{\theta}\phi_{\theta^{\ast}}(X)\right]
    \end{equation*}
    where the evaluation of the first expectation uses the centering convention for potentials.
    To evaluate the second expectation, note that we can differentiate both sides of the centering condition $0 = \int_{\Rbb^d}\phi_{\theta}\diff P_{\theta}$ using the chain rule and the continuity equation (by assumptions (i) and (iv) and Proposition~\ref{prop:ContinuityEqnLocallyLipschitz}), leading to
    \begin{equation*}
        0 = D_{\theta}\int_{\Rbb^d}\phi_{\theta}\diff P_{\theta} = \int_{\Rbb^d}D_{\theta}\phi_{\theta}\diff P_{\theta} + \int_{\Rbb^d}D\phi_{\theta}\Phi_{\theta}\diff P_{\theta} = \int_{\Rbb^d}D_{\theta}\phi_{\theta}\diff P_{\theta} + J(\theta)
    \end{equation*}
    hence
    \begin{equation*}
        \frac{1}{n}\sum_{i=1}^{n}D_{\theta}\phi_{T_n}(X_i)\to \Ebb_{\theta^{\ast}}\left[D_{\theta}\phi_{\theta^{\ast}}(X)\right] = -J(\theta^{\ast}).
    \end{equation*}
    in probability as $n\to\infty$.
    Now we use this to show $V_n\to 0$ in probability.
    We just showed $Z_n\to 0$ in probability, so we may apply condition (iii) and $J(\theta^{\ast})\succ 0$ to guarantee that $\theta\mapsto (J(\theta))^{-1}$ is continuously differentiable, hence
    \begin{equation*}
        D(J^{-1})(T_n)^{\ast}[\lambda Z_n^{\top}]\to 0
    \end{equation*}
    in probability (by the continuous mapping theorem).
    Thus, we have
    \begin{align*}
        V_n &= \lambda
        + D(J^{-1})(T_n)^{\ast}[\lambda Z_n^{\top}] +
        \left(\frac1n\sum_{i=1}^nD_\theta\phi_{T_n}(X_i)
        \right)^{\top}(J(\theta^{\ast}))^{-1}\lambda \\
        &\to
        \lambda + 0 - J(\theta^{\ast})(J(\theta^{\ast}))^{-1}\lambda \\
        &= 0
    \end{align*}
    in probability, as claimed.
 
    Now we finish the proof by analyzing \eqref{eqn:OSE-sens} as $n\to\infty$.
    We claim that the first term vanishes in probability; indeed, bound
    \begin{equation*}
        n\sum_{j=1}^{n}\|(D_{x_j}T_n)^\top V_n\|^2 \le n\sum_{j=1}^{n}\|D_{x_j}T_n\|^2\cdot\|V_n\|^2
    \end{equation*}
    and note that the first factor is tight by assumption and that the second factor vanishes by the previous paragraph.
    We also claim that the third term also vanishes, which we show by using Cauchy-Schwarz twice to obtain
    \begin{align*}
        &\left|\sum_{j=1}^{n}\Big\langle (D_{x_j}T_n)^\top V_n, (D\phi_{T_n}(X_j))^\top (J(T_n))^{-1}\lambda \Big\rangle\right| \\
        &\le\sum_{j=1}^{n}\Big\| (D_{x_j}T_n)^\top V_n\Big\|\cdot\Big\|(D\phi_{T_n}(X_j))^\top (J(T_n))^{-1}\lambda \Big\| \\
        &\le\sqrt{\sum_{j=1}^{n}\Big\| (D_{x_j}T_n)^\top V_n\Big\|^2\sum_{j=1}^{n}\Big\|(D\phi_{T_n}(X_j))^\top (J(T_n))^{-1}\lambda \Big\|^2} \\
        &\le\|V_n\|\cdot\|(J(T_n))^{-1}\lambda\|\sqrt{n\sum_{j=1}^{n}\Big\| D_{x_j}T_n\Big\|^2\cdot\frac{1}{n}\sum_{j=1}^{n}\Big\|D\phi_{T_n}(X_j)\Big\|^2}.
    \end{align*}
    Then, note that $V_n$ vanishes by the previous paragraph, $(J(T_n))^{-1}\lambda$ is tight by (iii) and the consistency of $T_n$, and the remaining factors are tight by assumption and by the WLLN, respectively.
    Lastly, we consider the second term of~\eqref{eqn:OSE-sens}, for which the ULLN in (ii) and the consistency of $T_n$ imply
    \begin{align*}
        &\frac{1}{n}\sum_{j=1}^{n}\Big\|(D\phi_{T_n}(X_j))^\top(J(T_n))^{-1}\lambda\Big\|^2 \\
        &=\lambda^{\top}(J(T_n))^{-1}\frac{1}{n}\sum_{j=1}^{n}D\phi_{T_n}(X_j)(D\phi_{T_n}(X_j))^\top(J(T_n))^{-1}\lambda \\
        &\to\lambda^{\top}(J(\theta^{\ast}))^{-1}J(\theta^{\ast})(J(\theta^{\ast}))^{-1}\lambda \\
        &=\lambda^{\top}(J(\theta^{\ast}))^{-1}\lambda.
    \end{align*}
    This finishes the proof.
    \end{proof}

\begin{proof}[Proof of Theorem~\ref{thm:loc-uniform-WCR}]
    Fix $\theta\in\Theta$ and $R>0$.
    Since $\Theta$ is open, we have  $\theta + h/\sqrt{n}\in\Theta$ for all $\|h\|\le R$ whenever $n$ is sufficiently large.
    
    By construction, $T_n$ is a locally Lipschitz unbiased estimator of the continuously differentiable function $\chi_n(\theta):=\Ebb_{\theta}[T_n]$, and, by assumption, hypotheses (i), (ii), and (iii) of Theorem~\ref{thm:WCR} are satisfied.
    Therefore, Theorem~\ref{thm:WCR} applied to the estimand $\chi_n$ yields
    \begin{equation}\label{eqn:loc-unif-proof-1}
        n\Sen_{\theta+\frac{h}{\sqrt{n}}}(T_n) \ge \left(\nabla \chi_n\left(\theta+\frac{h}{\sqrt{n}}\right)\right)^{\top}\left(J\left(\theta+\frac{h}{\sqrt{n}}\right)\right)^{-1}\nabla \chi_n\left(\theta+\frac{h}{\sqrt{n}}\right)
    \end{equation}
    for all $\|h\|\le R$.
    
    Now we show that the right side of \eqref{eqn:loc-unif-proof-1} converges to $(\nabla \chi(\theta))^{\top}(J(\theta))^{-1}\nabla \chi(\theta)$, uniformly over $\|h\|\le R$.
    To do this, write $a_{n,h}:=\nabla \chi_n(\theta+h/\sqrt{n})$ and $b:=\nabla\chi(\theta)$, as well as $M_{n,h}:=(J(\theta+h/\sqrt{n}))^{-1}$ and $N:=(J(\theta))^{-1}$.
    On the one hand, the uniform convergence assumption \eqref{eqn:uniform} states exactly that $\sup_{\|h\|\le R}\|a_{n,h}-b\|\to 0$; in particular, $\sup_{\|h\|\le R}\|a_{n,h}\|$ is bounded in $n$.
    On the other hand, since $\sup_{\|h\|\le R}\|h/\sqrt{n}\|\le R/\sqrt{n}\to 0$, and since $J$ is continuous at $\theta$ and $J(\theta)$ is invertible by hypothesis (i) of Theorem~\ref{thm:WCR}, the continuity of matrix inversion at $J(\theta)$ implies $\sup_{\|h\|\le R}\|M_{n,h}-N\|\to 0$.
    Combining these observations with the bound
    \begin{equation*}
        \left|a_{n,h}^{\top}M_{n,h}a_{n,h}-b^{\top}Nb\right| \le \|a_{n,h}\|^2\|M_{n,h}-N\| + \|N\|\left(\|a_{n,h}\|+\|b\|\right)\|a_{n,h}-b\|,
    \end{equation*}
    which holds for all $\|h\|\le R$, we conclude
    \begin{equation*}
        \sup_{\|h\|\le R}\left|\left(\nabla \chi_n\left(\theta+\frac{h}{\sqrt{n}}\right)\right)^{\top}\left(J\left(\theta+\frac{h}{\sqrt{n}}\right)\right)^{-1}\nabla \chi_n\left(\theta+\frac{h}{\sqrt{n}}\right) - (\nabla \chi(\theta))^{\top}(J(\theta))^{-1}\nabla \chi(\theta)\right| \to 0
    \end{equation*}
    as $n\to\infty$.
    
    Finally, taking the infimum over $\|h\|\le R$ in \eqref{eqn:loc-unif-proof-1} and then sending $n\to\infty$, the preceding display yields
    \begin{equation*}
        \liminf_{n\to\infty}\inf_{\|h\|\le R}n\Sen_{\theta + \frac{h}{\sqrt{n}}}(T_n) \ge (\nabla \chi(\theta))^{\top}(J(\theta))^{-1}\nabla \chi(\theta),
    \end{equation*}
    which is \eqref{eqn:uniform-LB}, as needed.
\end{proof}

\section{Multivariate Wasserstein Projection Estimators}\label{app:discussions-asymp}

In this appendix, we give a detailed discussion of our result on the asymptotic sensitivity-efficiency of the Wasserstein projection estimator (WPE) in the setting where the data are multidimensional, i.e., $d\ge 2$, including an explicit enumeration of the conditions required in our proof-of-concept result in Theorem~\ref{thm:WPE-efficient-gen}.
This complements the discussion in Subsection~\ref{subsec:WPE} in which we proved asymptotic sensitivity-efficiency of WPE in the univariate case, i.e., $d=1$, under mild assumptions.

To begin, we review some aspects of Kantorovich duality that will be needed in our statement and proof.
Recall that the Kantorovich duality theorem (see \cite[Theorem 1.3]{VillaniOT}) for the quadratic cost $c(x,y) := \|x-y\|^2$ implies that, for any two probability measures $P, Q \in \Pcal_2(\Rbb^d)$ with finite second moments, 
\begin{equation}
  W_2^2(P, Q) = \max_{\eta \in L^1(Q) }\int_{\Rbb^d } \eta^c(x) \diff P(x) + \int_{\Rbb^d} \eta(y) \diff Q (y).  
  \label{eqn:KDuality}
\end{equation}
In the above, 
\begin{equation}
    \eta^c(x ):= \inf_{y \in \Rbb^d} ( \|x- y\|^2 - \eta(y) )
    \label{def:cTransform}
\end{equation}
is the $c$-transform of the measurable function $\eta$. That the maximum is attained in \eqref{eqn:KDuality} is a consequence of \cite[Theorem 2.9]{VillaniOT} and the fact that $P$ and $Q$ have finite second moments. We will refer to $(\eta^c,\eta)$ as a pair of \textit{optimal Kantorovich potentials} for $P$ and $Q$.

When $P$ is a probability measure with a density with respect to the Lebesgue measure, a useful identity for a pair of optimal dual potentials is:
\begin{equation}
\eta^c(x) + \eta(\transport_{P \mapsto Q}(x)) = \| \transport_{P \to Q}(x) - x \|^2, \quad 
\label{eqn:AuxDuality1}
\end{equation}
where $\transport_{P \to Q}$ is the unique (that is, unique $P$-a.e.) optimal transport map between $P$ and $Q$ discussed in \eqref{eqn:Monge-OT}. Indeed, \eqref{eqn:AuxDuality1} follows from \eqref{eqn:KDuality}, which in this case gives 
\begin{eqnarray*}  
\int_{\Rbb^d} \eta^c(x) \diff P(x) +  \int_{\Rbb^d} \eta(\transport_{P \to Q}(x)) \diff P(x) & = & \int_{\Rbb^d} \eta^c(x) \diff P(x) +  \int_{\Rbb^d} \eta(y) \diff Q(y) \\
& = & \int_{\Rbb^d} \| \transport_{P \rightarrow Q}(x) - x \|^2 \diff P(x)  ,
\end{eqnarray*}
and from the fact that, by definition of $\eta^c$, the inequality $\eta^c(x) + \eta(y) \leq \|x-y\|^2$ holds for all $x,y \in \Rbb^d$.

\medskip 

While the existence of optimal Kantorovich potentials holds under very general assumptions, additional geometric conditions are required to guarantee uniqueness of optimal potentials up to additive constants. We discuss this next.

\begin{lemma}
\label{lem:AuxDuality}
Let $P \in \mathcal{P}_{2, \mathrm{ac}}(\Rbb^d)$, and suppose that the support of $P$ is the closure of an open and connected set $U$. Then, up to additive constants, there is a unique solution to \eqref{eqn:KDuality}. Precisely, for any two solutions $\eta, \tilde\eta$ of \eqref{eqn:KDuality}, there exists a constant $a \in \Rbb$ such that $\eta(y) = \tilde{\eta}(y) + a$ for $Q$-a.e. $y \in \Rbb^d$ and $\eta^c(x) = \tilde{\eta}^c(x) -a$ for $P$-a.e. $x \in \Rbb^d$.
\end{lemma}
\begin{proof}
Suppose that $\eta, \tilde{\eta}$ are two optimizers of \eqref{eqn:KDuality}. We start by observing that $\eta^{ccc} = \eta^c$ and that
\begin{equation}
   \eta^{cc}(y) = \eta(y), \quad Q\mbox{-a.e.}~y \in \Rbb^d. 
   \label{eqn:AuxDuality2}
\end{equation}
The first identity is straightforward to prove and is standard in the literature of optimal transport. For the second identity, note that the inequality $\eta^{cc}(y) \geq \eta(y)$ for all $y \in \Rbb^d$ follows directly from the definition of the $c$ transform. To get the reverse inequality, it suffices to note that, by optimality of $\eta$, we must have
\[   \int_{\Rbb^d} \eta^c(x) \diff P(x) +\int_{\Rbb^d} \eta(y) \diff Q(y)    \geq  \int_{\Rbb^d} \eta^c(x) \diff P(x) +\int_{\Rbb^d} \eta^{cc}(y) \diff Q(y) ,  \]
 which, when combined with $\eta^{cc} \geq \eta$, implies \eqref{eqn:AuxDuality2}.

Next, using \eqref{eqn:AuxDuality1}, \eqref{eqn:AuxDuality2}, and the definition of $\eta^{cc}$, we see that 
\begin{equation}
   \nabla \eta^c(x) = 2 (x- \transport_{P \to Q}(x) ), \quad P\mbox{-a.e.}\; x \in \Rbb^d. 
   \label{aux:GradientDualPotential}
\end{equation}
Using the analogous formula for the gradient of $\tilde \eta$, we obtain
\[ \nabla \eta^c(x) = \nabla \tilde{\eta}^c(x), \quad P\mbox{-a.e.}\; x \in \Rbb^d.   \]
In particular, the function $\eta^c-\tilde{\eta}^c$ has gradient equal to zero for $P$ almost every $x \in \Rbb^d$. Since the function $\eta-\tilde{\eta}$ is locally Lipschitz and $U$ is open and connected, this implies that there is a constant $a \in \Rbb$ such that
\[ \eta^c(x) = \tilde{\eta}^c(x) -a \]
for every $x$ in $P$'s support. Using this equality and \eqref{eqn:AuxDuality1} for both $\eta$ and $\tilde \eta$, we deduce that $\eta(y) = \tilde{\eta}(y) + a$ for $Q$-a.e. $y \in \Rbb^d$, completing in this way the proof. 
\end{proof}

\begin{remark}
The pathwise connectivity of the support of $P$ is an important assumption for the conclusion of Lemma \ref{lem:AuxDuality} to hold. To see why, imagine a probability measure $P \in \mathcal{P}_{2,\mathrm{ac}}(\Rbb)$ whose support is the set $[-2,-1] \cup [1,2]$ and satisfies $P([-2,-1]) = P([1,2]) =1/2$. Consider now an empirical measure $Q= \frac{1}{4}\sum_{i=1}^4 \delta_{x_i}$ with $x_1, x_2 \in [-2,-1]$ and $x_3, x_4 \in [1,2]$. Since the distance between the sets $[-2,-1]$ and $[1,2]$ is larger than the diameter of each of the sets, the optimal transport problem between $P$ and $Q$ can be completely decoupled into optimal transport problems within each of the connected components of $P$'s support. Due to this, optimal Kantorovich potentials are defined up to additive constants on \textit{each} of the connected components of the support of $P$, and not just globally. More general conditions for uniqueness of optimal Kantorovich potentials than the ones in Lemma \ref{lem:AuxDuality} are discussed in detail in \cite{UniquenessKantorovichPotentials}.
\end{remark}

Next, we can move to our statistical results.
Towards this end, we enumerate the following conditions we will require on the model $\Pcal$.

\begin{assumption}
\label{assum:AdditionalRegularity}
Suppose that $p=1$ and that $\Pcal=\{P_{\theta}:\theta\in\Theta\}$ is DWS and satisfies:
\begin{enumerate}
\item[(i)] There is a bounded, connected, open  $U \subseteq \Rbb^d$ such that the support of $P_{\theta}$ equals $\overline{U}$, for all $\theta\in\Theta$.
  \item[(ii)] There are constants $c_0, C_0>0$ such that, for all $\theta\in\Theta$, the distribution $P_{\theta}$ has a density $p_{\theta}$ with respect to the Lebesgue measure, which satisfies  $c_0 \leq p_\theta(x) \leq C_0$  for all $x\in U$.
  \item[(iii)] The function $(x, \theta) \in \overline{U} \times \Theta  \mapsto (p_\theta(x), \Phi_\theta(x)) $ is twice continuously differentiable.
\end{enumerate}
\end{assumption}

\begin{remark}[on Assumption~\ref{assum:AdditionalRegularity}]\label{rem:Assumps}
    Although Assumption~\ref{assum:AdditionalRegularity} consists of conditions directly on the model $\Pcal$ (rather than assumption on the empirical semi-discrete optimal transport problem itself, which we will impose below), the requisite regularity is too strict even to include the cases of location families and scale families; this is primarily due to the fact that we require all of our measures to have the same bounded support.
    While we envision future work that significantly relaxes this regularity, we note that most works in statistical optimal transport indeed rely on similar conditions \cite{HuetterRigollet, DebGhosalSen,Gunsilius_2022}.
    We are only aware of the recent work \cite{ManoleBalakrishnan} which concludes powerful convergence results without some bounded support condition.
\end{remark}

Specializing the above discussion to the case in which $Q$ is an empirical measure of the form $Q= \frac{1}{n}\sum_{i=1}^n \delta_{x_i}$, we see that the squared distance $W_2^2(P, Q)$ can be rewritten as 
\begin{equation}
W_2^2(P, Q) = \max_{(b_1, \dots, b_n) \in \Rbb^n}  \frac{1}{n}\sum_{i=1}^n b_i  + \int_{\Rbb^d} \min_{j=1, \dots, n} \{ \|x- x_j\|^2 - b_j \} \diff P(x).
\label{eqn:AuxDualitySemi}
\end{equation}
Further, if $P$ satisfies the assumptions of Lemma \ref{lem:AuxDuality}, we can deduce from that lemma the existence of a unique maximizer of \eqref{eqn:AuxDualitySemi} within the set of vectors $b \in \Rbb^n$ satisfying $\sum_{i=1}^n b_i =0$. For the measure $P = P_\theta$ as in Assumption \ref{assum:AdditionalRegularity}, this unique solution is denoted by $b(x_1, \dots, x_n, \theta)$ and in what follows we argue that $b(x_1, \dots, x_n, \theta)$ is a continuously differentiable function of the points $x_1, \dots, x_n$ and the parameter $\theta$.

Toward this aim, observe that, after setting $b_n= - \sum_{i=1}^{n-1}b_i$, the solution $b(x_1, \dots, x_n, \theta)$ can be characterized implicitly via the equation
\begin{equation}
   0 = \mathcal{G}(b_1,\dots, b_{n-1}{}, x_1, \dots, x_n, \theta) :=  \left(\begin{matrix}  1- n \int_{V_1} p_\theta(x) \diff x   \\ \vdots \\ 1- n \int_{V_{n-1}} p_\theta(x) \diff x  \end{matrix} \right) \in \Rbb^{n-1}, 
   \label{eqn:DefG}
\end{equation}
where $V_1, \dots,V_{n-1}$ are the sets (Laguerre cells):
\begin{align*}
    V_i&= V_i( b_1, \dots, b_{n-1}, x_1, \dots, x_n)\\
    &:= \{ x \in \Rbb^d \: \text{s.t.} \: \|x-x_i\|^2 - b_i \leq \|x-x_j\|^2 - b_j, \quad \forall j=1, \dots, n  \};
\end{align*}
see \cite[Chapter 6.4.2]{Santambrogio}. From this characterization, we can deduce the following results.

\begin{proposition}
\label{prop:ContDiffer}
 Suppose that the model $\mathcal{P}$ satisfies Assumption \ref{assum:AdditionalRegularity}. Then  
 \[ b(x_1, \dots, x_n,\theta) := \argmax_{b \in \Rbb^n \: \mathrm{ s.t. \:  } \sum_{i=1}^n b_i =0 } \left\{  \frac{1}{n}\sum_{i=1}^n b_i  + \int_{\Rbb^d} \min_{j=1, \dots, n} \{ \|x- x_j\|^2 - b_j \} \diff P_\theta(x) \right\} \]
 is continuously differentiable in all its arguments.
\end{proposition}
\begin{proof}
Using \cite[Lemma 6.5]{Santambrogio}
and the fact that the support of $P_\theta$ is pathwise connected, the matrix $$ \left[\frac{\partial \mathcal{G}_i}{\partial b_j}\right]_{i,j=1, \dots, n-1}$$ can be verified to be invertible; here, $\mathcal{G}$ is the map defined in \eqref{eqn:DefG}. The implicit function theorem implies that $b(x_1, \dots, x_n, \theta)$ is continuously differentiable. 
\end{proof}

While Assumption~\ref{assum:AdditionalRegularity} concerns the model $\Pcal$ directly, we also need to place two assumptions on some empirical quantities that appear throughout the analysis.
To do this, recall the proof sketch of sensitivity-efficiency in the case $d=1$ (see~\eqref{eqn:WPE_ImplicitFunctionThm} and~\eqref{eqn:WPE_Sensitivity}), which involves differentiating, with respect to the samples $X_1,\ldots, X_n$, the first-order optimality conditions for $T_n^{\WPE}$, and then rearranging the resulting equation.
Our proof in the case $d\ge 1$ is similar, but the expressions for the derivatives are slightly more complicated.
So, before getting to the proof of our main result, we spend some time developing the concepts and notation for these derivatives.
For now, let $x_1,\ldots, x_n$ denote distinct points in $\Rbb^d$, not necessarily random, and let us consider the optimal transport problem from $P_{\theta}$ to $\bar P_n=\frac{1}{n}\sum_{i=1}^{n}\delta_{x_i}$.

For the first assumption, use equation~\eqref{eqn:AuxDualitySemi} or \cite[Chapter 6.4.2]{Santambrogio} to write the squared-distance $W_2^2(P_\theta, \bar{P}_n)$ as
\begin{equation}
   \max_{(b_1, \dots, b_n) \in \Rbb^n}  \left(\frac{1}{n}\sum_{i=1}^n b_i  + \int_{\Rbb^d} \min_{1\le i\le n} (\|x- x_i\|^2 - b_i ) \diff P_\theta(x)\right). 
   \label{eqn:RepresentationW2}
\end{equation}
Under Assumption \ref{assum:AdditionalRegularity}, we may use Proposition~\ref{prop:ContDiffer} to see that there is a unique maximizer for the above problem, among the set of $b \in \Rbb^n$ satisfying $\sum_{i=1}^n b_i =0$. This maximizer will be denoted by $b= b(x_1, \dots, x_n, \theta)$, and Proposition~\ref{prop:ContDiffer} states that $b$ is continuously differentiable in all its arguments.
Also let $\psi_{P_\theta \to \bar{P}_n}:\Rbb^d\to\Rbb$ be the function given by
\[ \psi_{P_\theta \to \bar{P}_n}(x)  := \min_{1\le i\le n} ( \|x- x_i\|^2 - b_i ), \]
and write $\psi_{P_\theta \to \bar{P}_n}^c$ for its $c$-transform, defined as $\psi^c(y) := \inf_{x\in \Rbb^d}(\|x-y\|^2-\psi(x))$ for any $\psi:\Rbb^d\to\Rbb$.
By Lemma \ref{lem:AuxDuality}, and the optimality of $b= b(x_1, \dots, x_n,\theta)$, we have, for all $1\le i\le n$:
\begin{equation}
    \psi^c_{P_\theta \to \bar{P}_n}(x_i) = b_i.
    \label{eqn:CTransformDualPotential}
\end{equation}
Also note that $(\psi_{P_\theta \to \bar{P}_n },\psi^c_{P_\theta \to \bar{P}_n})$ is a pair of optimal Kantorovich potentials for $P_{\theta}$ and $\bar{P}_n$.

For the second assumption, we require the notion of decomposing $U$ into Laguerre cells.
That is, given $b= b(x_1, \dots, x_n, \theta)$, we define the following polyhedral set
\begin{align*}
    &V_i(x_1, \dots, x_n, \theta) \\
    &:= \{ x \in \Rbb^d:\|x-x_i\|^2 - b_i(\theta, x_1, \dots, x_n) \geq \|x-x_j\|^2 - b_j(\theta, x_1, \dots, x_n), \; \forall \; 1\le j\le n \}
\end{align*}
for each $1\le i\le n$.
To simplify the notation in our computations, we often remove the explicit dependence of the Laguerre cells on $\theta$ or $x_1, \dots, x_n$, and simply write $V_i(\theta)$ or $V_i$.
By construction, we have
\[ \psi_{P_\theta \to \bar{P}_n}(x) = \|x-x_i\|^2 - b_i, \qquad \mbox{for all $1\le i\le n$ and $x\in V_{i}$,}\]
 and also
\begin{align}
   \|x- x_i\|^2 - b_i = \|x- x_j\|^2 - b_j, \qquad \mbox{for all $1\le i,j\le n$ and $x \in \partial V_i \cap \partial V_j$,} \label{eqn:AuxWPELaguerre1}
\end{align}
 where $\partial A$ denotes the boundary of a set $A\subseteq\Rbb^d$.
We say that $V_i$ and $V_j$ are neighbors and write $i\sim j$ if we have $\mathcal{H}^{d-1}(\partial V_i \cap \partial V_j) >0 $, where $\mathcal{H}^{d-1}$ denotes the $(d-1)$-dimensional Hausdorff measure. Using the notation $\Gamma_{ij}:= \partial V_i \cap \partial V_j$, we see that, up to a set of zero $\mathcal{H}^{d-1}$-measure, the boundary $\partial V_i$ can be written as $\bigcup_{j \sim i}\Gamma_{ij}$. In addition, \eqref{eqn:AuxWPELaguerre1} implies that, for $i \sim j$, the set $\Gamma_{ij}$ is contained in a hyperplane with unit normal vector $(x_j - x_i)/\|x_j - x_i\|$. That is, for $i\sim j$, 
\begin{equation}
    \Gamma_{ij} \subseteq  \left\{  x \in \Rbb^d : 2 x^{\top}(x_j - x_i) + \|x_i\|^2 -\|x_j\|^2 +b_j-b_i  =0 \right\}.
    \label{eqn:AuxWPEHyperplane}
\end{equation}

We are almost ready to give our main result.
First, note the following in which we compute mixed partial derivatives of the squared distance.

\begin{proposition}
\label{prop:Derivatives}
Under Assumption \ref{assum:AdditionalRegularity} we have
\begin{align}
\begin{split}
 \nabla_{x_i} \frac{\partial}{\partial \theta} W_2^2(P_\theta, \bar{P}_n) =  - 2 \int_{V_i(\theta)} \Phi_\theta(x) p_\theta(x) \diff x  \hspace{1.5in}
 \\+ \sum_{j \sim i} \int_{\Gamma_{ij}} \left(2(\Phi_\theta(x))^{\top} (x_j - x_i) +  \left(  \frac{\partial b_j}{\partial \theta}  - \frac{\partial b_i}{\partial \theta}  \right) \right) \frac{x-x_i} {\|x_i - x_j\|}p_\theta(x) \diff \mathcal{H}^{d-1}(x)
\label{eqn:AuxWPEPartialDerivsDeveloped1}
\end{split}
\end{align}
and
\begin{align}
\begin{split}
\frac{\partial^2}{\partial \theta^2} W_2^2(P_\theta, \bar{P}_n) & =  \int_{\Rbb^d}  \frac{\partial}{\partial \theta}\psi_{P_\theta \rightarrow \bar{P}_n}(x) \frac{\partial}{\partial \theta }p_\theta(x) \diff x + \int_{\Rbb^d} \psi_{P_\theta \rightarrow \bar{P}_n}(x) \frac{\partial^2}{\partial \theta^2 }p_\theta(x)    \diff x.
\label{eqn:AuxWPEPartialDerivsDeveloped2}
\end{split}
\end{align}
\end{proposition}

The proof of this result requires the following well-known result, whose statement in a classical smooth setting can be found in \cite[Theorem 6 in Appendix C.4]{EvansPDE}. 

\begin{lemma}[differentiation formula for moving regions]\label{lem:DiffMovingRegions}
    Let $S\subseteq \Rbb^d$ be a bounded set with a piecewise smooth boundary $\partial S$, and let $\{ S(t) \}_{t \in (-\varepsilon, \varepsilon)}$ be a one-parameter family of sets obtained as $S(t)= \gamma(t, S)$ where $\gamma: (-\varepsilon, \varepsilon) \times  \Rbb^d \mapsto \Rbb^d$ is a differentiable function such that $\gamma(t, \cdot)$ is a bijection for all $t \in (-\varepsilon, \varepsilon)$ and $\gamma(0,x) = x$ for all $x \in \Rbb^d$. If $f=f(t,x)$ is a continuous and piecewise smooth function, then
    \begin{equation}\label{eq:moving-regions}
        \frac{\diff }{\diff t}\int_{S(t)}f(x) \diff x = \int_{S(t)} \frac{\partial}{\partial t} f(t,x)\diff x + \int_{\partial S(t)}f(t,x) v_t(x) \cdot\vec{n}_t(x) \diff\mathcal{H}^{d-1}(x)
    \end{equation}
    where $v_t$ is the velocity of the moving boundary $\partial S(t)$ and $\vec{n}_t$ is the outward pointing unit normal to the boundary of $S(t)$. 
\end{lemma}

\begin{proof}[Proof of Proposition~\ref{prop:Derivatives}]
First we show~\eqref{eqn:AuxWPEPartialDerivsDeveloped1}.
To do this, observe that 
\begin{align*}
   \nabla_{x_i} W_2^2(P_\theta, \bar{P}_n) & = \nabla_{x_i} \left( \sum_{j=1}^n \int_{V_j(\theta)} \|x_j - x\|^2 p_\theta(x) \diff x  \right) 
   \\& = \nabla_{x_i} \left( \sum_{j=1}^n \int_{V_j(\theta)} ( \|x_j - x\|^2 - b_j) p_\theta(x) \diff x  \right),  
\end{align*}
thanks to the fact that $\sum_{j=1}^n b_j=0$ and that $\int_{V_j(\theta)} p_\theta(x) \diff x = \frac{1}{n}$.
From the above and Lemma \ref{lem:DiffMovingRegions}, we see that
\begin{align*}
   \nabla_{x_i} \left( \sum_{j=1}^n \int_{V_j(\theta)} ( \|x_j - x\|^2 - b_j) p_\theta(x) \diff x  \right)  & = \sum_{j=1}^n \int_{V_j(\theta)}\nabla_{x_i}(\|x-x_j\|^2 - b_j) p_\theta(x) \diff x 
   \\& = \int_{V_i(\theta)} 2(x_i- x) p_\theta(x) \diff x,
\end{align*}
where we observe that the boundary terms coming from the differentiation formula for moving regions cancel out in pairs, given that for $i \sim j$ we have \eqref{eqn:AuxWPELaguerre1} and the outer normal of $V_i$ along $\Gamma_{ij}$ is equal to the negative of the outer normal of $V_j$ along $\Gamma_{ij}$.  Therefore,
\begin{equation}
   \nabla_{x_i} W_2^2(P_\theta, \bar{P}_n) = \int_{V_i(\theta)} 2(x_i -x) p_\theta(x) \diff x  . 
   \label{eqn:AuxWPEDiffxi}
\end{equation}
From \eqref{eqn:AuxWPEDiffxi} we thus obtain
\begin{align}
\begin{split}
&\frac{\partial}{\partial \theta}  \nabla_{x_i} W_2^2(P_\theta, \bar{P}_n)
  \\
  &= \frac{\partial }{\partial \theta} \int_{V_i(\theta)} 2(x_i - x) p_\theta(x) \diff x 
  \\ &=  \int_{V_i(\theta)} 2(x_i - x) \frac{\partial  p_\theta }{\partial \theta} \diff x \, + \, 2\int_{\partial V_i(\theta) } (x_i-x)(u_\theta(x))^{\top}\vec{n}(x) p_\theta(x) \diff  \mathcal{H}^{d-1}(x),
  \end{split}
 \label{eqn:AuxWPE1}
\end{align}
where in the last line we used again the differentiation formula for moving regions; in the above, $u_\theta$ is the velocity (at $\theta$) of the boundary $\partial V_i(\theta)$ as $\theta$ changes, and $\vec{n}$ represents the outer unit normal to $\partial V_i(\theta)$. In order to develop the right hand side of \eqref{eqn:AuxWPE1}, we first claim that
\begin{equation}(u_\theta(x))^{\top}\vec{n}(x) = -\frac{1}{2} \frac{\frac{\partial b_j}{\partial \theta} - \frac{\partial b_i}{\partial \theta}}{\|x_i- x_j\|}
\label{eqn:AuxWPEFlowMovingRegion}
\end{equation}
for all $x \in \Gamma_{ij}$.
To see this, consider a point $x(\theta) \in \Gamma_{ij}(\theta)$ traveling as the boundary $\Gamma_{ij}$ changes with $\theta$. By the definition of $\Gamma_{ij}(\theta)$ in equation \eqref{eqn:AuxWPEHyperplane}, we must have
\[ 0 =2  (x(\theta))^{\top}  (x_j -x_i)  +\|x_i\|^2 - \|x_j\|^2 +b_j-b_i.  \]
Differentiating both sides of the above equation with respect to $\theta$ and using the fact that $\vec{n}(x) = (x_j - x_i)/\|x_j - x_i\|$ for all $x \in \Gamma_{ij}$, we obtain \eqref{eqn:AuxWPEFlowMovingRegion}. With identity \eqref{eqn:AuxWPEFlowMovingRegion} in hand, we can rewrite \eqref{eqn:AuxWPE1} as
\begin{equation*}
 \frac{\partial}{\partial \theta}   \nabla_{x_i} W_2^2(P_\theta, \bar{P}_n)  = 2\int_{V_i(\theta)} (x_i - x) \frac{\partial p_\theta }{\partial \theta}  \diff x \, 
    -  \,  \sum_{j \sim i}\int_{\Gamma_{ij} }(x_i-x) \frac{\frac{\partial b_j}{\partial \theta} - \frac{\partial b_i}{\partial \theta}}{\|x_i- x_j\|}   p_\theta(x)\diff  \mathcal{H}^{d-1}(x).
\end{equation*}
On the other hand, the continuity equation in strong form (which holds under Assumption \ref{assum:AdditionalRegularity}) and the divergence theorem allow us to rewrite the first term on the right hand side of the above expression as 
\begin{align*}
  &2\int_{V_i(\theta)} (x_i - x) \frac{\partial p_\theta}{\partial \theta} \diff x\\
  &= - 2\int_{V_i(\theta)} (x_i - x)  \mathrm{div}(\Phi_\theta p_\theta)\diff x 
   \\ & =    - 2 \int_{V_i(\theta)} \Phi_\theta(x) p_\theta(x) \diff x  + 2 \int_{\partial V_i(\theta)} (x-x_i) (\Phi_\theta(x))^{\top} \vec{n}(x) p_\theta(x) \diff \mathcal{H}^{d-1}(x)
    \\ & =    - 2 \int_{V_i(\theta)} \Phi_\theta(x) p_\theta(x) \diff x  + 2 \sum_{j \sim i} \int_{\Gamma_{ij}} (x-x_i) (\Phi_\theta(x))^{\top} \frac{x_j-x_i}{\|x_j- x_i\|} p_\theta(x) \diff \mathcal{H}^{d-1}(x).
\end{align*}
Combining the two previous displays, we deduce that $\frac{\partial}{\partial \theta} \nabla_{x_i}  W_2^2(P_\theta, \bar{P}_n) $ is equal to the right hand side of \eqref{eqn:AuxWPEPartialDerivsDeveloped1} and in particular is continuous in all its arguments. Similarly, a direct computation reveals that  
$\nabla_{x_i} \frac{\partial}{\partial \theta} W_2^2(P_\theta, \bar{P}_n)$ is continuous in all its arguments (although the expression obtained from a direct computation is less interpretable), so Clairaut's theorem implies $\nabla_{x_i} \frac{\partial}{\partial \theta}  W_2^2(P_\theta, \bar{P}_n) = \frac{\partial}{\partial \theta} \nabla_{x_i}  W_2^2(P_\theta, \bar{P}_n)$, whence \eqref{eqn:AuxWPEPartialDerivsDeveloped1}.

    Second, we show equation~\eqref{eqn:AuxWPEPartialDerivsDeveloped2}.
    Towards this end, observe that
    \begin{equation*}\label{eqn:D-Representation_new}
        \frac{\partial}{\partial \theta} W_2^2(P_\theta,\bar{P}_n) =    \int_{\Rbb^d} \psi_{P_\theta \to \bar{P}_n }(x) \frac{\partial p_{\theta}}{\partial \theta}(x)\diff x,
    \end{equation*}
    which is derived similarly to the formula for $\nabla_{x_i} W_2^2(P_\theta, \bar{P}_n)$ at the beginning of the proof the previous part. Differentiating the above expression with the aid of the differentiation formula for moving regions we deduce \eqref{eqn:AuxWPEPartialDerivsDeveloped2}.
\end{proof}

Finally, we can state the additional assumptions required in our result on asymptotic sensitivity-efficiency for the multidimensional WPE.
We will give several remarks about these assumptions in the rest of the appendix.

\begin{assumption}\label{assum:AdditionalAssumps}
    Under Assumption~\ref{assum:AdditionalRegularity}, suppose that, if $X_1,X_2,\ldots$ are i.i.d. samples from $P_{\theta^{\ast}}$ for $\theta^{\ast}\in\Theta$, then we have
    \begin{itemize}
        \item[(i)] $\sup_{\theta\in\Theta}\|\nabla_{x}\frac{\partial}{\partial \theta} \psi^{c}_{P_{\theta}\to\bar P_n}-\nabla_{x}\frac{\partial}{\partial \theta} \psi^{c}_{P_{\theta}\to P_{\theta^{\ast}}}\|_{L^{\infty}(U)} = o_{P_{\theta^{\ast}}}(1)$, and
        \item[(ii)] $\frac{1}{n} \sum_{i=1}^n  \left(n \int_{\partial V_i} \|x-X_i\| p_{\hat{\theta}_n}(x) \diff \mathcal{H}^{d-1}(x)  \right)^2 = O_{P_{\theta^{\ast}}}(1)$.
    \end{itemize}
\end{assumption}

\begin{remark}[on assumption $(i)$ of Assumption~\ref{assum:AdditionalAssumps}]\label{rem:uniform-cvg-pot}
It may appear that assumption $(i)$ is rather strict (since it requires uniform convergence of some higher-order derivatives of the optimal dual potentials), but we argue through some basic examples (location family and scale family) that it is not unreasonable; however, recall from Remark~\ref{rem:Assumps} that both of these examples are excluded from the theorem because of Assumption~\ref{assum:AdditionalRegularity}.
As such, we believe one could verify this condition in some more difficult models of interest, albeit with significantly more work; see \cite{GonzalezSanzSheng} for similar results under additional regularity that need not hold in our applications.

First, we consider the case of a location family, as in Example~\ref{ex:loc-fam}, where we claim that we have
   \[ \nabla_x \frac{\partial}{\partial \theta} \psi^c_{P_\theta \to \bar{P}_n}(x) = \nabla_x \frac{\partial}{\partial \theta} \psi^c_{P_\theta \to {P}_{\theta^*}}(x) = -2\idmat_d   \]
   To see this, note that the optimal transport map between $P_\theta$ and an arbitrary $Q \in \mathcal{P}_2(\Rbb^d)$ can be written in terms of the transport map between $P_0$ and $Q$ according to
   \[ \transport_{P_\theta \to Q}(x) =  \transport_{P_0 \to Q} (x - \theta),\]
   thanks to Brenier's theorem. From this identity and \eqref{aux:GradientDualPotential} above, we have the following identity for the potentials:
   \[ \psi_{P_\theta \to Q}(x) = \psi_{P_0 \to Q}(x-\theta) + 2\theta \cdot (x-\theta) + C(\theta), \]
   where $C(\cdot)$ is a function of $\theta$. This implies
   \[  \psi^c_{P_\theta \to Q} (x) =  \psi^c_{P_0 \to Q}(x) -2 \theta \cdot x + \tilde{C}(\theta), \]
    so differentiating with respect to $\theta$ and then with respect to $x$ yields
\[ \nabla_x \nabla_\theta \psi^c_{P_\theta \to Q}(x) = -2\idmat_d,  \]
regardless of $Q$.
Taking $Q=P_{\theta^{\ast}}$ and $Q=\bar P_n$ gives the result.

Second, we consider the case of a scale family, as in Example \ref{ex:var-fam}.
We can write the optimal transport map between $P_\theta$ and an arbitrary $Q \in \mathcal{P}_2(\Rbb^d)$ in terms of the transport map between $P_0$ and $Q$ according to
   \[ \transport_{P_\theta \to Q}(x) =  \transport_{P_1 \to Q} \left(\frac{x}{\theta}\right),\]
thanks to Brenier's theorem. This leads to
\[   \psi^c_{P_\theta \to Q}(x) =     (1-\theta)\lVert x \rVert^2 + \theta \psi^c_{P_1 \to Q}(x) + C(\theta),     \]
which in turn implies 
\[  \nabla_x \frac{\partial}{\partial \theta} \psi^c_{P_\theta \to Q}(x) = \nabla_x \psi^c_{P_1 \rightarrow Q}(x) -2x.  \]
In particular, we have shown
\[   \nabla_x \frac{\partial}{\partial \theta} \psi^c_{P_\theta \to \bar{P}_n}(x) -  \nabla_x \frac{\partial}{\partial \theta} \psi^c_{P_\theta \to P_{\theta^*}}(x) =   \nabla_x  \psi^c_{P_1 \to \bar{P}_{n}}(x)  -\nabla_x  \psi^c_{P_1 \to P_{\theta^*}}(x). \]
We see from the above that $(i)$ in Theorem \ref{thm:WPE-efficient-gen} reduces to the uniform consistency for the gradients of the potentials for the single model $P_1$.
\end{remark}
 
\begin{remark}[on assumption $(ii)$ of Assumption~\ref{assum:AdditionalAssumps}]
    In order to argue that this condition is not so unreasonable, we give a heuristic calculation to explain why it may be expected to hold in some models of interest. 
    First, recall that all the densities $p_\theta$ are uniformly bounded away from zero, hence $P_{\theta}$-volumes are uniformly comparable to volumes under the Lebesgue measure.
    Thus, since each Laguerre cell $V_i$ has mass $n^{-1}$ under $P_{\hat{\theta}_n}$ by definition, it follows that the Lebesgue measure of $V_i$ is $n^{-1}$ up to constants.
    If $V_i$ is uniformly elliptic (that is, if we can inscribe a ball whose volume is a constant order multiple of the volume of the cell), then its perimeter has size $n^{-(d-1)/d}$ and the typical distance between a point $x \in \partial V_i$ and $x_i$ is $n^{-\frac{1}{d}}$, both up to constants; thus, the collection of random variables appearing in assumption $(ii)$ would indeed be tight.
    However, making this intuition precise requires a fine analysis of the geometry of the random Laguerre diagram induced by the data, and specifically requires estimates on the surface area of the $(d-1)$-dimensional faces of the Laguerre cells. Some works that have analyzed related geometric features of random Laguerre diagrams (albeit in the Poisson point process setting) include \cite{Laguerre1} and the recent book chapter \cite{Redenbach2025}.
\end{remark}

\end{appendix}

\end{document}